\documentclass[12pt]{amsart}
\usepackage{mycustompackage}

\title{The Deligne--Riemann--Roch isomorphism}
\author{Dennis Eriksson}
\author{Gerard Freixas i Montplet}

\address{Dennis Eriksson \\ Department of Mathematics \\ Chalmers University of Technology and  University of Gothenburg}
\email{dener@chalmers.se}

\address{Gerard Freixas i Montplet \\ Centre de Math\'ematiques Laurent Schwartz -- CNRS -- \'Ecole Polytechnique -- Institut Polytechnique de Paris}
\email{gerard.freixas@polytechnique.edu}

\makeatletter
\@namedef{subjclassname@2020}{\textup{2020} Mathematics Subject Classification}
\makeatother



\thanks{The first author was supported by the Swedish Research Council, VR grant 2021-03838 "Mirror symmetry in genus one". The second author was supported by the Knut och Alice Wallenberg foundation "Guest researcher program" and the French ANR-24-CE40-6184 (AdAnAr).}

\begin{document}
\setcounter{tocdepth}{2}

\begin{abstract}

This work establishes the geometric component of Deligne’s longstanding program on refined Grothendieck--Riemann--Roch formulas expressed through determinants of cohomology. The approach relies on a newly developed universal category of Chern classes together with an associated relative intersection theory. As an example of the applications, we provide a structural description of the coefficients in the Knudsen--Mumford expansion and establish a fundamental Mumford-type isomorphism for the alternating product of Griffiths bundles. 

\end{abstract}
\maketitle

\tableofcontents

\section{Introduction}
\begingroup
\setcounter{tmp}{\value{theorem}}
\setcounter{theorem}{0}
\renewcommand\thetheorem{\Alph{theorem}}

\subsection{Origin of the Deligne program}

 In his seminal article \emph{Le déterminant de la cohomologie}~\cite{Deligne-determinant}, based on a 1985 letter to Quillen, Deligne proposed a conjectural program in which he interprets the degree-one part of the Grothendieck--Riemann--Roch formula as a canonical isomorphism of line bundles. One of the earliest instances of such isomorphism goes back to Mumford's work on families of smooth projective curves $C \to S$ and their moduli spaces \cite{Mumford:stability}. Mumford showed that there exists a canonical isomorphism between determinant line bundles of pluricanonical forms
\begin{equation}\label{eq:Mumfordiso}
    \det Rf_\ast (\omega_{C/S}^{n})\simeq (\det Rf_\ast \omega_{C/S})^{6n^{2}-6n+1}.
  \end{equation}
This result follows from an application of the Grothendieck--Riemann--Roch theorem on moduli spaces of curves and relies on subtle algebraic and topological features of these spaces. Comparable isomorphisms, such as Moret-Bailly’s \emph{formule clé} \cite{FalChai, Moret-Bailly:pinceaux}, arise through similar techniques. This circle of ideas already appears in Faltings’ pioneering work on arithmetic intersections on arithmetic surfaces \cite{Faltings84}. As further related applications, explicit realizations of the Mumford isomorphism lead to the so-called Mumford forms and Polyakov measures on moduli spaces of curves \cite{Beilinson-Manin, Fay}. Moret-Bailly’s \emph{formule clé} also plays a central role in the study of heights of abelian varieties \cite{Bost:intrinsic-heights, deJong-Shokrieh}.
 
\subsection{The program and the example of curves}
The examples above remain limited to special geometric settings where moduli-space methods are available. Deligne's program seeks to bypass such arguments, aiming instead for a formulation valid for general families. His considerations begin by viewing the degree-one part of the Grothendieck--Riemann--Roch formula as an expression for the first Chern class of the Knudsen--Mumford determinant of cohomology,
\begin{equation}\label{eq:determinantofthecohomology}
\lambda_f(E) = \det Rf_\ast E,
\end{equation}
associated with a suitable projective morphism $f\colon X \to S$ of schemes and a vector bundle $E$ on $X$, through the equality
\begin{equation}\label{eq:GRR-equality1}
    c_1(\lambda_f(E)) = f_\ast\bigl(\ch(E)\cdot \td(T_f)\bigr)^{(1)}
\end{equation}
in $\Pic(S)\otimes \QBbb$. This is valid whenever $S$ admits an ample line bundle, and we suppose for the purpose of the introduction that this is the case. It should be \emph{canonical}, not merely determined up to an invertible regular function on $S$, and it should moreover be compatible with key features of the determinant of the cohomology, such as base change functoriality and multiplicativity over exact sequences. 

Deligne realized the first step of his program for families of smooth curves by explicitly constructing the right-hand side in that case and using a version of \eqref{eq:Mumfordiso}, obtaining a canonical isomorphism:
\begin{equation}\label{eq:DeligneIsoCurvesintro}
   \lambda_f(L)^{12} \simeq \langle \omega_{X/S}, \omega_{X/S} \rangle \otimes \langle L, L \otimes \omega_{X/S}^{-1} \rangle^{-6},
\end{equation}
where we have simplified by taking $E = L$ to be a line bundle, and where $\omega_{X/S}$ denotes the relative cotangent bundle. Taking first Chern classes recovers \eqref{eq:GRR-equality1} for a family of curves: the line bundles of the form $\langle L, M \rangle$, known as Deligne pairings, satisfy:
\begin{equation}\label{eq:c1(L,M)}  
    c_1(\langle L, M \rangle) = f_\ast(c_1(L)\cdot c_1(M)).
\end{equation}
The isomorphism \eqref{eq:DeligneIsoCurvesintro} is referred to as the Deligne--Riemann--Roch isomorphism for curves. Henceforth, we abbreviate Deligne--Riemann--Roch as DRR. This result has found applications in the study of the singularities of Quillen metrics for degenerating families of compact Riemann surfaces in \cite{ErikssonQuillen}, the study of discriminants of plane curves \cite{Eriksson:discriminants}, and extensions of the Riemann--Roch theorem in Arakelov geometry \cite{Freixas:ARR, Freixas:AHS}. 

\subsection{Interesection bundles}
To carry out Deligne's program in higher relative dimensions, one needs a conceptual generalization of Deligne’s construction of the right-hand side of \eqref{eq:DeligneIsoCurvesintro}. Deligne already sketched such an extension in \cite{Deligne-determinant}, and Elkik developed it in \cite{Elkikfib} into a formalism of intersection bundles in higher dimension. This yields a canonical $\QBbb$-line bundle representing the right-hand side of \eqref{eq:GRR-equality1},
\begin{displaymath}
    \langle \chfrak(E)\cdot \tdfrak(T_f)\rangle_{X/S},
\end{displaymath}
and hence, by the classical Grothendieck--Riemann--Roch, an isomorphism
\begin{equation}\label{eq:GRR-intro}
    \lambda_f(E)\xrightarrow{\sim} \langle \chfrak(E)\cdot \tdfrak(T_f)\rangle_{X/S},
\end{equation}
unique up to invertible regular functions on $S$.

A natural idea to obtain a distinguished DRR isomorphism of the form \eqref{eq:GRR-intro} is to mimic the classical proof based on deformation to the normal cone, within this formalism. However, Elkik's framework only accounts for Deligne pairings and their properties akin to \eqref{eq:c1(L,M)}, and not the intermediate operations such as $f_\ast$ and $c_1(L)$. Dealing with closed immersions poses an additional obstruction, since direct images of characteristic classes in this case  cannot, in general, be represented by line bundles. Moreover, replacing equalities by isomorphisms makes the commutativity of several diagrams subtle and genuinely non-trivial. These limitations indicate that a more flexible framework is needed.

\subsection{Line distributions}
To bypass the limitations of the theory of intersection bundles, we introduced in \cite{DRR1} a substantial extension taking the form of a genuine relative intersection theory with values in line bundles. We briefly outline the main constructions, without entering into technical details.

The central objects of \cite{DRR1} are the \emph{intersection distributions}
\begin{displaymath}
    [P]_{X/S},
\end{displaymath}
where $X\to S$ is a flat projective morphism of relative dimension $n$, and $P$ is a rational power series in the Chern classes on $X$. These are certain functors which associate, to another such series $Q$, the intersection bundle
\begin{equation}\label{eq:introintersectionbundle-short}
    [P]_{X/S}(Q)=\langle (Q\cdot P)^{(n+1)}\rangle_{X/S}.
\end{equation}
Here $Q$ is viewed as an object of the universal \emph{Chern category} $\CHfrak(X)_{\QBbb}$, a graded ring category encoding the formal properties of Chern classes. Intersection distributions thus act functorially on characteristic-class data to produce line bundles on $S$. 

Most classical identities of intersection theory continue to hold at this level, now interpreted through canonical isomorphisms: Chern classes of divisors, vanishing beyond rank, Whitney sums, and their mutual compatibilities. The role of $K$-theory is played by the \emph{virtual category}, viewed as a truncation of the  $\infty$-groupoid associated with the $K$-theory spectrum, and the construction of the Chern category is in fact reminiscent of this perspective. In particular, Thomason’s machinery in \cite{ThomasonTrobaugh} allows one to extend the constructions to perfect complexes, such as derived direct images under perfect morphisms. This additional flexibility is essential for formulating a DRR-type statement.

Intersection distributions are the basic examples of the more general \emph{line distributions}, that is, functors
\begin{displaymath}
    \CHfrak(X)_{\QBbb}\longrightarrow \Picfr(S)_{\QBbb},
\end{displaymath}
from the Chern category of $X$ to the Picard category of $\QBbb$-line bundles on $S$. While direct images do not exist in the Chern category itself, they can be defined for line distributions by duality from pullback of Chern classes: if $\pi\colon X'\to X$ is any morphism with $X'\to S$ flat and projective of relative dimension $n'$, then
\begin{displaymath}
    \pi_\ast [P]_{X'/S}(Q)=\langle (\pi^\ast Q\cdot P)^{(n'+1)}\rangle_{X'/S}.
\end{displaymath}
Thus, enlarging intersection bundles to line distributions makes direct images under closed immersions available, as required in arguments such as deformation to the normal cone.

Intersection distributions also satisfy splitting principles, established in \cite{Eriksson-Freixas-Wentworth} and applied systematically in \cite{DRR1}. As a key application, we proved in \emph{op. cit.} a DRR isomorphism for the closed immersion defined by the zero locus of a regular section of a vector bundle, a result that couldn't be formulated in the original framework.

\subsection{The main theorem}

Within our new framework, one can envision Grothendieck’s strategy for Riemann--Roch, namely the factorization of a projective morphism into a regular closed immersion followed by a projective bundle. Our convention is that a morphism $X \to S$ is \emph{projective} if it factors through a closed immersion $X \to \PBbb(\Ecal)$, for some vector bundle $\Ecal$ of constant rank on $S$, and \emph{locally projective} if this condition holds after passing to an open cover of $S$. This is the overall approach that we adopt. Several delicate points arise, whose resolution blends geometric constructions with fundamental features of the $K$-theory spectrum. Moreover, the arguments involving deformation to the normal cone require a careful analysis of birational invariance within this setting. The treatment of projective bundles likewise relies on parallel considerations at the level of the $K$-theory spectrum.

As remarked, there are many isomorphisms of the form \eqref{eq:GRR-intro}, determined only up to multiplication by an invertible regular function on $S$. To remove this ambiguity and obtain a canonical choice, one must require a minimal set of structural properties. Concretely, it is natural to impose that any such natural isomorphism satisfies the following (non-exhaustive) list of conditions:

\begin{enumerate}
    \item\textsc{Functoriality in the morphism:} The construction is compatible with isomorphisms of $S$-schemes.
    \item \textsc{Functoriality in the base}: The construction is compatible with arbitrary base changes $S' \to S$, with $S^{\prime}$ not necessarily quasi-compact.
    \item\label{item:iso-DRR-intro-3} \textsc{Functoriality and additivity in $E$}: The construction is compatible with isomorphisms in $E$, it is additive with respect to short exact sequences and extends naturally to perfect complexes.
    \item \textsc{Projection formula:} Both sides behave naturally under the projection formula, and the isomorphism respects this compatibility.
   \item \textsc{Grothendieck duality:} Both sides transform naturally under Grothendieck duality, and the isomorphism is compatible with this in the Noetherian setting.
\end{enumerate}

 With this view in mind, in this paper we obtain the following conclusion of Deligne's program (cf. Theorem \ref{thm:DRR-det-coh} and Theorem \ref{thm:Groth-duality}).
\begin{theorem}[\textsc{Deligne--Riemann--Roch isomorphism}]\label{thm:A}
Let $f\colon X \to S $ be a flat surjective locally projective morphism of quasi-compact schemes, of local complete intersection and relative dimension $n\geq 0$. Let $E$ be a vector bundle on $X$. Then, there exists a canonical isomorphism of $\QBbb$-line bundles
\begin{equation}\label{eq:iso-DRR-intro}
    \lambda_{f}(E) \xrightarrow{\sim} \langle \chfrak(E) \cdot \tdfrak(T_f) \rangle_{X/S},
\end{equation}
satisfying the above list of properties. The isomorphism is characterized by further requiring a compatibility with closed immersions over $S$.
\end{theorem} 


 
We stress that the theorem is valid even without $S$ admitting an ample line bundle, since the functorial setting allows for gluings of local constructions, which are not possible in Chow groups. By \cite[Proposition 7.1]{DRR1}, in the case when $S$ does admit an ample line bundle, Theorem \ref{thm:A} recovers the formula \eqref{eq:GRR-equality1} upon taking the first Chern class. In general, it is not known if Elkik's intersection bundles represent direct images of the type \eqref{eq:c1(L,M)}. Given a positive answer to this question, \eqref{eq:GRR-equality1} would hold under more general assumptions on the base scheme $S$ than those of the original Grothendieck--Riemann--Roch theorem.

We note that there exists a more general, genuinely relative form of the Deligne--Riemann--Roch isomorphism (cf. Theorem \ref{thm:general-DRR}). This formulation is elaborated in Section \ref{section:DRR-isos}. It is expressed at the level of line distributions and it has the advantage of exhibiting a natural compatibility with the composition of morphisms, a key feature of our construction and the raison d'\^etre of the theory of line distributions. The version stated in Theorem \ref{thm:A} is a direct consequence of it. This generalized DRR isomorphism also plays a central role in the characterization of the isomorphism \eqref{eq:iso-DRR-intro} and in establishing compatibility with Grothendieck duality.

\subsection{Some geometric consequences}
Deligne’s work on the determinant of the cohomology \cite{Deligne-determinant} has inspired extensive work, both building on his original ideas and developing alternative frameworks. Below we recall some of these and provide applications of our theorems to these contexts. We also present new results, such as applications to the determinant of the de Rham cohomology, Griffiths bundles  and BCOV bundles of Calabi--Yau families. For proofs and further details, we refer to Section \ref{sec:compatibilities}.

\subsubsection{Examples in low dimensions}\label{subsubsec:low-dim}

We first show that our isomorphism coincides with Deligne's for families of smooth curves, cf. Proposition \ref{prop:comparison-with-Deligne}.
\begin{proposition}
For families of smooth projective curves, the isomorphism in Theorem \ref{thm:A} coincides with the original Deligne--Riemann--Roch isomorphism, as isomorphisms of $\QBbb$-line bundles. 
\end{proposition}

The proof exploits the splitting principles from \cite{DRR1, Eriksson-Freixas-Wentworth} to reduce to the line bundle case, and then proceeds by comparing our isomorphism to Deligne's original one \eqref{eq:DeligneIsoCurvesintro} on a Picard stack over a moduli space of curves over $\ZBbb$, on which the only invertible functions are $\pm 1$. The argument thus hinges in a fundamental manner on the additivity on short exact sequences and the base-change functoriality. 

For families of projective surfaces, the Deligne--Riemann--Roch isomorphism for the trivial rank one bundle readily takes a particularly simple form. It may be viewed as a two-dimensional analogue of Mumford's isomorphism and a relative version of Noether's formula for the determinant of the cohomology:

\begin{proposition}
Let $f\colon X\to S$ be a smooth projective family of surfaces with relative canonical bundle $K_{X/S}$. Then, there is a canonical isomorphism of $\QBbb$-line bundles
\begin{displaymath}
    \lambda_{f}(\Ocal_{X})^{24}\simeq \langle\cfrak_{1}(K_{X/S})\cdot\cfrak_{2}(\Omega_{X/S})\rangle_{X/S}^{-1},
\end{displaymath}
compatible with base change.
\end{proposition}

\subsubsection{The Knudsen--Mumford expansion}
Let $f: X\to S$ be a morphism of quasi-compact schemes as in Theorem \ref{thm:A} and let $L$ be a line bundle on $X$. By the work of Knudsen and Mumford \cite{KnudsenMumford} there exists a canonical isomorphism
\begin{displaymath}
    \lambda_{f}(L^k) \simeq \bigotimes_{\ell=0}^{n+1} \mathcal{M}_\ell^{k\choose\ell} \simeq \mathcal{M}_{n+1}^{k \choose n+1} \otimes \mathcal{M}_n^{k \choose n} \otimes \ldots ,
\end{displaymath}
for some natural line bundles $\mathcal{M}_\ell, \ell = 0, \ldots, n+1$, on $S$. The dominant and subdominant terms of this expansion have been described in various levels of generality. For instance, for general $f$ as above, one has an expression in terms of Deligne pairings (cf. \cite[Theorem 1.4]{Zhangheights})
\begin{equation}\label{eq:KM-dominant-intro}
    \Mcal_{n+1} \simeq \langle L, \ldots, L \rangle_{X/S}.
\end{equation}
If one restricts to smooth quasi-projective varieties and relatively very ample line bundles, then \cite[Theorem 1]{PhongSturmRoss} provides the subdominant term
\begin{equation}\label{eq:KM-subdominant-intro}
    \Mcal_{n}^{2}\simeq \langle L^{n} K_{X/S}^{-1},L,\ldots, L\rangle_{X/S}.
\end{equation}

At the expense of working with isomorphisms of $\QBbb$-line bundles and lci morphisms, our work allows one to theoretically determine all the terms of the Knudsen--Mumford expansion for any line bundle $L$. Beyond the expressions \eqref{eq:KM-dominant-intro} and \eqref{eq:KM-subdominant-intro}, we can make the next term explicit: there is a canonical isomorphism of $\QBbb$-line bundles
\begin{equation}\label{eq:subsubKM}
    \Mcal_{n-1}^{12}\simeq \Mcal_{n+1}^{-12}\otimes\Mcal_{n}^{6}\otimes\langle K_{X/S}, K_{X/S},L,\ldots, L\rangle_{X/S}
    \otimes \langle\cfrak_{2}(T_{f})\cdot\cfrak_{1}(L)^{n-1}\rangle_{X/S},
\end{equation}
which commutes with base change (cf. Corollary \ref{cor:Knudsen-Mumford-subsubdominant}). Since by the work of Ducrot \cite{Ducrot} the Deligne pairings can be naturally expressed in terms of the determinant of the cohomology (cf. \textsection \ref{subsubsec:Ducrot} below), the isomorphism \eqref{eq:subsubKM} can also be recast as an interpretation of the intersection bundle $\langle\cfrak_{2}(T_{f})\cdot\cfrak_{1}(L)^{n-1}\rangle_{X/S}$ in such terms. 

\subsubsection{The determinant of the de Rham cohomology}
 Let $U$ be a smooth quasi-projective variety over $\CBbb$ of dimension $n$, that we tacitly identify with the associated complex manifold. Suppose that $\mathbb{V}$ is a local system on $U$, of rank $r$, with associated flat vector bundle $(\Vcal,\nabla)$. If $X$ is a smooth projective compactification of $U$, with simple normal crossings boundary divisor $D$, there exists a natural extension of $\Vcal$ to a vector bundle $E$ on $X$, in such a way that $\nabla$ extends to a flat logarithmic connection $\nabla\colon E\to E\otimes\Omega_{X/\CBbb}(\log D)$ and the associated de Rham complex $\Ecal^{\bullet}$ satisfies
\begin{displaymath}
    R\Gamma(X,\Ecal^{\bullet})\simeq R\Gamma(U,\mathbb{V}).
\end{displaymath}
 In \textsection \ref{subsec:det-dR-coh}, we explain how to obtain the following result:
\begin{proposition}\label{prop:det-dR-intro}
There exists a canonical isomorphism of complex lines
\begin{equation}\label{eq:DRR-det-dR-intro}
    \begin{split}   
         &\det R\Gamma(U,\mathbb{V})\otimes \det R\Gamma(U,\underline{\CBbb}_{U})^{-r}\simeq \langle\cfrak_{1}(\det E)\cdot\cfrak_{n}(\Omega_{X/\CBbb}(\log D))\rangle_{X/\CBbb}^{(-1)^{n}},
    \end{split}
\end{equation}
well defined up to a power, which is multiplicative on exact sequences. An analogous statement holds for the compactly supported cohomology of $\VBbb$, which is compatible with \eqref{eq:DRR-det-dR-intro} via Poincar\'e duality.
\end{proposition}

The proof relies on simplifications on the right-hand side of the Deligne--Riemann--Roch isomorphism applied to the de Rham complex. In the case of curves, the isomorphism is already a consequence of the original Deligne--Riemann--Roch isomorphism, and it played a key role in the work of T. Saito and Terasoma on periods of connections \cite{Saito-Terasoma}. In \cite[Remark, Section 6 (b)]{Saito-Terasoma}, they posed the problem of extending such an isomorphism from curves to higher dimensions, which our result addresses.

It would be interesting to consider variants of Proposition \ref{prop:det-dR-intro} for more general $\Dcal$-modules, $p$-adic cohomologies or Higgs bundles. It would also be interesting to elucidate the relationship with the theory of $\varepsilon$-factors, cf. \cite{Patel}. Finally, along the lines of \cite{Eriksson-Freixas-Wentworth} and \cite{Saito-Terasoma}, one may wonder about upgrading both sides of \eqref{eq:DRR-det-dR-intro} with natural connections and studying the compatibility of these with the isomorphism. We note the central role played by the functorial approach in these works. 

\subsubsection{Griffiths bundles and the BCOV  isomorphism}\label{subsubsec:BCOV-iso}
The previous applications focused on simplifications on the right-hand side of the Deligne--Riemann--Roch isomorphism for the de Rham complex. Similar simplifications arise more generally for weighted versions of the de Rham complex. For example, Hirzebruch used this in \cite{Hirzebruch:Gritsenko} to derive congruences for the Euler characteristics of Calabi--Yau varieties. One can consider an analogous phenomenon at the level of the determinant of the  cohomology.

Let $f\colon X\to S$ be a smooth projective morphism of Noetherian schemes, of relative dimension $n$. Suppose for simplicity that the relative Hodge bundles $\Hcal^{q}(\Omega_{X/S}^p):= R^q f_\ast (\Omega_{X/S}^p)$ are locally free, for example in the complex setting and for $S$ reduced. Define the Griffiths bundle of the relative de Rham cohomology by 
\begin{displaymath}
 \Gr(\Hcal^k) = \bigotimes_{p+q=k} \det \Hcal^{q}(\Omega_{X/S}^p)^{p}.
\end{displaymath}
In the complex setting, as the Hodge--de Rham spectral sequence degenerates, this coincides with the usual definition in \cite{Griffiths-periods-3} of the Griffiths bundle of a variation of Hodge structures as a product of determinants of Hodge filtrations $\bigotimes_{p \leq k} \det \left(\Fcal^p/\Fcal^{p+1}\right)^{p}=\bigotimes_{0 < p \leq k} \det \Fcal^p$. Then we can prove the following Mumford--type isomorphism for the alternating products of the Griffiths bundles, cf. Proposition \ref{prop:BCOViso}.

\begin{proposition}\label{prop:pre-bcov-iso-intro}
Under the assumptions as above, we have a canonical isomorphism of  $\QBbb$-line bundles
\begin{displaymath} \left(\bigotimes_{k=0}^n \Gr(\Hcal^k)^{(-1)^k}\right)^{12} \simeq \langle \cfrak_1(\Omega_{X/S}) \cdot \cfrak_n(\Omega_{X/S}) \rangle^{(-1)^n},
\end{displaymath}
which commutes with base change.
\end{proposition}
The proof of the proposition builds upon a reinterpretation of this specific combination of Griffiths bundles in terms of determinants of cohomology. More precisely, the BCOV line bundle, after Bershadsky, Cecotti, Ooguri and Vafa \cite{bcov}, is defined as 
\begin{displaymath}
    \lambda_{BCOV}(X/S)=\bigotimes_{p=0}^{n}\lambda_{f}(\Omega_{X/S}^p)^{ (-1)^p p }
\end{displaymath}
A direct comparison shows that we in fact can write this as an alternating product of Griffiths bundles: 
\begin{displaymath}
    \lambda_{BCOV}(X/S)= \bigotimes_{k=0}^n \Gr(\Hcal^k)^{(-1)^k}
\end{displaymath}
  This reinterpretation was applied in \cite{CDG2} in the context of Calabi--Yau manifolds and more generally in \cite{Mordant2022Griffiths} in the context of Kato heights. 

The following statement was conjectured in \cite[Section 6, Conjecture 1]{cdg3}, which is part of a rigorous mathematical formulation of the genus one mirror symmetry conjecture of Bershadsky, Cecotti, Ooguri and Vafa in \cite{bcov}. See Theorem \ref{thm:bcoviso} and Corollary \ref{cor:BCOV-analytic} below, stated under more general assumptions.

\begin{proposition}[BCOV isomorphism]\label{prop:BCOV-iso-intro}
Suppose that $S$ is reduced and the fibers of $f\colon X\to S$ are geometrically connected with trivial canonical bundle. Then, there is a canonical isomorphism of $\QBbb$-line bundles
\begin{equation}\label{eq:BCOV-iso-intro-2}
    \lambda_{BCOV}(X/S)^{12} \simeq  \left(f_\ast K_{X/S}\right)^{\chi},
\end{equation}
which commutes with base change. Here $\chi\colon S\to\ZBbb$ is the Euler characteristic of the fibers. The isomorphism can be uniquely extended to the analytic category, compatibly with the analytification functor. 
\end{proposition}

In this statement, $\chi$ is to be interpreted as the $\ell$-adic or topological Euler characteristic, depending on the context. It coincides with the locally constant function $s \mapsto \int_{X_{s}} c_{n}(T_{X_{s}})$. Furthermore, the assumption that $S$ is reduced can be dropped if one requires that the fibers have $h^{0,1}=0$, as for strict Calabi--Yau varieties.

In order to highlight the relavance of Proposition \eqref{prop:BCOV-iso-intro}, we briefly review our formulation of genus one mirror symmetry. Let $\Xcal \to \Dbold^{\times}$ be a projective mirror family of Calabi--Yau manifolds defined over a punctured multidisc. Assuming a Hodge--Tate condition on the limiting mixed Hodge structures, one obtains a canonical change of coordinates on $\Dbold$ (the mirror map), together with canonical trivializations on both sides of \eqref{eq:BCOV-iso-intro-2}, allowing the latter to be interpreted as a holomorphic function $F$ on $\Dbold^{\times}$ in the new coordinate. The conjecture from \cite[Section 6, Conjecture 2]{cdg3} asserts that this function should equal the exponential of a generating series of genus one Gromov--Witten invariants of the mirror Calabi--Yau variety. In the case of the mirror family of Calabi--Yau hypersurfaces in projective space, we established in \cite[Theorem 6.13]{cdg3} a variant of this conjecture, replacing \eqref{eq:BCOV-iso-intro-2} with a weaker isomorphism derived from the arithmetic Riemann--Roch theorem of Gillet and Soulé \cite{GS-ARR}. The connection with the Deligne--Riemann--Roch isomorphism sheds new light on the geometric structures underlying mirror symmetry, and may provide a natural framework for its extension to higher genus.

Finally, we note that Yoshikawa proved in \cite[Theorem 1.11]{yoshikawa-orbifold} a variant of the BCOV isomorphism on coarse moduli spaces of polarized Calabi--Yau threefolds, by analytic means and the theory of the BCOV invariant. The relationship between both should follow from the compatibility of \eqref{eq:BCOV-iso-intro-2} with the BCOV invariant: the latter is expected to be the norm of our isomorphism with respect to $L^{2}$-metrics.

\subsection{Relations to other works}
Other authors have also contributed to Deligne's program. Among alternative directions, Franke developed a parallel approach in a series of papers \cite{FrankeChow, FrankeChern, Franke}. As in our work, he categorifies the individual terms in the Grothendieck--Riemann--Roch formula. His framework, however, rests on a homological formalism that departs fundamentally from Deligne’s and relies crucially on regularity assumptions. This restriction limits potential applications to moduli spaces, which are often singular, and likewise constrains the use of complex or Berkovich analytifications, where Hilbert scheme techniques are needed. Related work by the first author on a functorial Adams--Riemann--Roch isomorphism \cite{Dennis-these} imposes similar regularity conditions. More recently, D. Rössler proposed an alternative approach to the functorial Adams--Riemann--Roch isomorphism \cite{Rossler-ARR}, based on fixed point formula techniques. His contributions include the derivation of explicit denominators, though under several technical geometric hypotheses, such as smoothness of the morphisms. These developments have been a source of inspiration for us. For a more detailed historical discussion, see \cite[Section 1.2]{DRR1}.

Deligne’s original motivations were closely related to early developments in arithmetic intersection theory, as seen, for example, in the foundational works of Arakelov \cite{Arakelov2, Arakelov1} and Faltings \cite{Faltings84}. In particular, he proposed to understand the Quillen metric on the determinant of the cohomology via natural isomorphisms and canonical Hermitian metrics on intersection bundles, structures which now arise naturally in the context of Arakelov geometry. These ideas ran parallel to the work of Gillet and Soulé on higher-dimensional Arakelov theory, which culminated in their celebrated arithmetic Riemann–Roch theorem \cite{GS-ARR}, developed with significant analytic inputs from Bismut. The holomorphic analytic torsion studied in these works is the $\Ccal^{\infty}$ counterpart of our Deligne--Riemann--Roch isomorphism. This analogy is further reinforced by the theory of generalized holomorphic analytic torsion classes developed by Burgos, Freixas, and Li\c{t}canu \cite{BFL2}. We will examine the compatibility of our isomorphism with the analytic torsion and the Quillen metric, along with other analytic aspects such as the BCOV invariant, in a subsequent work. 

\subsection{Organization of the paper}

This article is a continuation of our previous work \cite{DRR1}, where the formalism of intersection distributions was developed and several first applications were given. In order to facilitate the reading of the present work, we begin with a utilitarian reminder of \cite{DRR1} in Section \ref{section:preliminaries}, along with some complements.

In Section \ref{section:DRR-isos} we discuss the abstract concept of Deligne--Riemann--Roch isomorphism, which in particular encapsulates the functorial behavior we might expect of it.

In Section \ref{sec:def-normal-cone} we recall the theory of the deformation to the normal cone, that we need to adapt to our setting, and we derive some consequences for direct images along zero sections of projective bundles, at the level of virtual categories.

Section \ref{section:closed-immersions} is devoted to the construction of the Deligne--Riemann--Roch isomorphism for regular closed immersions, by the deformation to the normal cone technique. This section discusses the expected properties, such as the compatibility with the projection formula and the composition of regular closed immersions. The main result of this section is Theorem \ref{thm:DRRi-general}, which provides the existence and a characterization of the Deligne--Riemann--Roch isomorphism for regular closed immersions.

In Section \ref{section:projective-bundles} we construct the Deligne--Riemann--Roch isomorphism for projective bundles (cf. Theorem \ref{thm:RR-P(E)}), and we also establish various expected features. 

In Section \ref{sec:construction-DRR}, we analyze several key compatibilities between the Deligne--Riemann--Roch isomorphisms for regular closed immersions and for projective bundles. These compatibilities are required to define the Deligne--Riemann--Roch isomorphism for lci morphisms by factoring them into a regular closed immersion followed by a projective bundle. This analysis leads to Theorem \ref{thm:general-DRR}, which provides a line-distribution formulation of the Deligne--Riemann--Roch isomorphism. We then specialize to the determinant of cohomology in Theorem \ref{thm:DRR-det-coh}, which recovers Theorem \ref{thm:A} except for the compatibility with Grothendieck duality, handled separately in Theorem \ref{thm:Groth-duality}.

Finally, Section \ref{sec:compatibilities} concludes with a discussion of the applications presented in the Introduction, and some others, in more detail.

\endgroup
\setcounter{theorem}{\thetmp}

\section{Preliminaries}\label{section:preliminaries}
The main results and methods of this article build upon the theory of intersection distributions developed in our previous work \cite{DRR1}, itself a refinement of the formalism of intersection bundles introduced by Elkik \cite{Elkikfib}. In this section, we review the foundations of this formalism from a utilitarian perspective, with the aim of facilitating the reading of the remainder of the article. We also take this opportunity to fix some conventions and provide a few complements to \cite{DRR1}.

\subsection{Notation and conventions}\label{sec:notations-conventions}
In this subsection, we collect several conventions that we use throughout this article. These are common to our earlier paper, cf. \cite[Section 1.5]{DRR1}. 

A vector bundle on a scheme $X$ is a locally free sheaf of $\Ocal_{X}$-modules of finite rank. The rank is in general a locally constant function for the Zariski topology. A vector bundle of constant rank one is equivalently called a line bundle. 

If $\Ecal$ is a vector bundle on a scheme $X$, we work with Grothendieck's convention for the associated projective bundle $\PBbb(\Ecal)$, as parametrizing line bundle quotients of $\Ecal$. Hence, it can be represented by $\mathrm{Proj}\Sym\Ecal$, and if $\pi$ is the structure map over $X$, then there is a universal exact sequence
\begin{displaymath}
    0\to \Kcal\to\pi^{\ast}\Ecal\to\Ocal(1)\to 0. 
\end{displaymath}
In the case $\Ecal=\Ocal_{X}$, we identify $\PBbb(\Ocal_{X})$ with $X$, so that $\pi$ is the identity map. 

A morphism of schemes $X\to S$ is said to be projective if it factors through a closed immersion $X\to\PBbb(\Ecal)$, where $\Ecal$ is a vector bundle of constant rank on $S$. We say that the composition $X\to\PBbb(\Ecal)\to S$ is a projective factorization of $X\to S$. The morphism $X\to S$ is called locally projective if it is projective locally over $S$.

We will need the following definition of a particularily well-behaved morphism, which is required by many of our constructions.

\begin{definition-intro}[{Condition $(C_n)$}]
A morphism $f\colon X\to S$ satisfies $(C_n)$ if it is locally projective, faithfully flat of finite presentation, and of pure relative dimension~$n$.
\end{definition-intro}

We will deal with regular closed immersions of schemes in a relative setting. More precisely, let $X\to S$ and $Y\to S$ be schemes that satisfy the conditions $(C_{n})$ and $(C_{n-c})$, and let $Y\to X$ be a closed immersion. In this setting, we have two equivalent notions of regularity. The first one is the usual regularity in the sense that $Y\to X$ is locally cut out by a finite regular sequence of functions. The other one is the regularity in the sense that $Y\to X$ is locally cut out by a Koszul-regular sequence. Usual regularity entails Koszul-regularity. The flatness condition on the morphism guarantees that Koszul regularity holds on fibers. Since the fibers are Noetherian, Koszul-regularity on fibers entails usual regularity on fibers. Finally, the flatness and finite presentation conditions ensure that the usual regularity on fibers is equivalent to the usual regularity of $Y\to X$. Hence, in this context, both notions are equivalent and will be used indistinguishably. Moreover, these conditions commute with base change. We refer to \cite[Proposition 19.2.4]{EGAIV4} and \cite[\href{https://stacks.math.columbia.edu/tag/0638}{0638}]{stacks-project} for details on these notions. 

If $\Ecal$ is a vector bundle on a scheme $X$, and $\sigma$ is a global section of $\Ecal$, let $Y$ be the zero-locus scheme of $\sigma$. We denote by $K(\sigma)$ the associated Koszul complex. We say that $\sigma$ is a regular section of $\Ecal$ if $K(\sigma)$ is a resolution of $\Ocal_{Y}$. If $X$ and $Y$ satisfy the conditions $(C_{n})$ and $(C_{n-c})$, respectively, then by the previous paragraph $Y\to X$ is a regular closed immersion. Furthermore, in this situation, $\sigma$ remains regular after any base change $S^{\prime}\to S$.

The empty scheme will appear exceptionally in certain intermediate arguments. For instance, in the deformation to the normal cone, we will encounter a nowhere-vanishing section of a vector bundle. In such situations, we regard the empty scheme as satisfying the condition $(C_{-\infty})$ over any base scheme $S$, and all associated intersection bundles and line distributions (cf. \textsection\ref{subsec:int-bun-line-dist} below) will be taken to be trivial, namely equal to $\Ocal_{S}$. For a scheme $X$, we also consider the canonical closed immersion $\emptyset \to X$ as a regular closed immersion of codimension $+\infty$. Most of the time, however, unless otherwise stated, whenever we refer to the condition $(C_{n})$, we tacitly assume that $n \geq 0$ is an integer.

A morphism of schemes $f\colon X\to Y$ is a local complete intersection, or simply lci, if locally on $X$ it factors as a Koszul-regular immersion $X\hookrightarrow P$ followed by a smooth morphism $P\to Y$. If $X$ and $Y$ satisfy a condition of type $(C_{n})$ over a base scheme $S$, then the closed immersion is necessarily regular in the usual sense. Furthermore, in this situation, $f$ remains a local complete intersection after any base change $S^{\prime}\to S$. Finally, if $f$ is projective and $X\to\PBbb(\Ecal)\to Y$ is a projective factorization, then  $X\to\PBbb(\Ecal)$ is Koszul-regular. More generally, if we only assume that $X$ and $Y$ satisfy a condition of type $(C_{n})$, then $f\colon X\to Y$ is projective locally over $S$, and hence  a projective factorization exists locally over $S$. See \cite[Section 19.3]{EGAIV4} and \cite[\href{https://stacks.math.columbia.edu/tag/068E}{068E}]{stacks-project} for further details. 

We recall the definition of a divisorial scheme from \cite[Definition 2.1.1]{ThomasonTrobaugh}.
\begin{definition}[Divisorial scheme]\label{def:divisoria-prelim}
A scheme $X$ is called divisorial if it is quasi-compact and quasi-separated, and admits an ample family of line bundles. 
\end{definition}
Examples of divisorial schemes include affine schemes. In general, if $Y$ is divisorial and $X\to Y$ is quasi-projective over $Y$, then $X$ is divisorial too. This applies to closed immersions, quasi-compact open immersions, and projective morphisms that satisfy the condition $(C_{n})$. We refer to \cite[Examples 2.1.2]{ThomasonTrobaugh} for more examples. 

Finally, we recall that in \cite{DRR1} we fixed a Grothendieck universe for practical reasons and also to conform with \cite{ThomasonTrobaugh}. Up to equivalence of categories, the constructions do not depend on this choice. Since we rely on \cite{DRR1}, here we also implicitly fix a universe. 

\subsection{Virtual categories and their functorial properties}
Our categorical framework starts with the formalism of virtual categories, originally developed by Deligne in \cite{Deligne-determinant}. We elaborated on this theory in \cite[Section 4]{DRR1}, drawing on foundational work by Waldhausen \cite{Waldhausen}, Thomason \cite{ThomasonTrobaugh}, Knudsen \cite{Knudsen}, and Muro, Tonks and Witte \cite{Muro:determinant}. The theory is based on the notion of commutative Picard category, which we take as given. For readers unfamiliar with this concept, we refer to \cite[Section 2]{DRR1} for an overview. In this section, we briefly recall the key ingredients of the theory of virtual categories.

\subsubsection{Virtual categories} \label{intro:virtualcategories}
We follow the discussion in \cite[Sections 4.1--4.3]{DRR1}. Denote by $(\Vect_{X},\iso)$ the exact category of vector bundles on a scheme $X$, with isomorphisms as morphisms. Given a commutative Picard category $(\Pcal,\otimes)$, such as a Picard category of line bundles, there is a notion of a multiplicative (or sometimes determinant) functor
\begin{displaymath}
    F\colon (\Vect_{X},\iso)\to\Pcal.
\end{displaymath}
The main defining properties are the following: $F$ is multiplicative on short exact sequences, in a way compatible with admissible filtrations, and it is also compatible with the commutativity constraints of the addition on $(\Vect_{X},\iso)$ and the tensor product on $\Pcal$. Sometimes the monoidal structure on $\Pcal$ is seen as an addition, in which case we may adapt the terminology and say, for instance, that $F$ is additive on short exact sequence.

The virtual category of $X$, denoted by $V(X)$, is a commutative Picard category, endowed with a universal multiplicative functor 
\begin{displaymath}
    [\ \cdot\ ]\colon (\Vect_{X},\iso)\to V(X).
\end{displaymath}
The monoidal structure on $V(X)$ is induced by the direct sum of vector bundles. The universal multiplicative functor factors via the derived category of bounded complexes of vector bundles $D^{b}(\Vect_{X})$, with the quasi-isomorphisms as morphisms. 

The category $V(X)$ can be constructed as the fundamental groupoid of the loop space of Quillen's $Q$-construction. As such, its homotopy groups are given in terms of Quillen's $K$-theory of vector bundles as
\begin{displaymath}
    \begin{split}
        &\pi_{0}(V(X))=\text{ group of isomorphism classes of objects of }V(X)\ \simeq\ K_{0}(X),\\
        &\pi_{1}(V(X))=\text{ group of automorphisms of any object in }V(X)\ \simeq\ K_{1}(X),\\
        &\pi_{i}(V(X))=0\text{ if } i\geq 2.
    \end{split}
\end{displaymath}
This relationship is central in view of \cite[Lemma 2.2]{DRR1}, stating that a functor of commutative Picard categories $\Pcal\to\Pcal^{\prime}$ induces an equivalence of categories if and only if it induces isomorphisms on the $\pi_{0}$ and $\pi_{1}$ groups. Hence, for the purpose of establishing equivalences between virtual categories, we are often reduced to dealing with analogous $K$-theoretic statements. 

As a prototype example of the above formalism, we mention the determinant of vector bundles. Let $\Picfr(X)$ be the Picard category of line bundles on $X$ and $\Picfr_{\mathrm{gr}}(X)$ the Picard category of $\ZBbb$-graded line bundles $(L,r)$, where the commutativity constraint for the tensor product incorporates the Koszul rule of signs with respect to the degree. If $E$ is a vector bundle of rank $r$, then we define $\det E=(\Lambda^{r}E,r)$. The construction induces a determinant functor, and hence a functor of commutative Picard categories
\begin{equation}\label{eq:det-prel}
    \det\colon V(X)\to\Picfr_{\mathrm{gr}}(X).
\end{equation}
The determinant functor is compatible with pullback functoriality; see \textsection \ref{subsubsec:functorial-properties-virtual-prelim} below for the functoriality properties of virtual categories. Notice that without the grading, $\det$ does not induce a functor of commutative Picard categories with values in $\Picfr(X)$.

For the purpose of performing direct images, one generalizes the above construction to Waldhausen categories. In particular, let $\Pcal_{X}$ denote the category of perfect complexes of $\Ocal_{X}$-modules, of globally bounded (also called finite) Tor-amplitude, together with the quasi-isomorphisms as morphisms. Then, we can associate to it a virtual category $V(\Pcal_{X})$, which we call the virtual category of perfect complexes. The inclusion $(\Vect_{X},\iso)\to\Pcal_{X}$ induces a functor of commutative Picard categories
\begin{equation}\label{eq:VX-to-VPX}
    V(X)\to V(\Pcal_{X}),
\end{equation}
which commutes with base change and is an equivalence of categories when $X$ is divisorial. Besides, we note that the Tor-amplitude condition above is automatic on quasi-compact schemes.

To conclude this subsection, we observe that by the work of Knudsen and Mumford \cite{KnudsenMumford}, the determinant functor \eqref{eq:det-prel} extends to a functor
\begin{equation}\label{eq:det-prel-bis}
    \det\colon V(\Pcal_{X})\to\Picfr_{\mathrm{gr}}(X),
\end{equation}
which commutes with base change. The construction of this extension relies on on three facts. First, the sought extension is to be compatible with pullback functoriality and, in particular, with restriction to open subschemes. Second, on divisorial schemes, the natural functor \eqref{eq:VX-to-VPX} is an equivalence of categories. Finally, by gluing, one reduces to affine schemes, in which case one can use the equivalence \eqref{eq:VX-to-VPX} and the already constructed functor \eqref{eq:det-prel}. The extended functor inherits from \eqref{eq:det-prel} the properties of a multiplicative functor.

\subsubsection{Functorial properties of virtual categories}\label{subsubsec:functorial-properties-virtual-prelim}
The virtual categories considered above inherit a tensor product structure from that of vector bundles or complexes. They also carry pullback functoriality, induced by the pullback on vector bundles or complexes. Moreover, the duality operator acts on the virtual categories, compatibly with pullback functoriality. For virtual categories of perfect complexes, these operations are taken in the derived sense, though in this context we will use non-derived notation. For instance, we denote the dual of a virtual perfect complex $E$ by $E^{\vee}$, cf. \cite[\href{https://stacks.math.columbia.edu/tag/0FP7}{0FP7}]{stacks-project}. Thus, the virtual category construction, whether for vector bundles or for perfect complexes of finite Tor-amplitude, defines a category fibered in groupoids over the category of schemes, equipped with a duality autoequivalence.

For direct images, one needs to impose some perfection conditions on the morphisms. We follow \cite[Proposition 4.12]{DRR1} and the subsequent statements. Let be given a proper morphism of quasi-compact schemes $f\colon X\to Y$, and assume that the derived functor $Rf_{\ast}$ preserves perfect complexes. Note that the finite Tor-amplitude condition is automatic on quasi-compact schemes. Then, $Rf_{\ast}$ induces a pushforward functor
\begin{equation}\label{eq:pushfwd-1}
    f_{!}\colon V(\Pcal_{X})\to V(\Pcal_{Y}).
\end{equation}
This is a functor of commutative Picard categories, and its formation is compatible with the composition of morphisms: 
\begin{equation}\label{eq:fun-composition-virtual}
    (g\circ f)_{!}\simeq g_{!}\circ f_{!}.
\end{equation}
It also satisfies the projection formula:
\begin{equation}\label{eq:proj-formula-virtual-cat}
    f_{!}(E\otimes f^{\ast}F)\simeq f_{!}E\otimes F.
\end{equation}
Finally, it commutes with any Tor-independent base change $g\colon Y^{\prime}\to Y$, assuming that the derived direct image along the base-changed map $f^{\prime}\colon X^{\prime}\to Y^{\prime}$ preserves perfect complexes of finite Tor-amplitude:
\begin{equation}\label{eq:virtualbasechange}
    g^{\ast}f_{!}E\simeq f^{\prime}_{!}g^{\prime\ast}E,
\end{equation}
where $g^{\prime}\colon X^{\prime}\to X$ is the projection map.

We note that if $X$ and $Y$ are divisorial, we can then deduce from \eqref{eq:pushfwd-1} a pushforward functor with similar properties
\begin{displaymath}
    f_{!}\colon V(X)\to V(Y),
\end{displaymath}
thanks to the equivalences of categories of the form \eqref{eq:VX-to-VPX} in this case.

In practice, our morphisms will satisfy the requirements of \cite[Proposition 4.12]{DRR1}. We will mostly deal with morphisms satisfying the condition $(C_{n})$ and with proper perfect (e.g. lci) morphisms between quasi-compact, and even divisorial, schemes \cite[Proposition 4.15]{DRR1}. 

As an example of application, we mention the construction of the determinant of the cohomology. If $f\colon X\to Y$ is a morphism of schemes for which $f_{!}$ is defined, then composing with \eqref{eq:det-prel-bis} (for $Y$) we obtain a functor of commutative Picard categories
\begin{equation}\label{eq:det-coh-prelim}
    \lambda_{f}=\det f_{!}\colon V(\Pcal_{X})\to\Picfr_{\mathrm{gr}}(Y),
\end{equation}
which is a reinterpretation of the Knudsen--Mumford construction \cite{KnudsenMumford}. Actually, this construction is local on the base, which allows one to weaken the assumptions by a gluing argument. In particular, if $f$ satisfies the condition $(C_{n})$, then $\lambda_{f}$ is defined without the quasi-compactness restriction on the schemes, and the formation of the functor commutes with base change. 

\subsection{Signs and rationalization}\label{subsec:signs}
In the theory of commutative Picard categories one encounters the notion of sign. For such a category $(\Pcal,\otimes)$, let $c_{A,A}\colon A\otimes A\to A\otimes A$ be the commutativity isomorphism of an object $A$. The translation functor by a fixed object in $\Pcal$ is an autoequivalence of categories. Hence, translating $c_{A,A}$ by $A\otimes A$ induces an automorphism $\varepsilon(A)$ of the neutral object $0$. The construction gives rise to a sign homomorphism $\varepsilon\colon \pi_{0}(\Pcal)\to\pi_{1}(\Pcal)$, with $\varepsilon^{2}=1$. The image is referred to as signs, and in $\Pcal$ it hence makes sense to say that a diagram commutes up to a sign. In general, $\varepsilon$ is not the identity map and we say that $\Pcal$ is strictly commutative if it is. 

In \cite[Section 2]{DRR1}, we introduced a rationalization procedure for a commutative Picard category $\Pcal$. This produces a strictly commutative Picard category $\Pcal_{\QBbb}$. In particular, a diagram that commutes up to a sign in $\Pcal$ gives rise to a commutative diagram in $\Pcal_{\QBbb}$. 

Rationalizing the Picard category $\Picfr_{\mathrm{gr}}(X)$ of graded line bundles and then forgetting the grading produces the usual Picard category of $\QBbb$-line bundles, denoted by $\Picfr(X)_{\QBbb}$. Since the latter is strictly commutative, the determinant functors \eqref{eq:det-prel} and \eqref{eq:det-coh-prelim} induce functors of commutative Picard categories with values in $\QBbb$-line bundles. The rationalization procedure also applies to virtual categories. In some circumstances, it will be enough for us to work with the rational virtual categories, thus dispensing with the need to care about sign issues.

\subsection{Intersection bundles and line distributions}\label{subsec:int-bun-line-dist}In this subsection we review the basics of the theory of intersection bundles and line distributions. This is the core of our previous work \cite{DRR1}, and we refer to it for a comprehensive exposition.  

\subsubsection{Intersection bundles}
Let $f\colon X\to S$ be a morphism that satisfies the condition $(C_{n})$. Suppose that we are given vector bundles $E_i$ on $X$ and integers $k_i\geq 0$, for $i = 1, \ldots, m$, such that $\sum k_i = n+1$. Then, after \cite[Section V]{Elkikfib} and \cite[Section 7]{DRR1}, there is a canonically associated line bundle on $S$, called intersection bundle,
\begin{equation}\label{eq:intersectionbundle}
    \langle \cfrak_{k_1}(E_1) \cdots\cfrak_{k_m}(E_m) \rangle_{X/S},\quad\text{or simply}\quad \langle \cfrak_{k_1}(E_1) \cdots\cfrak_{k_m}(E_m) \rangle.
\end{equation}
This bundle is expected to represent the direct image of the Chern classes. This is known for a divisorial base $S$ \cite[Proposition 7.1]{DRR1}. Thus, in this case, we have an equality of Chern classes on $S$
\begin{equation}\label{eq:c1-direct-image}
    c_{1}\big( \langle \cfrak_{k_1}(E_1)\cdots \cfrak_{k_m}(E_m) \rangle\big)=f_\ast( c_{k_1}(E_1)\cdots c_{k_m}(E_m)).
\end{equation}
It is an open question to extend this property to general base schemes. Whenever the $E_{i}=L_{i-1}$ are line bundles and the $k_{i}=1$, we recover the Deligne pairing \cite[Section 6]{DRR1}
\begin{displaymath}
    \langle L_{0},\ldots,L_{n}\rangle_{X/S},\quad\text{or simply}\quad\langle L_{0},\ldots,L_{n}\rangle.
\end{displaymath}
The latter can be presented explicitly in terms of generators and relations. The construction of the more general \eqref{eq:intersectionbundle} mimics the method via Segre classes in \cite[Chapter 3]{Fulton}, and the intersection bundle analogues of the latter are defined as Deligne pairings of tautological bundles of projective bundles.

More generally, the construction of \eqref{eq:intersectionbundle} extends, by tensor product multiplicativity, to any a priori non-commutative polynomial $P$ in Chern classes of total degree $n+1$, with integer coefficients. The resulting line bundle is denoted by $\langle P \rangle_{X/S}$ or simply $\langle P\rangle$. By projecting onto the degree-$(n+1)$ component, this construction extends to arbitrary polynomials or even formal power series with integer coefficients. We may further allow rational coefficients, by considering $\langle P\rangle_{X/S}$ as a $\QBbb$-line bundle. 

The intersection bundles in \eqref{eq:intersectionbundle} are functorial in isomorphisms of the vector bundles $E_i$ and isomorphisms of $S$-schemes, and commute with base change. Frequently, the functoriality in the $S$-schemes is subsumed in the base-change property and not mentioned explicitly, since an isomorphism of $S$-schemes may be seen as a base-change Cartesian diagram. The construction also satisfies a set of natural isomorphisms, corresponding to properties of Chern classes and the direct image thereof whenever \eqref{eq:c1-direct-image} holds. These are stated in \cite[Proposition 7.5]{DRR1}. Most importantly, for an exact sequence of vector bundles
\begin{displaymath}
    0 \to E_i^\prime \to E_i \to E_i^{\bis} \to 0
\end{displaymath} 
there is a Whitney-type isomorphism
\begin{equation}\label{eq:iso-Whitney-preliminaries}
    \langle \cfrak_{k_1}(E_1)\cdots \cfrak_{k_i}(E_i)\cdots \cfrak_{k_m}(E_m) \rangle \simeq \bigotimes_{k_i=k^{\prime}_{i}+k^{\bis}_{i}} \langle \cfrak_{k_1}(E_1)\cdots \cfrak_{k_i^\prime}(E_i^\prime)\cdot \cfrak_{k^{\bis}_i}(E^{\bis}_i)\cdots \cfrak_{k_m}(E_m) \rangle,
\end{equation}
which commutes with base change. One may combine the natural isomorphisms between intersection bundles, and a key result of the theory is that these operations commute  \cite[Section 7.3]{DRR1}. Since the precise statement involves a large number of diagrams, we refer the reader to \emph{loc. cit.} for details and  to \textsection\ref{subsubsec:some-natural-isos-prel} below for an example formulated within the language of line distributions.

\subsubsection{Chern categories and categorical characteristic classes}\label{subsubsec:categorical-characteristic}
For the purpose of working with intersection bundles as a genuine relative intersection theory, in \cite[Section 5]{DRR1} we developed a formalism of line distributions, recalled in \textsection\ref{subsubsec:line-dist-prel} below. In this direction, in \emph{loc. cit.} we first introduced a graded ring category $\CHfrak(X)$, called \emph{Chern category} of $X$, whose objects can be thought of as formal power series in Chern classes on $X$ with integer coefficients. We call these objects \emph{Chern power series} on $X$. The grading of $\CHfrak(X)$ is induced by the degree of the Chern classes. Addition of Chern power series defines a structure of strictly commutative Picard category on $\CHfrak(X)$. Multiplication of Chern power series induces a structure of strictly commutative Picard category on the Chern power series whose constant term is isomorphic to $1$. There is also a rational version $\CHfrak(X)_{\QBbb}$ of this construction, with analogous properties, whose objects are thought of as formal power series in Chern classes, with rational coefficients.

The construction of the Chern category is tailored so that the total Chern class induces a functor
\begin{displaymath}
    \cfrak\colon (\Vect_{X},\iso)\to \CHfrak(X)
\end{displaymath}
 which is multiplicative on short exact sequences $0 \to E' \to E\to E^{\bis}\to 0$, namely there is an isomorphism 
 \begin{displaymath}
    \cfrak(E)\simeq \cfrak(E^{\prime})\cdot\cfrak(E^{\bis})
 \end{displaymath}
 compatible with admissible filtrations and some additional constraints. It factors through the virtual category:
 \begin{displaymath}
    \cfrak\colon V(X)\to \CHfrak(X).
\end{displaymath}
The morphisms in the Chern category therefore lift isomorphisms such as \eqref{eq:iso-Whitney-preliminaries} to the categorical level. One can perform algebraic manipulations with Chern power series (addition, subtraction, multiplication, commutativity, associativity, etc.) in any order and functorially.

There is a formalism of categorical characteristic classes in the Chern categories. In particular, additive or multiplicative characteristic classes of vector bundles, such as the Chern character or the Todd genus, can be lifted to $\CHfrak(X)_{\QBbb}$. Additive or multiplicative classes may be interpreted as functors defined on the virtual category, for instance
\begin{equation}\label{eq:char-class-functors}
    \chfrak\colon V(X)\to\CHfrak(X)_{\QBbb},\quad\tdfrak\colon V(X)\to\CHfrak(X)_{\QBbb}.
\end{equation}
The Chern character $\chfrak$ extends to the rational virtual category:
\begin{displaymath}
    \chfrak\colon V(X)_{\QBbb}\to\CHfrak(X)_{\QBbb}.
\end{displaymath}

The formalism of categorical characteristic classes extends to the derived category of bounded complexes of vector bundles, and when $X$ is divisorial, since $V(X)$ and $V(\Pcal_{X})$ are equivalent, we can extend it further to perfect complexes. However, we stress that, in general, if $E$ is a virtual perfect complex on $X$, then $\chfrak(E)$ and $\tdfrak(E)$ are not defined. In practice, we will be given a morphism $X\to S$ satisfying the condition $(C_n)$, and we will tackle this issue by localizing over $S$ in order to reduce to the case where $X$ is divisorial. In the context of line distributions, a detailed procedure in this sense is expounded in \textsection\ref{subsubsec:gluing-functors}.

\subsubsection{Line distributions}\label{subsubsec:line-dist-prel} Chern categories enjoy pullback functoriality, induced by pullback of vector bundles, which upgrades the construction into a category fibered in groupoids over the category of schemes. However, Chern categories do not admit pushforward functors. The language of line distributions introduced in \cite[Section 5]{DRR1} is meant to remedy this. It is modeled on the theory of intersection bundles, and is inspired by the theory of distributions (or currents) in analysis\footnote{In this analogy, the Chern power series correspond to the test forms.} and by Fulton's bivariant intersection theory \cite[Chapter 17]{Fulton}. 

A line distribution $T_{X/S}$ for $X\to S$ is a collection of functors 
\begin{displaymath}
    \CHfrak(X^{\prime})_{\QBbb}\to\Picfr(S^{\prime})_{\QBbb},
\end{displaymath}
where $S^{\prime}\to S$ is any base change and we set $X^{\prime}=X\times_{S}S^{\prime}$. Line distributions fulfill some additional conditions, of which we do not detail all. But in particular, if $Q$ denotes a Chern power series on a base change of $X$, then $T_{X/S}(Q)$ satisfies:
\begin{enumerate}
    \item It is compatible with base change. In particular, it is compatible with restriction to open subschemes of $S$. 
    \item It is additive in $Q$ and it is compatible with Whitney-type isomorphisms.
      \item It has bounded denominators: there exists a universal integer $N\geq 1$ such that $T_{X/S}(Q)^{N}$ is an actual line bundle whenever $Q$ has integer coefficients, and induces a functor on the integral Chern category $\CHfrak(X)$. We then say that $T_{X/S}^{N}$ is an entire line distribution.  
\end{enumerate}

We note that a line distribution $T_{X/S}$ by definition induces a line distribution $T_{X^{\prime}/S^{\prime}}$ for $X^{\prime}\to S^{\prime}$, for any base change $q\colon S^{\prime}\to S$. We will write
\begin{equation}\label{eq:base-change-lin-dist}
    T_{X^{\prime}/S^{\prime}}=q^{\ast}T_{X/S}=\text{line distribution for } X^{\prime}\to S^{\prime} \text{ induced by } T_{X/S}.
\end{equation}

\subsubsection{The Picard category of line distributions}\label{subsubsec:pic-cat-line-dist}
The line distributions (resp. entire line distributions) for $X\to S$ constitute a strictly commutative Picard category $\Dcal(X/S)$ (resp. $\Dcal_{\ZBbb}(X/S)$):
\begin{enumerate}
    \item\emph{Monoidal structure.} The tensor product of line bundles induces a tensor product of line distributions. To simplify the notation, we frequently write the tensor product of line distributions in additive form: $T_{X/S}+T^{\prime}_{X/S}$ in place of $T_{X/S}\otimes T^{\prime}_{X/S}$.
    \item\emph{Isomorphisms. } An isomorphism of line distributions $T_{X/S}\simeq T^{\prime}_{X/S}$ is an isomorphism of the $\QBbb$-line bundles (resp. line bundles) $T_{X/S}(Q)\simeq T^{\prime}_{X/S}(Q)$ which respects the above data. In particular, an appropriate power of the isomorphism induces an isomorphism of entire line distributions. 
\end{enumerate}
Associating, with a morphism of schemes $S^{\prime}\to S$, the category $\Dcal(X^{\prime}/S^{\prime})$ (resp. $\Dcal_{\ZBbb}(X^{\prime}/S^{\prime})$), for $X^{\prime}=X\times_{S}S^{\prime}$, defines a category fibered in groupoids over $(\mathrm{Sch}/S)$. An important fact of line distributions and their isomorphisms is that they are determined by their restriction to affine base schemes \cite[Proposition 5.31, Remark 5.32]{DRR1}.

We shall perform several formal operations on line distributions, inspired by the analogous ones in the theory of distributions in analysis:
\begin{enumerate}
    \item \emph{Product by a Chern power series.} If $T_{X/S}$ is a line distribution and $P$ is a Chern power series on $X$, then we can define the product $P\cdot T_{X/S}$ by
\begin{equation}\label{eq:left-product}
    (P\cdot T_{X/S})(Q)=T_{X/S}(Q\cdot P),
\end{equation}
which is again a line distribution on $X$. This induces a categorical module structure on $\Dcal(X/S)$ over the Chern category $\CHfrak(X)_{\QBbb}$. Here and elsewhere, we do not write the corresponding base-changed isomorphisms, which are part of the line distribution data recalled above. 

We will also have use for the following variant, which we did not introduce in \cite{DRR1}, and that enjoys similar properties:
\begin{equation}\label{eq:right-product}
    (T_{X/S}\cdot P)(Q)=T_{X/S}(P\cdot Q).
\end{equation}
For concreteness, below, all identities stated for the left product have obvious right-product analogues.
    \item \emph{Pushforward of line distributions.} If $h\colon Y\to X$ is a morphism of schemes, such that $Y\to S$ satisfies the condition $(C_{m})$ for some $m$, and $T_{Y/S}$ is a line distribution on $Y$, then we define
    \begin{displaymath}
        (h_{\ast}T_{Y/S})(Q)=T_{Y/S}(h^{\ast} Q),
    \end{displaymath}
    which is a line distribution on $X$. This induces a functor of commutative Picard categories $h_{\ast}\colon\Dcal(Y/S)\to\Dcal(X/S)$, clearly compatible with composition. 
    \item\emph{Projection formula.} The following \emph{projection formula} is automatic from the previous definitions:
        \begin{equation}\label{eq:projection-formula-preliminaries}
              h_{\ast}(h^{\ast}P\cdot T_{Y/S})=P\cdot h_{\ast}T_{Y/S}.
         \end{equation}
         We have a similar projection formula for the product from the right.
    \item \emph{Pushforward and base change.} Continuing with the previous point, if $q\colon S^{\prime}\to S$ is any morphism, then the convention \eqref{eq:base-change-lin-dist} yields the following obvious compatibility of pushforward with base change:
    \begin{equation}\label{eq:base-change-bete}
        q^{\ast}h_{\ast}T_{Y/S}=h_{\ast}^{\prime} q^{\ast}T_{Y/S},
    \end{equation}
    where $h^{\prime}\colon Y^{\prime}\to X^{\prime}$ is the morphism deduced from $h$ by base change along $q$. 
\end{enumerate}

Finally, we complete the list with a relevant nontrivial fact:
\begin{enumerate}[resume]
    \item\label{item:splitting-principle}\emph{Splitting principle.} While there is no splitting principle available for Chern categories, there is a substitute for line distributions of the form $\phi(E)\cdot T_{X/S}$, where $E\mapsto\phi(E)$ is a characteristic class on vector bundles: to check that two such distributions $\phi(E)\cdot T_{X/S}$ and $\psi(E)\cdot T_{X/S}$ are isomorphic for all $E$, one can restrict oneself to direct sums of line bundles. This is a particular instance of more general splitting principles for line functors, first proven in \cite[Section 2]{Eriksson-Freixas-Wentworth} and summarized in \cite[Section 5.4]{DRR1}, to which we refer for an account.
\end{enumerate}

\subsubsection{Intersection distributions}
The main examples of line distributions are provided by the theory of intersection bundles. For a fixed Chern power series $Q$ on $X$, in  \cite[Section 8]{DRR1} we show that the assignment 
\begin{displaymath}
    Q \mapsto \langle Q \cdot P\rangle_{X/S}
\end{displaymath}
defines a line distribution on $X$.\footnote{Recall from \textsection\ref{subsubsec:line-dist-prel} that a line distribution is defined for all base changes of $X\to S$, but for simplicity, we omit this from the notation.} This is a subtle nontrivial fact, related to the interaction between the commutativity constraint on the virtual category and the Whitney isomorphism for intersection bundles \eqref{eq:iso-Whitney-preliminaries}, cf. \cite[Lemma 8.1]{DRR1}. We denote the line distribution associated with $P$ by 
\begin{displaymath}
    [P]_{X/S}, \quad\text{or simply}\quad [P].
\end{displaymath}
We call intersection distributions those line distributions obtained as pushforwards of line distributions of the form $[P]_{X/S}$. 

The following are some useful observations and conventions for intersection distributions: 
\begin{enumerate}
    \item If $P, P^{\prime}$ are Chern power series on $X$, the product rules \eqref{eq:left-product} and \eqref{eq:right-product} provide the identities
\begin{equation}\label{eq:prod-int-dist-prel}
    P\cdot [P^{\prime}]_{X/S}=[P\cdot P^{\prime}]_{X/S}=[P]_{X/S}\cdot P^{\prime}.
\end{equation}   
\item The base-change convention \eqref{eq:base-change-lin-dist} for intersection distributions can be written as 
\begin{displaymath}
    q^{\ast}[P]_{X/S}=[q^{\prime\ast}P]_{X^{\prime}/S^{\prime}},
\end{displaymath}
where $q^{\prime}\colon X^{\prime}\to X$ is the projection map. Similarly, \eqref{eq:base-change-bete} then becomes
\begin{displaymath}
    q^{\ast}h_{\ast}[P]_{Y/S}=h^{\prime}_{\ast}[q^{\prime\ast}P]_{Y^{\prime}/S^{\prime}},
\end{displaymath}
where $q^{\prime}$ now denotes the projection map $Y^{\prime}\to Y$.
\item If $h\colon Y\to X$ is a morphism of schemes such that $Y\to S$ satisfies the condition $(C_{m})$, then the projection formula \eqref{eq:projection-formula-preliminaries} for intersection distributions becomes
\begin{equation}\label{eq:proj-for-int-dist}
    h_{\ast}(h^{\ast} P\cdot [P^{\prime}]_{Y/S})=P\cdot h_{\ast}[P^{\prime}]_{Y/S}.
\end{equation}
We have a similar projection formula for the product from the right.

\item If $i\colon Y\to X$ is a closed immersion such that $Y\to S$ satisfies the condition $(C_{m})$, then we denote the line distribution $i_{\ast}[1]_{Y/S}$ on $X$ by $\delta_{Y/S}$. Hence, if $Q$ is a Chern power series on $X$, we have
\begin{displaymath}
    \delta_{Y/S}(Q)=\langle i^{\ast}Q\rangle_{Y/S}.
\end{displaymath}
In particular, the projection formula in this case can be written as
\begin{equation}\label{eq:rest-dist-prel}
    P\cdot\delta_{Y/S}=[i^{\ast} P]_{Y/S}.
\end{equation}
Note also that $\delta_{X/S}=[1]_{X/S}$ is simply the distribution $\delta_{X/S}(Q)=\langle Q\rangle_{X/S}$.
\end{enumerate}
\subsubsection{Some natural isomorphisms of intersection distributions} \label{subsubsec:some-natural-isos-prel}
Due to its relevance, we recall \cite[Corollary 8.6]{DRR1} in the statement below, summarizing the main types of isomorphisms of intersection distributions and their compatibilities.
\begin{theorem}[{\cite[Corollary 8.6]{DRR1}}]\label{thm:cor86}
Let $f:X \to S$ be a morphism satisfying the condition $(C_n).$ The intersection distributions satisfy the following properties:
\begin{enumerate}
    \item\label{item:cor86-1} (Projection formulas)  Let $g\colon X^{\prime}\to S$ and $h \colon X' \to X$ be morphisms satisfying the conditions $(C_{n+n'})$ and $(C_{n'})$, respectively, and let $P^{\prime}$ be 
    in $\CHfrak(X^{\prime})_{\QBbb}$ have pure degree $\deg P^{\prime} \leq n'+1$. Then, there are canonical isomorphisms of line distributions
    \begin{displaymath}
            h_\ast [P^{\prime}]_{X^{\prime}/S} \simeq [h_{\ast}P^{\prime}]_{X/S},\quad\text{where}\quad h_{\ast}P^{\prime}:=
\begin{cases}
    \cfrak_1 (\langle P^{\prime} \rangle_{X^{\prime}/X}), & \text{if }\; \deg P^{\prime}=n^{\prime}+1.\\ \\
    \int_{X^{\prime}/X}P^{\prime} &  \text{if }\; \deg P^{\prime} = n^{\prime}.  \\ \\
         0,  & \text{if }\; \deg P^{\prime} < n^{\prime}. \\
\end{cases}
\end{displaymath}
In the middle case, we assume that the intersection-theoretic degree of $P^{\prime}$ on the fibers of $X^{\prime}\to X$ is constant, and we denote it by $\int_{X^{\prime}/X}P^{\prime}$.
    \item\label{item:prop-int-dist-whitney} (Whitney isomorphism) Let $0 \to E^{\prime}\to E \to E^{\bis} \to 0$ be a short exact sequence of vector bundles on $X$. Then, there is a canonical isomorphism
    \begin{displaymath}
        [\cfrak_{k}(E)]_{X/S}\simeq \sum_{i=0}^{k} [\cfrak_{i}(E^{\prime})\cdot\cfrak_{k-i}(E^{\bis})]_{X/S},
    \end{displaymath}
    in a way compatible with admissible filtrations.
    \item\label{item:prop-int-dist-c1-det} (First Chern class isomorphism) Let $E$ be a vector bundle on $X$. Then, there is a canonical isomorphism 
 \begin{displaymath}
        [\cfrak_1(E)]_{X/S} \simeq [\cfrak_1(\det E)]_{X/S}
    \end{displaymath}
    in a way that is compatible with the Whitney isomorphism. 
    \item\label{item:prop-int-dist-rank} (Rank triviality) Let $E$ be a vector bundle on $X$ and $q$ an integer such that $q > \rk E$. Then, there is an isomorphism 
    \begin{displaymath}
        [ \cfrak_q(E)]_{X/S} \simeq 0.
        \end{displaymath}
    \item\label{item:prop-int-dist-restriction} (Restriction isomorphism) Let $E$ be a vector bundle of constant rank $r$ on $X$. Suppose that $\sigma$ is a regular section of $E$,  whose zero locus $Y$, possibly empty, is flat over $S$. Then, there is a canonical isomorphism
    \begin{displaymath}
        [ \cfrak_r(E) ]_{X/S} \simeq \delta_{Y/S}, 
    \end{displaymath}
    where $i\colon Y\hookrightarrow X$ is the closed immersion of $Y$ in $X$. 
    \item\label{item:prop-int-dist-birational} (Birational invariance) Let $g: X^{\prime} \to S$ be a morphism satisfying condition $(C_n)$. Assume that there exists a morphism $h\colon X^{\prime} \to X$ and a quasi-compact open immersion $U \to X$, such that $h^{-1}(U) \to U$ is an isomorphism and $U$ is fiberwise dense in $X$. Then, there is a canonical isomorphism
    \begin{displaymath}
        h_\ast \delta_{X'/S}\simeq \delta_{X/S}.
    \end{displaymath}
    In particular, $h_{\ast}[h^{\ast}P]_{X'/S}\simeq [P]_{X/S}$. 
\end{enumerate}
These operations can be composed with each other in a natural way, and commute with each other. 
\end{theorem}

For intersection bundles, we refer the reader to  \cite[Section 7.3]{DRR1}, where the commutativity of the operations is developed at length. For concreteness, we discuss the case of the Whitney and the restriction isomorphisms. Suppose that we are given an exact sequence of vector bundles as in \eqref{item:prop-int-dist-whitney} above, and a vector bundle $F$ as in \eqref{item:prop-int-dist-restriction} (in place of $E$). Then, we can write a diagram of isomorphisms of line distributions
\begin{displaymath}
\resizebox{\textwidth}{!}{
    \xymatrix{
        [\cfrak_{k}(E)\cdot\cfrak_{r}(F)]_{X/S}\ar[r]^{\eqref{eq:prod-int-dist-prel}}\ar[d]_{\eqref{eq:prod-int-dist-prel}} &[\cfrak_{k}(E)]\cdot\cfrak_{r}(F)\ar[r]^-{\text{Whitney}}    &\sum_{i=0}^{k} [\cfrak_{i}(E^{\prime})\cdot\cfrak_{k-i}(E^{\bis})]_{X/S}\cdot\cfrak_{r}(F)\ar[r]^{\eqref{eq:prod-int-dist-prel}} &\sum_{i=0}^{k} [\cfrak_{i}(E^{\prime})\cdot\cfrak_{k-i}(E^{\bis})\cdot\cfrak_{r}(F)]_{X/S}\ar[d]^{\eqref{eq:prod-int-dist-prel}}\\
        \cfrak_{k}(E)\cdot[\cfrak_{r}(F)]_{X/S}\ar[d]_-{\text{restriction}}    &   &   &\left(\sum_{i=0}^{k} \cfrak_{i}(E^{\prime})\cdot\cfrak_{k-i}(E^{\bis})\right)\cdot[\cfrak_{r}(F)]_{X/S}\ar[d]^{\text{restriction}}\\
        \cfrak_{k}(E)\cdot\delta_{Y/S}\ar[d]_{\eqref{eq:rest-dist-prel}}  &   &   &\left(\sum_{i=0}^{k} \cfrak_{i}(E^{\prime})\cdot\cfrak_{k-i}(E^{\bis})\right)\cdot\delta_{Y/S}\ar[d]^{\eqref{eq:rest-dist-prel}}\\
        [\cfrak_{k}(E|_{Y})]_{Y/S}\ar[rrr]^{\text{Whitney}}  &   &   &\sum_{i=0}^{k} [\cfrak_{i}(E^{\prime}|_{Y})\cdot\cfrak_{k-i}(E^{\bis}|_{Y})]_{Y/S}.
    }
    }
\end{displaymath}
\normalsize
In the diagram, we have named the arrows according to the effected operations. The claim of the theorem is that the diagram commutes. 

For a thorough treatment of the properties of intersection bundles and intersection distributions we refer to \cite[Section 7 \& Section 8]{DRR1}. The upshot is that, in practice, we can manipulate intersection distributions as in classical intersection theory, without worrying about the order of polynomial operations and isomorphisms derived from Theorem \ref{thm:cor86}.

\subsubsection{Multiplicativity of Chern and Borel--Serre isomorphisms}\label{subsubsec:mult-chern-Borel-Serre}
For completeness, we recall a couple of applications of the theory of intersection distributions, which play a key role in this paper. 

For a scheme $X$, we know that the categorical Chern character $\chfrak\colon V(X)\to\CHfrak(X)_{\QBbb}$ behaves additively with respect to the sum on $V(X)$. However, the structure of $\CHfrak(X)_{\QBbb}$ does not ensure that $\chfrak$ behaves multiplicatively with respect to the tensor product, as in the classical intersection theory. However, in \cite[Proposition 9.1]{DRR1} we prove that the associated line distribution does. Concretely, if $X\to S$ satisfies the condition $(C_{n})$ and $E, F$ are virtual vector bundles on $X$, or even virtual perfect complexes, then there is a canonical isomorphism
\begin{equation}\label{eq:chern-tensor-product}
    [\chfrak(E\otimes F)]_{X/S}\simeq [\chfrak(E)\cdot\chfrak(F)]_{X/S},
\end{equation}
which is functorial in $E$ and $F$ separately, meaning, e.g., that for fixed $F$, this defines an isomorphism of functors (in $E$) of commutative Picard categories. It satisfies additional naturality properties. For instance, there is also a canonical isomorphism
\begin{equation}\label{eq:chern-trivial}
    [\chfrak(\Ocal_{X})]_{X/S}\simeq 1, 
\end{equation}
and \eqref{eq:chern-tensor-product} is compatible with it whenever $E=\Ocal_{X}$ or $F=\Ocal_{X}$. We refer to \emph{op. cit.} for details. 

A variant of \eqref{eq:chern-tensor-product} concerns the first Chern categorical class: if $L$ and $M$ are line bundles on $X$, then there is a canonical isomorphism
\begin{equation}\label{eq:chern-tensor-product-line}
    [\cfrak_{1}(L\otimes M)]_{X/S}\simeq [\cfrak_{1}(L)+\cfrak_{1}(M)]_{X/S}.
\end{equation}
Actually, in the case of line bundles, the isomorphism \eqref{eq:chern-tensor-product} together with the rank triviality isomorphism from Theorem \ref{thm:cor86} \eqref{item:prop-int-dist-rank} specializes to \eqref{eq:chern-tensor-product-line}. The latter amounts to the multilinearity of Deligne pairings with respect to the tensor product.

Another example of application is the Borel--Serre isomorphism from \cite[Theorem 9.5]{DRR1}. It states that, given a vector bundle $E$ on $X$, of constant rank $r$, there is a canonical isomorphism
\begin{equation}\label{eq:Borel-Serre-prelim}
    \sum_{k}(-1)^{k}[\chfrak(\Lambda^{k}E)]_{X/S}\simeq [\cfrak_{r}(E^{\vee})\cdot\tdfrak(E^{\vee})^{-1}]_{X/S}.
\end{equation}
Among other properties, the isomorphism is compatible with exact sequences. 

The constructions of \eqref{eq:chern-tensor-product} and \eqref{eq:Borel-Serre-prelim} both rely on the splitting principles from \cite[Section 5.4]{DRR1} recalled in \textsection \ref{subsubsec:pic-cat-line-dist}, item \eqref{item:splitting-principle}, and on the properties of intersection distributions from Theorem \ref{thm:cor86}. It follows that these isomorphisms commute with the properties stated in that theorem. This was not explicitly stated in \cite{DRR1}, but we will need it and use it in the sequel.

\subsubsection{A K\"unneth-type isomorphism for intersection distributions}
We address a special case of the K\"unneth formula in the framework of line distributions, that we will later need and which is not contained in \cite{DRR1}.

Let $X\to S$ be a morphism that satisfies the condition $(C_n)$. We begin with a Cartesian diagram of morphisms of schemes
\begin{displaymath}
    \xymatrix{
        Z\ar[r]^{q^{\prime}}\ar[d]_{q}\ar[rd]^{f}   &Y\ar[d]^{\pi}\\
        Y^{\prime}\ar[r]^{\pi^{\prime}}  &X.
    }
\end{displaymath}
we assume that the morphism $\pi$ (resp. $\pi^{\prime}$) satisfies the condition $(C_{m})$ (resp. $(C_{m'})$). Given Chern power series $P$ on $Y$, and $P'$ on $Y'$, of pure degrees, we may consider the line distribution
\begin{displaymath}
    f_{\ast}[q^{\prime\ast}P \cdot q^{\ast}P^{\prime}]_{Z/S}.
\end{displaymath}   
Provided that the degrees of $P$ and $P'$ are restricted to the applicability range of the projection formulas in Theorem \ref{thm:cor86}, we can apply the latter twice, factoring $f$ as $q\pi^{\prime}$ or $q^{\prime}\pi$. We will record a lemma to the effect that both procedures yield the same result. 

\begin{lemma}\label{lemma:basechangesmalldegree}
Let the assumptions and notation be as above and as in Theorem \ref{thm:cor86} \eqref{item:cor86-1}. 
\begin{enumerate}
    \item\label{item:basechangesmalldegree-1} Suppose that $\deg P\leq m+1$, $\deg P^{\prime}\leq m'+1$ and $\deg P+\deg P^{\prime}\leq m+m'+1$. Then, the projection formula and the base-change functoriality induce two canonical isomorphisms
\begin{displaymath}
     f_{\ast}[q^{\prime\ast}P \cdot q^{\ast}P^{\prime}]_{Z/S}\simeq [\pi_{\ast}P\cdot \pi^{\prime}_{\ast} P^{\prime}]_{X/S},
\end{displaymath}
and they coincide.
    \item Suppose furthermore that $X=S$. If $\deg P>m+1$ or $\deg P^{\prime}>m^{\prime}+1$, then 
    \begin{displaymath}
     f_{\ast}[q^{\prime\ast}P \cdot q^{\ast}P^{\prime}]_{Z/S}=0.
\end{displaymath}
\end{enumerate}
\end{lemma}

\begin{proof}
For the first item, it suffices to treat the cases $\deg P\leq m+1$ and $\deg P'\leq m'$. For simplicity, we even suppose that these are equalities. By the very construction of intersection bundles \cite[Section 7.1]{DRR1}, the lemma reduces to an analogous statement for Deligne products of line bundles. In this case, consider line bundles $L_{0},\ldots, L_{m}$ on $Y$, $L_{1}^{\prime},\ldots, L_{m^{\prime}}^{\prime}$ on $Y^{\prime}$, and $M_{1},\ldots, M_{n}$ on $X$. Setting  $d=\int_{Y'/X}c_{1}(L^{\prime}_1)\cdots c_{1}(L^{\prime}_{m'})$, combining the symmetry of Deligne pairings, the projection formula \cite[Proposition 6.10]{DRR1} for $q^{\prime}$ and $\pi$ and the base-change functoriality of Deligne pairings, we obtain a chain of canonical isomorphisms, where 
\begin{displaymath}
\begin{split}
    \langle q^{\prime\ast}L_{0},\ldots, q^{\prime\ast} L_{m},q^{\ast}L_{1}^{\prime},\ldots, q^{\ast} L_{m'}^{\prime}, &f^{\ast}M_{1},\ldots, f^{\ast} M_{n}\rangle_{Z/S} \\
    &\simeq \langle L_{0},\ldots, L_{m}, \pi^{\ast}M_{1},\ldots, \pi^{\ast} M_{n}\rangle_{Y/S}^{d}\\
    &\hspace{0.5cm}\simeq\langle \langle L_{0},\ldots, L_{m}\rangle_{Y/X} , M_{1},\ldots, M_{n}\rangle_{X/S}^{d}.
\end{split}
\end{displaymath}
Similarly, we obtain another isomorphism by applying the projection formula for $q$ and $\pi^{\prime}$ instead. To verify that both isomorphisms coincide, one proceeds analogously to the proof of \cite[Corollary 6.12]{DRR1}. That is, since the argument is \'etale local over $S$, and Deligne pairings are multilinear with respect to the tensor product, one can suppose that all the involved line bundles are relatively very ample over $S$. One can then check the claim by introducing symbols that generate the Deligne pairings, associated with sections of the line bundles whose divisors are in general position, thanks to \cite[Section 6.1.3]{DRR1}. This allows us to reduce to the case where all the maps are of relative degree 0. We then have to compare the following canonical isomorphisms involving norms of line bundles:

\begin{displaymath}
     N_{Z/X}(q^{\prime\ast} L_{0}) \simeq N_{{Y'}/X }N_{Z/Y'}(q^{\prime\ast}L_{0}) = N_{Y'/X}( \pi^{\prime \ast } N_{Y/X}(L_0)) \simeq N_{Y/X} (L_0)^d
\end{displaymath}
and
\begin{displaymath}
     N_{Z/X}(q^{\prime\ast} L_{0}) \simeq N_{Y/X }N_{Z/Y}(q^{\prime\ast}L_{0}) \simeq N_{Y/X}  (L_0^d)\simeq N_{Y/X}(L_{0})^{d}.
\end{displaymath}
One verifies locally, with symbols, that these two isomorphisms coincide.

For the second item, since $X=S$, taking $f_{\ast}$ amounts to considering the Deligne pairing for $Z\to S$. Then, if $\deg P+\deg P^{\prime}\neq m+m'+1$, the result is trivial for degree reasons. We may thus assume that we have an equality. Assume, without loss of generality, that $\deg P>m+1$. Then, $\deg P^{\prime}\leq m'-1$ and by the last case in the projection formula from Theorem \ref{thm:cor86} \eqref{item:cor86-1}, we find
\begin{displaymath}
    \langle q^{\prime\ast}P\cdot q^{\ast}P\rangle_{Z/S}=(q^{\prime}_{\ast}[q^{\ast}P^{\prime}]_{Z/S})(P)\simeq\Ocal_{S},
\end{displaymath}
since the relative dimension of $q^{\prime}$ is $m^{\prime}$ and $q^{\ast}P^{\prime}$ has degree strictly smaller than $m^{\prime}$.
\end{proof}

By the lemma, if $X=S$, then we may define $\pi_{\ast}P\cdot \pi_{\ast}^{\prime}P^{\prime}$ to be 0 if $\deg P>m+1$ or $\deg P^{\prime}>m^{\prime}+1$, and in this case we have a general K\"unneth formula of the form
\begin{equation}\label{eq:general-Kunneth-S}
    f_{\ast}[q^{\prime\ast}P \cdot q^{\ast}P^{\prime}]_{Z/S}\simeq \pi_{\ast}P\cdot \pi^{\prime}_{\ast} P^{\prime},
\end{equation}
where the right-hand side is now to be interpreted as a line bundle on $S$, possibly trivial. 

\subsection{Functoriality and base-change compatibility}\label{subsubsec:base-change-conventions}
In this article, we will systematically deal with functors and isomorphisms which exhibit some base-change compatibility. This is best formulated in the language of categories fibered in groupoids. In order to simplify related discussions, we now fix some conventions on this topic. For concreteness, below we argue with arbitrary base schemes, but analogous results can be formulated with quasi-compact or Noetherian base schemes. 

\subsubsection{Functors and isomorphisms compatible with base change}\label{subsubsec:prelim-on-base-change}
Let $X\to S$ and $Y\to S$ satisfy the conditions $(C_{n})$ and $(C_{m})$, respectively. Associating, to a morphism $S^{\prime}\to S$, the categories $V(\Pcal_{X^{\prime}})$, $V(X^{\prime})$, $\Dcal(Y^{\prime}/S^{\prime})$ and $\Dcal_{\ZBbb}(Y^{\prime}/S^{\prime})$ for $X^{\prime}=X\times_{S}S^{\prime}$ and $Y^{\prime}=Y\times_{S}S^{\prime}$, together with pullback functoriality, naturally defines categories fibered in groupoids over $(\mathrm{Sch}/S)$, cf. \textsection\ref{subsubsec:functorial-properties-virtual-prelim} and \textsection\ref{subsubsec:line-dist-prel}. Here and below we implicitly rely on the natural equivalence between the notion of category fibered in groupoids over $(\mathrm{Sch}/S)$ and pseudo-functor from $(\mathrm{Sch}/S)$ to the 2-category of groupoids, cf. \cite[Expos\'e VI, Sections 7\& 8]{SGA1} and \cite[\href{https://stacks.math.columbia.edu/tag/0048}{0048}, \href{https://stacks.math.columbia.edu/tag/0GWH}{0GWH}]{stacks-project}. By abuse of notation, we will simply denote these fibered categories by $V(\Pcal_{X})$, $V(X)$, $\Dcal(Y/S)$ and $\Dcal_{\ZBbb}(Y/S)$, but we will specify these are considered as being fibered over $(\mathrm{Sch}/S)$. 

\begin{definition}[Functor compatible with base change]\label{def:functor-compat-base-change}
A functor of commutative Picard categories $F\colon V(\Pcal_{X})\to\Dcal(Y/S)$ (resp. $F\colon V(X)\to\Dcal(Y/S)$) compatible with base change, consists in giving a functor of commutative Picard categories fibered over $(\mathrm{Sch}/S)$. Equivalently:
\begin{enumerate}
    \item\label{item:functor-compat-base-change-1} For every morphism of schemes $q\colon S^{\prime}\to S$, we are given a functor $F_{S^{\prime}}\colon V(\Pcal_{X^{\prime}})\to\Dcal(Y^{\prime}/S^{\prime})$ of commutative Picard categories.
    \item\label{item:functor-compat-base-change-2} For every $q^{\prime}\colon S^{\bis}\to S^{\prime}$, we are given an isomorphism of functors of commutative Picard categories $\mu_{q^{\prime}}\colon q^{\prime\ast}\circ F_{S^{\prime}}\simeq F_{S^{\bis}}\circ q^{\prime\ast}$, compatible with further base change $q^{\bis}\colon S^{\tris}\to S^{\bis}$. Thus, via the natural identification $(q^{\prime}\circ q^{\bis})^{\ast}\simeq q^{\bis\ast}\circ q^{\prime\ast}$, the composition $(\mu_{q^{\bis}}\circ q^{\prime\ast})\circ (q^{\bis\ast}\circ\mu_{q^{\prime}})$ corresponds to $\mu_{q^{\prime}\circ q^{\bis}}$.
\end{enumerate}
Furthermore, we say that $F$ has bounded denominators if there exists an integer $N\geq 1$ such that $F^{N}$ induces a functor of commutative Picard categories $F\colon V(\Pcal_{X})\to\Dcal_{\ZBbb}(Y/S)$ (resp. $F\colon V(X)\to\Dcal_{\ZBbb}(Y/S)$) compatible with base change in the sense above. 
\end{definition}

We note that a functor $F\colon V(\Pcal_{X})\to \Dcal(Y/S)$ compatible with base change trivially induces a functor $F^{\prime}\colon V(\Pcal_{X^{\prime}})\to \Dcal(Y^{\prime}/S^{\prime})$ compatible with base change, for every morphism $q\colon S^{\prime}\to S$, and similarly for $V(X)$. We call this the functor deduced from $F$ by base change, and we shall tacitly denote it by $q^{\ast}F$. Prominent examples of functors compatible with base change are the Riemann--Roch distributions introduced in \cite[Section 9.4]{DRR1}, and recalled below in \textsection\ref{subsec:RR-distributions}. For the sake of clarity, we also note that the notion of a functor compatible with base change is a particular case of the usual notion of a Cartesian functor between categories fibered in groupoids.

\begin{definition}[Isomorphism of functors compatible with base change]\label{def:isom-functor-compat-base-change}
Let $F,G\colon V(\Pcal_{X})\to \Dcal(Y/S)$ (resp. $F,G\colon V(X)\to \Dcal(Y/S)$) be two functors of commutative Picard categories compatible with base change. An isomorphism of functors $\mu\colon F\to G$ compatible with base change is an isomorphism of functors of commutative Picard categories fibered over $(\mathrm{Sch}/S)$. Equivalently:
\begin{enumerate}
    \item For every morphism of schemes $q\colon S^{\prime}\to S$, we are given an isomorphism of functors of commutative Picard categories $\nu_{S^{\prime}}\colon F_{S^{\prime}}\to G_{S^{\prime}}$.
    \item For every $q^{\prime}\colon S^{\bis}\to S^{\prime}$, $q^{\prime\ast}\circ \nu_{S^{\prime}}$ corresponds to $\nu_{S^{\bis}}\circ q^{\prime\ast}$ via the base-change identifications $ q^{\prime\ast}\circ F_{S^{\prime}}\simeq F_{S^{\bis}}\circ q^{\prime\ast}$ and $q^{\prime\ast}\circ G_{S^{\prime}}\simeq G_{S^{\bis}}\circ q^{\prime\ast}$ as in Definition \ref{def:functor-compat-base-change} \eqref{item:functor-compat-base-change-2}.
\end{enumerate}
If $F$ and $G$ have bounded denominators, we say that $\nu$ has bounded denominators if there exists an integer $N\geq 1$ such that $\nu^{N}\colon F^{N}\to G^{N}$ is an isomorphism of functors of commutative Picard categories compatible with base change in the sense above.
\end{definition}
We note that an isomorphism of functors $\mu\colon F\to G$ compatible with base change trivially induces an isomorphism compatible with base change between the functors deduced from $F$ and $G$ by any base change. The central examples of isomorphisms of functors compatible with base change are the Deligne--Riemann--Roch isomorphisms introduced in \textsection\ref{subsec:DRR-iso-expected-functor} below, and whose existence is the goal of this work. 

\subsubsection{Gluing functors and isomorphisms compatible with base change}\label{subsubsec:gluing-functors}
The functors of commutative Picard categories, compatible with base change and with bounded denominators, together with the isomorphisms of such, define a category $\Fcal(V(\Pcal_{X}),\Dcal(Y/S))$, which is naturally fibered in groupoids over $(\mathrm{Sch}/S)$. There is a variant of this construction, where we restrict to affine base schemes. Namely, in Definition \ref{def:functor-compat-base-change} and Definition \ref{def:isom-functor-compat-base-change}, the schemes $S^{\prime}$, $S^{\bis}$, $S^{\tris}$ are all required to be affine. We obtain a corresponding category of functors $\Fcal_{\mathrm{Aff}}(V(\Pcal_{X}),\Dcal(Y/S))$. We can also form the analogous categories for $V(X)$ in place of $V(\Pcal_{X})$. We then have a natural commutative diagram of restriction functors
\begin{equation}\label{eq:diag-restr-functors}
    \xymatrix{
        \Fcal(V(\Pcal_{X}),\Dcal(Y/S))\ar[r]\ar[d] &\Fcal(V(X),\Dcal(Y/S))\ar[d]\\
        \Fcal_{\mathrm{Aff}}(V(\Pcal_{X}),\Dcal(Y/S))\ar[r] &\Fcal_{\mathrm{Aff}}(V(X),\Dcal(Y/S)).
    }
\end{equation}
In this diagram, the horizontal arrows are induced by the natural functor $V(X)\to V(\Pcal_{X})$. The vertical arrows are induced by the inclusion of the category of affine $S$-schemes into the category of $S$-schemes. 

The following proposition states that to construct functors and isomorphisms compatible with base change, we may assume that all the base schemes are divisorial, even affine, and we can work with virtual categories of vector bundles instead of perfect complexes. 

\begin{proposition}\label{prop:base-reduction}
With the assumptions and notation as above, the functors in \eqref{eq:diag-restr-functors} define equivalences of categories. 
\end{proposition}

\begin{proof}
This is a variant of \cite[Proposition 5.31]{DRR1}, to the effect that line distributions are determined over affine base schemes. Proceeding by the gluing method as in \emph{loc. cit.}, one shows that the vertical functors in \eqref{eq:diag-restr-functors} induce equivalences of categories. 

To conclude, it is enough to show that the lower horizontal arrow induces an equivalence of categories. For this, we first note that if $S^{\prime}$ is affine, then it is a divisorial scheme, and hence so is $X^{\prime}=X\times_{S}S^{\prime}$, because $X\to S$ satisfies the condition $(C_{n})$, cf. \textsection \ref{sec:notations-conventions} above. Therefore, the functor $V(X^{\prime})\to V(\Pcal_{X^{\prime}})$ defines an equivalence of categories. We can construct an inverse $G_{S^{\prime}}\colon V(\Pcal_{X^{\prime}})\to V(X^{\prime})$, in a way compatible with base change by affine schemes $S^{\bis}\to S^{\prime}$, see \cite[\href{https://stacks.math.columbia.edu/tag/003Z}{003Z}]{stacks-project}. Composing with this inverse, we obtain an inverse functor of the lower horizontal arrow in \eqref{eq:diag-restr-functors}.
\end{proof}

\subsubsection{Some simplifications} In the sequel, instances of the previous formalism will be derived from the theory of intersection distributions. In this setting, the compatibility with base change will follow from the implicit base-change compatibility of intersection distributions. This in turn relies on the base-change compatibility of intersection bundles. To lighten the presentation, we will specify our constructions on a fixed base scheme $S$, and usually omit details on the base changes.  

\section{Deligne--Riemann--Roch isomorphisms}\label{section:DRR-isos}
In this section, all the $S$-schemes are assumed to satisfy the condition $(C_{n})$, for some $n$, except possibly for base changes $S^{\prime}\to S$. The goal of this section is to formulate the various properties one might expect from a functorial Riemann--Roch theorem. These are encapsulated in the notion of Deligne--Riemann--Roch isomorphism, abbreviated as DRR isomorphism. In this section, we follow the framework and terminology from \textsection\ref{subsubsec:base-change-conventions} regarding base-change compatibility.

\subsection{The Riemann--Roch distributions}\label{subsec:RR-distributions}
Let $f\colon X\to Y$ be a local complete intersection morphism of $S$-schemes, and $E$ a virtual perfect complex on $X$. In \cite[Section 9]{DRR1} we constructed line distributions on $Y$
\begin{equation}\label{eq:LHS-RR}
    [\chfrak(f_{!}E)]_{Y/S}
\end{equation}
and
\begin{equation}\label{eq:RHS-RR}
 f_{\ast}[\chfrak(E)\cdot\tdfrak(T_{f})]_{X/S},
\end{equation}
where $T_{f}$ stands for the tangent complex of $f$.\footnote{In loc. cit., we equivalently deal with the cotangent complex, which is dual to the tangent complex.} We call them the Riemann--Roch distributions.\footnote{In \cite[Section 9]{DRR1}, we reserved this terminology to the right-hand side distribution \eqref{eq:RHS-RR}.} We distinguish them simply as "left-hand side" or "right-hand side", according to their place in the statement of the expected DRR isomorphism. In this subsection, we briefly recall how the construction proceeds. 

By Proposition \ref{prop:base-reduction}, we can restrict to the category of divisorial schemes and suppose that $f$ is projective. In this case, we can also suppose that $E$ and $f_{!}E$ are virtual vector bundles. The Riemann--Roch distributions are then defined as follows:
\begin{enumerate}
 \item For \eqref{eq:LHS-RR}, since $f_{!}E$ is now a virtual vector bundle, the objects $\chfrak(f_{!}E)$ and $[\chfrak(f_{!}E)]_{Y/S}$ are defined. This construction is compatible with base change, and we can thus work with the convention in \cite[Section 9.4]{DRR1}, to the effect that for a morphism $q: S^{\prime} \to S$, with the base change of $X\to S$ along $f$ denoted by  ${f}^\prime: X^{\prime} \to S^{\prime}$, we define
 \begin{displaymath}
    q^\ast [\chfrak(f_! E)]_{Y/S}=[\chfrak({f}^\prime_!p^{\ast}E)]_{Y^{\prime}/S^{\prime}},
 \end{displaymath}
  where $p\colon X^{\prime}\to X$ is the projection map. Here we implicitly employ the base-change isomorphism $q^\ast f_! \simeq {f}^\prime_! p^\ast $ from \eqref{eq:virtualbasechange}, which holds by Tor-independence, due to the flatness assumption on $X\to S$ and $Y\to S$.

 \item For \eqref{eq:RHS-RR}, we consider a projective factorization $X\to\PBbb(\Ecal)\to Y$, for some vector bundle $\Ecal$ on $Y$ of constant rank. One can then realize the cotangent complex, and hence also  $T_{f}$, as a 2-term complex of vector bundles. For this complex, $\tdfrak(T_{f})$, and hence $f_{\ast}[\chfrak(E)\cdot\tdfrak(T_{f})]_{X/S}$, are defined. Two different global factorizations can be coherently compared. The construction commutes with base change over $S$, and we can thus work with the convention that for a base change along $q\colon S^{\prime}\to S$, we define
 \begin{displaymath}
    q^{\ast}f_{\ast}[\chfrak(E)\cdot\tdfrak(T_{f})]_{X/S}=f_{\ast}^{\prime}[\chfrak(p^{\ast}E)\cdot\tdfrak(T_{f^{\prime}})]_{X^{\prime}/S^{\prime}},
 \end{displaymath}
 with the notation as in the previous point. Note that this construction actually produces a line distribution
 \begin{equation}\label{eq:pre-RHS-RR}
    [\chfrak(E)\cdot\tdfrak(T_{f})]_{X/S},
 \end{equation}
from which the right-hand side Riemann--Roch distribution is obtained by taking $f_{\ast}$.
\end{enumerate}

The Riemann--Roch distributions above define functors of commutative Picard categories $V(\Pcal_{X})\to\Dcal(Y/S)$ compatible with base change over $S$. They have bounded denominators, which depend on the denominators of the power series defining $\chfrak$ and $\tdfrak$ up to a finite order, bounded by the relative dimensions of $X\to S$ or $Y\to S$ plus one. If $f$ is the identity map, we note that the Riemann--Roch distributions amount simply to $E\mapsto [\chfrak(E)]_{X/S}$. 

The Riemann--Roch distributions are compatible with the projection formula, in the following sense:
\begin{proposition}[Projection formulas for the Riemann--Roch distributions]\label{prop:proj-for-RR-dist}
Let the assumptions and notation be as above. Let $F$ be a virtual vector bundle on $Y$. Then:

\begin{enumerate}
    \item\label{item:proj-for-RR-dist-1} There is a canonical isomorphism
    \begin{displaymath}
        [\chfrak(f_{!}(E\otimes f^{\ast}F))]_{Y/S}\simeq \chfrak(F)\cdot [\chfrak(f_{!}E)]_{Y/S}
    \end{displaymath}
    of line distributions on $Y$. 
    \item\label{item:proj-for-RR-dist-2} There is a canonical isomorphism
    \begin{displaymath}
        f_{\ast}[\chfrak(E\otimes f^{\ast}F)\cdot\tdfrak(T_{f})]_{X/S}\simeq \chfrak(F)\cdot f_{\ast}[\chfrak(E)\cdot\tdfrak(T_{f})]_{X/S}
    \end{displaymath}
    of line distributions. 
\end{enumerate}
Moreover, these induce isomorphisms of functors in $E$ (resp. $F$) of commutative Picard categories $V(\Pcal_{X})\to\Dcal(Y/S)$ (resp. $V(\Pcal_{Y})\to\Dcal(Y/S)$), compatible with base change and with bounded denominators, and they are compatible with the composition of morphisms.
\end{proposition}
\begin{proof}
By Proposition \ref{prop:base-reduction}, it is enough to construct the isomorphisms in the category of divisorial schemes. In this case, the first isomorphism is obtained by composing: (1) the projection formula at the virtual category level and the commutativity of the tensor product, namely the functorial (in $E$ and $F$) isomorphism
\begin{displaymath}
    f_{!}(E\otimes f^{\ast}F)\simeq (f_{!}E)\otimes F\simeq F\otimes f_{!}E;
\end{displaymath}
and (2) the isomorphism from \cite[Proposition 9.1]{DRR1}, recalled above in \eqref{eq:chern-tensor-product}, which gives
\begin{displaymath}
    [\chfrak(F\otimes G)]_{Y/S}\simeq [\chfrak(F)\cdot \chfrak(G)]_{Y/S},
\end{displaymath}
functorial in $F$ and $G$. 

The second isomorphism of the statement is obtained in a similar fashion, this time applying the projection formula for line distributions \eqref{eq:proj-for-int-dist}. 

For the last assertion, the functoriality and base-change property are clear from the construction, and the fact that $f_{!}$ and $f^{\ast}$ commute with any base change $S^{\prime}\to S$. For the compatibility with the composition of morphisms in \eqref{item:proj-for-RR-dist-1}, this follows from the analogous property for the pseudofunctors $f\mapsto f_{!}$ and $f\mapsto f^{\ast}$. In the case \eqref{item:proj-for-RR-dist-2}, this follows from the projection formula for intersection distributions \eqref{eq:proj-for-int-dist} and \cite[Lemma 9.2]{DRR1}, which is a Whitney-type isomorphism of the form $[\tdfrak(T_{gf})]\simeq [\tdfrak(T_{f})\cdot f^{\ast}\tdfrak(T_{g})]$.

\end{proof}
We note that, in the statement above, the restriction that $F$ be a virtual vector bundle on $Y$, rather than a virtual perfect complex, is purely a matter of notation. Indeed, the Chern power series $\chfrak(F)$ is only defined for an object $F$ of $V(Y)$, hence the notation $\chfrak(F)\cdot[\chfrak(f_{!}E)]_{Y/S}$ makes sense only in this case. However, the line distribution $[\chfrak(F)\cdot\chfrak(f_{!}E)]_{Y/S}$ is defined for any object $F$ of $V(\Pcal_{Y})$. A similar observation is valid for the right-hand side Riemann--Roch distribution. 

The Riemann--Roch distributions also depend on the morphism $f$ in a functorial manner. Concretely, we consider an isomorphism $f^{\prime}\simeq f$ between morphisms, meaning a diagram of the form
\begin{equation}\label{eq:isomor-of-mor}
    \xymatrix{
        X^{\prime}\ar[r]^{f^{\prime}}\ar[d]_{\varphi}^{\sim}  &Y^{\prime}\ar[d]^{\psi}_{\sim}\\
        X\ar[r]^{f} &Y,
    }
\end{equation}
where the vertical arrows are isomorphisms of $S$-schemes. Then, the Riemann--Roch distributions for $f$ and $f^{\prime}$ compare according to the following rules.
\begin{proposition}\label{prop:isomor-of-mor}
Let the setting be as in \eqref{eq:isomor-of-mor}. Then:
\begin{enumerate}
    \item There are natural isomorphisms
    \begin{equation}\label{eq:isomor-of-mor-RR-left}
    \psi_{\ast}[\chfrak(f_{!}^{\prime}\varphi^{\ast} E)]_{Y^{\prime}/S}\to [\chfrak(f_{!}E)]_{Y/S}
    \end{equation}
and
\begin{equation}\label{eq:isomor-of-mor-RR-right}
    \psi_{\ast} f_{\ast}^{\prime}[\chfrak(\varphi^{\ast}E)\cdot\tdfrak(T_{f^{\prime}})]_{X^{\prime}/S}\to f_{\ast}[\chfrak(E)\cdot\tdfrak(T_{f})]_{X/S},
\end{equation}
which induce isomorphisms of functors of commutative Picard categories $V(\Pcal_{X})\to\Dcal(Y/S)$, compatible with base change and with bounded denominators.
    \item The formation of \eqref{eq:isomor-of-mor-RR-left}--\eqref{eq:isomor-of-mor-RR-right} is compatible with compositions of isomorphisms $f^{\bis}\simeq f^{\prime}\simeq f$ between morphisms. 
    \item The formation of \eqref{eq:isomor-of-mor-RR-left}--\eqref{eq:isomor-of-mor-RR-right} is compatible with the projection formulas from Proposition \ref{prop:proj-for-RR-dist}.
\end{enumerate}
\end{proposition}
\begin{proof}
We discuss the case of \eqref{eq:isomor-of-mor-RR-left}, and omit the details of  \eqref{eq:isomor-of-mor-RR-right}, which are similar. 

By Proposition \ref{prop:base-reduction}, we may assume that the base schemes are divisorial, that $f$ is projective, and we can work with virtual categories of vector bundles. In particular, we may assume that $E$ and $f_{!}E$ are virtual vector bundles. Next, we write $f^{\prime}=\psi^{-1}\circ f\circ\varphi$. Then, we have $f^{\prime}_{!}\simeq \psi^{-1}_{!}\circ f_{!}\circ\varphi_{!}.$ Moreover, since $\varphi$ and $\psi$ are isomorphisms, there are  natural isomorphisms of functors $\varphi_{!}\simeq (\varphi^{\ast})^{-1}$ and $\psi^{-1}_{!}\simeq \psi^{\ast}$.  Therefore, we can write the natural isomorphisms
\begin{equation}\label{eq:isomor-of-mor-RR-left-1}
     \psi_{\ast}[\chfrak(f_{!}^{\prime}\varphi^{\ast} E)]_{Y^{\prime}/S}\simeq \psi_{\ast}[\chfrak(\psi^{\ast}f_{!} (\varphi^{\ast})^{-1}\varphi^{\ast} E)]_{Y^{\prime}/S}\simeq \psi_{\ast}[\chfrak(\psi^{\ast} f_{!}E)]_{Y^{\prime}/S}.
\end{equation}
Now, since $f_{!}E$ is a virtual vector bundle, we can safely write $\chfrak(\psi^{\ast} f_{!}E)\simeq \psi^{\ast}\chfrak(f_{!}E)$, and then we can apply the projection formula from Theorem \ref{thm:cor86} \eqref{item:prop-int-dist-birational} to conclude
\begin{equation}\label{eq:isomor-of-mor-RR-left-2}
    \psi_{\ast}[\chfrak(\psi^{\ast} f_{!}E)]_{Y^{\prime}/S}\simeq \psi_{*}[\psi^{\ast}\chfrak(f_{!}E)]_{Y'/S}\simeq [\chfrak(f_{!}E)]_{Y/S}.
\end{equation}
The sought isomorphism is obtained as the composition of \eqref{eq:isomor-of-mor-RR-left-1}--\eqref{eq:isomor-of-mor-RR-left-2}, which all induce isomorphisms of commutative Picard categories, which commute with base change and have bounded denominators depending only on the denominators of $\chfrak$ up to a fixed degree, given by the relative dimension of $Y\to S$ plus one. The compatibility with further composition similarly follows from the pullback and pushforward functoriality of virtual categories and intersection distributions, and analogously for the projection formula. We omit the details.
\end{proof}

\subsection{Deligne--Riemann--Roch isomorphisms and expected functorialities}\label{subsec:DRR-iso-expected-functor}

In this subsection we introduce and elaborate on the notion of a Deligne--Riemann--Roch (DRR) isomorphism between the Riemann--Roch distributions, which includes various compatibilities and key equalities from the proof of the classical Grothendieck--Riemann--Roch theorem. This in particular gives a detailed meaning to the compatibilities in Theorem \ref{thm:general-DRR} below and in particular our Theorem \ref{thm:A}. Recall that we follow the terminology from \textsection\ref{subsubsec:base-change-conventions} regarding functors and isomorphisms compatible with base change.

\begin{definition}[DRR isomorphism for a morphism $f$]
Let $f\colon X\to Y$ be an lci morphism of $S$-schemes. A DRR isomorphism $\RR_f$ for $f$ consists in giving, for every virtual perfect complex $E$ on $X$, an isomorphism of line distributions
\begin{displaymath}
    \RR_{f}(E)\colon [\chfrak(f_{!} E)]_{Y/S}\to f_{\ast}[\chfrak(E)\cdot\tdfrak(T_{f})]_{X/S},
\end{displaymath}
such that the association $E\mapsto \RR_{f}(E)$ induces an isomorphism of functors of commutative Picard categories $V(\Pcal_{X})\to\Dcal(Y/S)$, compatible with base change and with bounded denominators.
\end{definition}Whenever we impose additional conditions on a DRR isomorphism, we will tacitly require the compatibility with base change. 

We next address the precise definition of compatibility with the projection formula: 

\begin{definition}[Compatibility with the projection formula]\label{def:compatibilitywithproj}
A DRR isomorphism $\RR_{f}$ for a lci morphism $f\colon X\to Y$ of $S$-schemes is said to be compatible with the projection formula if, for every virtual perfect complex $E$ on $X$ and every virtual vector bundle $F$ on $Y$, the diagram
    \begin{displaymath}
        \xymatrix{
            [\chfrak(f_{!}(E\otimes f^{\ast}F))]_{Y/S}\ar[rr]^-{\RR_{f}(E\otimes f^{\ast}F)}\ar[d]    &    &f_{\ast}[\chfrak(E\otimes f^{\ast}F)\cdot\tdfrak(T_{f})]_{X/S}\ar[d]\\
             \chfrak(F)\cdot[\chfrak(f_{!}E)]_{Y/S}\ar[rr]^-{\chfrak(F)\cdot\RR_{f}(E)}      &     &\chfrak(F)\cdot f_{\ast}[\chfrak(E)\cdot\tdfrak(T_{f})]_{X/S}
        }
    \end{displaymath}
    commutes. Here, the vertical arrows are the projection formulas from Proposition \ref{prop:proj-for-RR-dist}. 
\end{definition}

We also address the precise meaning of compositions of DRR isomorphisms. 

\begin{definition}[Composition of DRR isomorphisms]\label{def:compositionofmorphisms}
    Suppose that $X\stackrel{f}{\to} Y \stackrel{g}{\to} Z$ is a composition of lci morphisms of $S$-schemes, which admit DRR isomorphisms $\RR_{f}$ and $\RR_{g}$. We define the composition of $\RR_{f}$ and $\RR_{g}$ as follows. Suppose first that the base schemes are divisorial and the morphisms are projective. Then, we consider the composition
\begin{equation}
    \begin{split}
        (\RR_{f}\RR_{g})(E)\ \colon\ [\chfrak((gf)_{!}E)]  \stackrel{\eqref{eq:fun-composition-virtual}}{\simeq} [\chfrak(g_! f_{!} E)] &\stackrel{\RR_g}{\simeq} g_{\ast}[\chfrak(f_! E) \cdot \tdfrak(T_{g})]  \\ 
       & \hspace{10pt} \stackrel{\RR_{f}}{\simeq}  g_{\ast} \left(f_\ast [\chfrak(E) \cdot \tdfrak(T_{f})] \cdot \tdfrak(T_{g})\right) \\
       &\hspace{20pt}\stackrel{\eqref{eq:proj-for-int-dist}}{\simeq} g_\ast f_\ast [\chfrak(E) \cdot \tdfrak(T_{f}) \cdot \tdfrak({f}^\ast T_{g})] \\
       &\hspace{30pt}\stackrel{\mathrm{Whitney}}{\simeq} (gf)_\ast [\chfrak(E) \cdot \tdfrak(T_{gf})],
    \end{split}
\end{equation}
where the last isomorphism is induced by the Whitney isomorphism $[\tdfrak(T_{gf})]\simeq [\tdfrak(T_{f})\cdot\tdfrak(f^{\ast}T_{g})]$ from \cite[Lemma 9.2 (2)]{DRR1}. This defines an isomorphism of functors of commutative Picard categories compatible with base change and with bounded denominators. In general, the construction is extended by means of Proposition \ref{prop:base-reduction}.
\end{definition}
A comment on the definition is in order. The reduction to divisorial schemes and projective morphisms is needed to ensure that $T_{g}$ is a virtual vector bundle rather than just a virtual perfect complex, and hence the Chern power series $\tdfrak(T_{g})$ can be defined. This is necessary in the application of the projection formula \eqref{eq:proj-for-int-dist} (see the last paragraph in \textsection\ref{subsubsec:categorical-characteristic}). However, as explained after Proposition \ref{prop:proj-for-RR-dist}, this is just a matter of notation. 

The following lemma shows that the composition of DRR isomorphisms is associative, which thus allows us to perform repeated compositions without specifying the choices of parentheses.

\begin{lemma}\label{lemma:RRassociative}
    The composition of DRR isomorphisms is associative. More precisely, suppose that $X\stackrel{f}{\to} Y \stackrel{g}{\to} Z \stackrel{h}{\to} W$ is a composition of morphisms all admitting a DRR isomorphism. Then, 
    \begin{equation}\label{eq:comp-rr-tr}
        \RR_f (\RR_g \RR_h) = (\RR_f \RR_g) \RR_h.
    \end{equation}
\end{lemma}
\begin{proof}
Without loss of generality, we can assume that the morphisms are projective and that the tangent complexes are finite complexes of vector bundles. 

The comparison of the two sides in \eqref{eq:comp-rr-tr} is straightforward as the very definitions involve the same three DRR isomorphisms, except for the commutativity of the following diagram of Whitney isomorphisms :
\begin{displaymath}
    \xymatrix{
        [\tdfrak(T_{hgf})] \ar[r] \ar[d] & \ar[d] [\tdfrak(T_f) \cdot \tdfrak(f^\ast T_{hg})] \ar[d] \\ 
        [\tdfrak(T_f) \cdot \tdfrak(T_{gf}) \cdot \tdfrak((gf)^\ast T_{h}) ]\ar[r] & [\tdfrak(T_f) \cdot \tdfrak(f^\ast T_{g}) \cdot \tdfrak((gf)^\ast T_{h})].
    }
\end{displaymath}
The commutativity is in turn implied by the existence of the following commutative diagram of triangles of tangent complexes, taking into account \cite[Lemme 4.8]{Deligne-determinant} which provides a comparison of determinant functors of exact sequences of exact sequences:
    \begin{displaymath}
        \xymatrix{T_f \ar[r] \ar@{=}[d] & T_{gf} \ar[d] \ar[r] & f^\ast T_g \ar[d] \\ 
        T_f \ar[r] & T_{hgf} \ar[d] \ar[r] & f^\ast T_{hg} \ar[d] \\ 
          & (gf)^\ast T_h \ar@{=}[r] & (gf)^\ast T_h. }
    \end{displaymath}
Although this is a priori a diagram of triangles, it can be realized as a diagram of true triangles using the same staircase of factorizations as in the proof of \cite[Lemma 9.2]{DRR1}, so that \cite[Lemme 4.8]{Deligne-determinant} can be applied. 
\end{proof}

The last compatibility that we consider amounts to the functoriality with respect to the morphism.
\begin{definition}\label{def:compatibilitymorphisms}
Given lci morphisms $f\colon X\to Y$ and $f^{\prime}\colon X^{\prime}\to Y^{\prime}$ of $S$-schemes and an isomorphism $f^{\prime}\simeq f$ as in \eqref{eq:isomor-of-mor}, we say that $\RR_{f}$ and $\RR_{f^{\prime}}$ are compatible if the diagram
    \begin{displaymath}
        \xymatrix{
            \psi_{\ast}[\chfrak(f_{!}^{\prime}\varphi^{\ast} E)]_{Y^{\prime}/S}\ar[rr]^-{\psi_{\ast}\RR_{f^{\prime}}(\varphi^{\ast}E)}\ar[d]_{\eqref{eq:isomor-of-mor-RR-left}}  &   &\psi_{\ast} f_{\ast}^{\prime}[\chfrak(\varphi^{\ast}E)\cdot\tdfrak(T_{f^{\prime}})]_{X^{\prime}/S}\ar[d]^{\eqref{eq:isomor-of-mor-RR-right}}\\
            [\chfrak(f_{!}E)]_{Y/S}\ar[rr]^-{\RR_{f}(E)}   &  & f_{\ast}[\chfrak(E)\cdot\tdfrak(T_{f})]_{X/S}.
        }
    \end{displaymath}
    commutes for every $E$.
\end{definition}

We are now in a position to introduce the main object of interest in this article. 
\begin{definition}\label{def:DRRisoforclassofmorphisms}
A DRR isomorphism for lci morphisms of $S$-schemes consists in associating, to every lci morphism $f\colon X\to Y$ between $S$-schemes, a DRR isomorphism $\RR_{f}$, satisfying:
\begin{enumerate}
    \item If $f=\id_{X}$, then $\RR_{f}$ is the identity isomorphism. 
     \item Compatibility with the projection formula (cf. Definition \ref{def:compatibilitywithproj}).
    \item Compatibility with the composition of morphisms (cf. Definition \ref{def:compositionofmorphisms}). 
    \item Functoriality in the morphism (cf. Definition \ref{def:compatibilitymorphisms}).
\end{enumerate}
We follow a similar terminology for a class of lci morphisms which is closed under composition and base change, e.g. regular closed immersions.
\end{definition}

In the rest of the article, we establish the existence of a canonical DRR isomorphism (cf. Theorem \ref{thm:general-DRR}) and we deduce Theorem \ref{thm:A} from it. 

\section{Normal cones and deformations of complexes}\label{sec:def-normal-cone}

In this section, we revisit and refine several constructions associated with the deformation to the normal cone, including the behavior of resolutions of direct images of vector bundles under deformations. Originally conceived by MacPherson under the name "Grassmannian graph construction", this framework played a central role in the Riemann–Roch setting developed by Baum, Fulton, and MacPherson \cite{BFM}.
Below, all the $S$-schemes are assumed to satisfy the condition $(C_{n})$, for some $n$, except possibly for base changes $S^{\prime}\to S$. 

\subsection{Deformation to the normal cone}

Consider a morphism of schemes $f\colon X\to S$ of relative dimension $n$. We suppose that we are given a closed subscheme $Y$ of $X$, such that the map $Y\to S$ satisfies the condition $(C_{n-c})$ and the inclusion $i\colon Y\to X$ is a regular closed immersion of codimension $c$ (cf. Section \ref{sec:notations-conventions}). We denote the normal bundle by $N_{Y/X}$. We may also write $N_{i}$ or even $N$ if there is no danger of confusion.  

Recall that the deformation to the normal cone refers to the process of blowing up $\PBbb^1_X$ along the closed subscheme $Y \times \{\infty \}$, denoted by  
\begin{displaymath}
    \pi : M = \Bl_{Y\times\{\infty\}} (\PBbb^1_X ) \to \PBbb^1_X.
\end{displaymath}
By functoriality, there is a natural diagram
\begin{equation}\label{eq:deformationtonormalcone}
    \xymatrix{ \PBbb^1_Y \ar@{^{(}->}[rr]^{j} \ar[dr]_{\widetilde{g}} & & M \ar[dl]^\varrho  \\
    &  \PBbb^1_S & }
\end{equation}
where $j$ is the strict transform of $\PBbb^1_Y$ in $\PBbb^1_X$. In the proposition below, denote by $N$ the normal bundle of $Y$ in $X$.
\begin{proposition}\label{prop:Deformationcone}

The diagram \eqref{eq:deformationtonormalcone} enjoys the following properties:
\begin{enumerate}
    \item \label{item:Deformationcone-1} The formation commutes with base change $S' \to S$.
    \item \label{item:Deformationcone-3} The scheme $M_\infty = \varrho^{-1}(\infty)$ is the sum of two relative Cartier divisors over $S$, $\PBbb(N^\vee \oplus 1)$  and $\widetilde{X} \simeq \Bl_{Y} X$. They intersect in a mutual relative Cartier divisor over $S$, $\PBbb(N^{\vee})$, seen as the hyperplane at infinity of $\PBbb(N^{\vee}\oplus 1)$, which is also the exceptional divisor in $\widetilde{X}$. In particular, $M_{\infty}$ is the pushout scheme of $\widetilde{X}$ and $\PBbb(N^\vee \oplus 1)$ along $\PBbb(N^{\vee})$. 
    \item \label{item:Deformationcone-4} The morphism $j$ is a regular closed immersion of codimension $c$, and:
     \begin{enumerate}        
        \item\label{item:restriction-j-A1} The restriction of $j$ to $\ABbb^1_S$ identifies with the natural morphism $\ABbb^{1}_{Y} \to \ABbb^{1}_{X}$;
         \item\label{item:Y-zero-section} The restriction of $j$ to $S\times \{\infty\}$ identifies with the zero section $j_{\infty}\colon Y\to \PBbb(N^\vee \oplus 1)\subset M_{\infty}$.
    \end{enumerate} 
    \item \label{item:Deformationcone-2} The morphism $\varrho$ satisfies the condition $(C_n)$ over $\PBbb^{1}_{S}$.\end{enumerate}
\end{proposition}

\begin{proof}

The compatibility with arbitrary base changes follows from \cite[Proposition 19.4.6]{EGAIV4}, which also ensures that $M \to S$ is  flat and of finite presentation. It follows, from \cite[\href{https://stacks.math.columbia.edu/tag/02FV}{02FV}]{stacks-project}, that $M \to \PBbb^1_S$ is of finite presentation. This can be used to show that $\varrho: M \to \PBbb^1_S$ is faithfully flat. Indeed, since we already know that $M\to S$ is flat and of finite presentation, by the flatness criterion by fibers \cite[Th\'eor\`eme 11.3.10]{EGAIV3} it suffices to verify that $M_s \to \PBbb^1_{k(s)}$ is flat for all $s \in S$. This is addressed in \cite[Chapter 5]{Fulton}. To see that $\varrho$ is surjective, we note that $M\to\PBbb^{1}_{X}$ is surjective by construction, and $X\to S$ is surjective by assumption, so that $\PBbb^{1}_{X}\to\PBbb^{1}_{S}$ is surjective. 

The next two points follow as in \cite[Chapter 5]{Fulton}, with the exception of the pushout claim and the fact that $j$ is a regular immersion. 

For the pushout claim, write $D_{1}=\PBbb(N^\vee \oplus 1)$  and $D_{2}=\widetilde{X} \simeq \Bl_{Y} X$. Denote by $k_{\infty}^{i}\colon D_{i}\to M_{\infty}$ and $k^{12}_{\infty}\colon \PBbb(N^{\vee})\to M_{\infty}$ for the closed immersions. The fact that $M_{\infty}$ is the sum of $D_{1}$ and $D_{2}$, which intersect in the common Cartier divisor $\PBbb(N^{\vee})$, means that the natural complex
\begin{displaymath}
    0\to\Ocal_{M_{\infty}}\to k_{\infty,\ast}^{1}\Ocal_{D_{1}}\oplus k_{\infty,\ast}^{2}\Ocal_{D_{2}}\to k^{12}_{\infty,\ast}\Ocal_{\PBbb(N^{\vee})}\to 0
\end{displaymath}
is exact. This agrees with the characterization of the structure sheaf of the pushout provided by \cite[\href{https://stacks.math.columbia.edu/tag/0B7M}{0B7M}]{stacks-project}. 

To see that $j\colon\PBbb^1_Y\to M$ is a regular immersion, since it is a closed immersion of locally finitely presented $S$-flat schemes, by \cite[Proposition 19.2.4]{EGAIV4} it is enough to verify the claim over the points of $S$, so that we can assume it is the spectrum of a field $K$. Now, $\PBbb^1_Y \to M$ is an immersion of $\PBbb^1_K$-flat schemes, so it is enough to verify the statement over the points of $\PBbb^1_K$. Over $\ABbb^{1}_{K}$ there is nothing to prove. Over $\infty$, this follows from the description in \eqref{item:Y-zero-section}.

We now focus on the last point, and first show that $\varrho$ satisfies the remaining properties of the condition $(C_n).$ To prove that $\varrho$ is locally projective, we assume that $S$ is affine. In this case, by \cite[\href{https://stacks.math.columbia.edu/tag/07RM}{07RM}]{stacks-project} it is enough to prove that that $M \to \PBbb^{1}_X$ is projective. This follows from the fact that the ideal defining $Y \to X$ is quasi-coherent of finite type as an $\Ocal_X$-module, hence so is the ideal defining $Y\times\lbrace\infty\rbrace$ in $\PBbb^{1}_{X}$. 

To prove that the fibers of $\varrho$ are equidimensional, by the base-change property, it is enough to assume that $S$ itself is the spectrum of a field $K$. Next, it is enough to consider the fiber over $\infty$, since the other fibers are isomorphic to (a base change of) $X$, which is equidimensional by assumption. As remarked in the second point, the fiber $M_{\infty}$ is the union of $\PBbb(N^\vee \oplus 1)$ and $\widetilde{X}$, intersecting along the common effective Cartier divisor $\PBbb(N^\vee)$. Since $Y$ is equidimensional and $\PBbb(N^\vee\oplus1)\to Y$ is smooth of constant relative dimension $c$, we deduce that $\PBbb(N^\vee\oplus 1)$ is equidimensional. Similarly for $\PBbb(N^\vee)$. Using that $\PBbb(N^\vee)$ is an effective Cartier divisor in $\widetilde{X}$, it is purely of codimension one, see for instance \cite[Corollary 5.1.8]{EGAIV2}. Combining the several properties on dimension and codimension provided by \cite[Proposition 5.1.9, Corollary 5.2.3]{EGAIV2}, we deduce that for any closed point $\widetilde{x} \in \PBbb(N^\vee)$:
\begin{displaymath}
    \dim_{\widetilde{x}} \widetilde{X} = \dim_{\widetilde{x}} \PBbb(N^\vee)+1.
\end{displaymath}
Now, if $x$ is the image of $\widetilde{x}$ in $Y$, we have 
\begin{displaymath}
    \dim_{\widetilde{x}}\PBbb(N^\vee)+1=\dim_{x} Y+c=\dim_{x} X=n, 
\end{displaymath}
where we have used that $Y\to X$ is a regular closed immersion of codimension $c$, and the above properties of dimension and codimension. Therefore, we see that $\dim_{\widetilde{x}} \widetilde{X}=n$. The same is true for closed points in $\widetilde{X}\setminus\PBbb(N^\vee) =X\setminus Y$.

\end{proof}

The construction above often allows one to reduce questions about closed immersions to the model immersion $j_{\infty}\colon Y \to \PBbb(N^\vee \oplus 1)$, given by the zero section of $N$. We refer to this as the \emph{linearized case}.

\subsection{Deformation of resolutions} We maintain the setting and notation of the previous subsection. We will need to explicitly follow how resolutions of direct images deform when passing to the linearized case. 

We start by remarking that on $\PBbb(N^{\vee}\oplus 1)$ there is a tautological exact sequence
\begin{equation}\label{eq:tautologicalsequenceP}
    0 \to \Ocal(-1) \to p^*  (N \oplus 1) \to Q\to 0,
\end{equation}
where $p\colon\PBbb(N^{\vee} \oplus 1)\to Y$ is the projection and $Q  = T_{\PBbb(N^\vee \oplus 1)/Y}(-1)$ is the universal quotient bundle. Let 
\begin{displaymath} 
    \sigma_0: \Ocal_P \to p^*  (N \oplus 1) \to Q
\end{displaymath} 
be the section of $Q$ obtained by the inclusion of $\Ocal_P \to  p^*  (N \oplus 1)$ into the second factor and then projecting onto $Q$, multiplied by $-1$.\footnote{Here we follow the convention of \cite[p. 311]{BGSImmersions}.} Then, the associated Koszul complex provides a resolution $K^\bullet = K^\bullet(\sigma_0)$ of $(j_{\infty})_{\ast}\Ocal_{Y}$. Given any vector bundle $E$ on $Y$, we hence obtain a Koszul-type resolution 
\begin{equation} \label{eq:Koszulresolution}
    K^\bullet \otimes p^* E \to (j_{\infty})_{\ast} E.
\end{equation}
Below we explain the geometry of these resolutions on the deformation to the normal cone.

\begin{proposition}\label{prop:GS-normal-cone}

Let $i: Y \to X$ be a regular closed immersion of $S$-schemes as above, with normal bundle $N$. Suppose $E$ is a vector bundle on $Y$. 

\begin{enumerate}
    \item\label{item:BFM-GS-1} If $F^\bullet \to i_* E$ is a finite resolution by vector bundles, there is a canonically associated finite resolution by vector bundles $\widetilde{F}^\bullet \to j_* \widetilde{E}$ on $M$, where $\widetilde{E}$ is the pullback of $E$ along $\PBbb^1_Y\to Y$. Moreover: 
    \begin{enumerate}
        \item The restriction of $\widetilde{F}^\bullet \to j_* \widetilde{E}$ to $\ABbb^1_X \subseteq M$  identifies with the pullback of $F^\bullet \to i_* E$ along $\ABbb^1_X \to X.$ 
        \item The restriction of $\widetilde{F}^\bullet \to j_* \widetilde{E}$ to $\PBbb(N^{\vee}\oplus 1)$ above $S \times \{\infty\}$  sits in a natural short exact sequence 
        \begin{displaymath}
            0 \to G^\bullet \to  \widetilde{F}^\bullet|_{\PBbb(N^{\vee}\oplus 1)} \to K^\bullet \otimes p^\ast E \to 0,
        \end{displaymath}
    \end{enumerate}
    where $K^{\bullet}$ is the Koszul resolution of the zero section $Y\to\PBbb(N^\vee \oplus 1)$ and $G^{\bullet}$ is a finite acyclic complex of vector bundles. The restriction of $\widetilde{F}^\bullet$ to $\widetilde{X}$ over $ \infty$ is split acyclic.
    \item\label{item:BFM-GS-2} Suppose that $Y$ is the zero locus of a regular section $\sigma$ of a vector bundle on $X$, and let $K^{\bullet}(\sigma)$ be the associated Koszul resolution of $i_{\ast}\Ocal_{Y}$. Denote by $F$ a vector bundle on $X$ and by $F^\bullet = K^{\bullet}(\sigma) \otimes F$, which is a resolution of $i_\ast i^\ast F$ as in \eqref{item:BFM-GS-1}. 
    \begin{enumerate}
        \item The immersion $\PBbb^{1}_{Y}\to M$ is cut out by the zero locus of a regular section $\widetilde{\sigma}$ of a vector bundle on $M$, and the complex $\widetilde{K}^{\bullet}= {K}^{\bullet}(\widetilde{\sigma})$ is canonically isomorphic to a Koszul resolution of $j_\ast \Ocal_{\PBbb^1_Y}.$ 
        \item The complex $\widetilde{F}^\bullet$ is canonically isomorphic to $\widetilde{K}^\bullet \otimes \widetilde{F}$, where $\widetilde{F}$ denotes the pullback of $F$ along $M \to X.$
        \item The restriction of $\widetilde{F}^\bullet$ to $\ABbb^1_X$ is naturally isomorphic to the pullback of $F^\bullet$ along $\ABbb^1_X \to X.$
        \item The restriction $\widetilde{F}^\bullet$ to $\PBbb(N^\vee \oplus 1)$ is canonically isomorphic to $K^\bullet(\sigma_0) \otimes p^\ast i^\ast F$, i.e. in this case $G^\bullet$ is zero.
                     \end{enumerate}    \item The constructions above commute with base changes along morphisms of schemes $S' \to S$. 
\end{enumerate}

\end{proposition}

\begin{proof}

The results can be found in \cite{BGSImmersions} and \cite{GS-ARR}, which are inspired and build on earlier work in \cite{BFM}. We will indicate how the statements transfer to our setting, by pointing out the key points and the necessary modifications. 

In \cite{GS-ARR}, the vector bundles are assumed to be coherent locally free sheaves, but in the applications it is enough to assume they are locally free sheaves of finite type.

In the Grassmannian graph construction, one repeatedly performs Zariski closures, which are well defined as schemes in their setting due to the integral hypothesis. This needs to be replaced by the schematic closure. For the comparison with the deformation to the normal cone, it is important to note that the strict transform of a subscheme $Z$ in a blowup of a closed subscheme $Y$ of a scheme $X$ coincides with schematic closure of $Z \setminus Y$ in the blowup. Also, one relies on the fact that $\ABbb^{1}_{X}$ is schematically dense in $\PBbb^{1}_{X}$, which is true for any scheme $X.$ This follows from \cite[\href{https://stacks.math.columbia.edu/tag/01RG}{01RG}]{stacks-project}, as well as \cite[\href{https://stacks.math.columbia.edu/tag/01RE}{01RE}]{stacks-project} combined with \cite[Proposition 19.9.8]{EGAIV4}.

There is a hypothesis of integrality in \cite[Lemma 1]{GS-ARR}, which is used repeatedly too. Its main statement is that a homomorphism of locally free sheaves $\phi: A\to B$ which vanishes on a dense Zariski open set is zero on the whole scheme. This is valid in general instead supposing that the open set is schematically dense, without assuming that the scheme is integral. The lemma has one more part, but this relies on the first part. With this understood, the first point of the proposition can be found in \cite[Theorem 4.8]{BGSImmersions} and \cite[Theorem 1]{GS-ARR}. The second point is \cite[Lemma 4.7]{BGSImmersions} in the case of a regular section of a trivial vector bundle. The constructions extend to general vector bundles, as in the proof \cite[Theorem 4.8]{BGSImmersions}. The compatibility with base change can be obtained from Proposition \ref{prop:Deformationcone} \eqref{item:Deformationcone-1} and the explicit construction in \cite{GS-ARR} using the functoriality of the Grassmannian construction, together with the Tor-independent base change applied to the complexes of $S$-flat sheaves $F^\bullet \to i_\ast E$.

\end{proof}
The statement in Proposition \ref{prop:GS-normal-cone} \eqref{item:BFM-GS-1} shows that resolutions can be transformed into the linearized case in a controlled manner. Nevertheless, the formalism of direct images in virtual categories enables us to circumvent such constructions altogether. However, this approach remains important in contexts where one needs to keep track of additional data, such as metrics, and we thus include it for future reference. In this article, we will only need the particular case in Proposition \ref{prop:GS-normal-cone} \eqref{item:BFM-GS-2}.

\subsection{Composition of linearizations} 
We remark upon some additional compatibilities of the previous constructions in the case of the composition of closed immersions. 

\begin{proposition}\label{prop:doubleinclusion}
    Let $Z \overset{i'}{\to} Y \overset{i}{\to} X$ be a composition of regular closed immersions over $S$, defined by closed subschemes, and denote by $N=N_{i \circ i'}$ and $N'=N_{i'}$ the associated normal bundles on $Z$, and $N^{\bis}=N/N^{\prime}$.
    \begin{enumerate}
        \item There is a natural inclusion of deformations to the normal cones 
    \begin{displaymath}
            \PBbb^1_Z \overset{j'}{\to} M' \overset{j}{\to} M,
        \end{displaymath}
    such that the following holds:
    \begin{enumerate}
       \item \label{doubleinclusion-item1} $j'$ and $j$ are both regular immersions. 
       \item  \label{doubleinclusion-item2} The restriction to $\infty$ naturally induces  the inclusion of linearized cases 
       \begin{equation}\label{eq:linearcase}
            Z \overset{j'_\infty}{\to} \PBbb(N^{\prime\vee} \oplus 1) \overset{j_\infty}{\to} \PBbb(N^\vee \oplus 1).
       \end{equation}

    \end{enumerate}
    \item \label{doubleinclusion-item3} In the setting of \eqref{eq:linearcase}, all three closed immersions are regular closed immersions, cut out by regular sections. More precisely, consider the short exact sequence 
\begin{equation}\label{eq:Fultonexactsequence}
    0 \to {N'}(1) \to N(1) \to N^{\bis}(1) \to 0
\end{equation}
    on $\PBbb(N^\vee \oplus 1)$. Then:
    \begin{enumerate}
    \item There is a regular section $\sigma$ of $N(1)$ which cuts out $Z$. 
    \item The projection of $\sigma$ onto $N^{\bis}(1)$ defines a regular section $\sigma''$ of $N^{\bis}(1)$, which cuts out $\PBbb(N^{\prime\vee} \oplus 1) $.
    \item The restriction of $\sigma$ to $\PBbb(N^{\prime\vee} \oplus 1) $ defines a regular section $\sigma'$ of $N'(1)$, which cuts out $Z. $
    \end{enumerate}
    \end{enumerate}
\end{proposition}

\begin{proof}
    The existence of the maps in the first point is clear, as well as the fact that the restriction to $\infty$ induces the claimed inclusions. We have already proven that $j'$ is a regular closed immersion in Proposition \ref{prop:Deformationcone} \eqref{item:restriction-j-A1}. We then claim that $j$ is a regular closed immersion. Since all of our $\PBbb^{1}_{S}$-schemes are flat of finite presentation by Proposition \ref{prop:Deformationcone}, by a pointwise argument as in the proof of \eqref{item:Deformationcone-4} in \emph{loc. cit.} it is enough to verify that the inclusion $M^{\prime}_{\infty} \to M_{ \infty}$ is a regular closed immersion whenever the base is a field. In particular, we may assume that all the schemes are Noetherian. In this case, it is enough to prove that $j_\infty$ is a regular closed immersion at the points of $\PBbb(N^{\prime\vee} \oplus 1) \cap \widetilde{Y} = \PBbb(N^{\prime\vee})$. Since this is a Cartier divisor in $M_{\infty}^{\prime}$, it is enough to prove that the inclusion $\PBbb(N^{\prime\vee}) \to M_{\infty}$ is regular, cf. \cite[\href{https://stacks.math.columbia.edu/tag/0690}{0690}]{stacks-project}. This factors as 
    \begin{displaymath}
        \PBbb(N^{\prime\vee}) \to \PBbb(N^\vee) \to M_{\infty},
    \end{displaymath}
    which are clearly two regular closed immersions.
 
The statement concerning the regular sections follows from \cite[Appendix B.5.6]{Fulton}
\end{proof}

We introduce some notation in preparation for the following lemma. If $X$ is a scheme and $E$ is a vector bundle of rank $r$ on $X$, then we denote by $\lambda_{-1}(E)$ the virtual vector bundle in $V(X)$
\begin{equation}\label{eq:lambda-1-def}
    \lambda_{-1}(E):=\sum_{k=0}^{r}(-1)^{k}\Lambda^{k}E.
\end{equation}
By \cite[Proof of Lemma 9.3]{DRR1}, the operation $\lambda_{-1}$ behaves multiplicativily on short exact sequences, for the tensor product structure on $V(X)$, possibly up to sign. As in \emph{op. cit.}, this sign ambiguity is harmless in our context, since we ultimately work rationally and hence in particular invert 2, which kills the sign. Namely, we are allowed to consider $\lambda_{-1}$ as taking values in the rational virtual category $V(X)_{\QBbb}$. See also \textsection \ref{subsec:signs} above.

Given a section $\sigma$ of $E^\vee$, note that the object $\lambda_{-1}(E)$ is naturally isomorphic to the object $[K(\sigma)]$ of the associated Koszul complex $K(\sigma)$ in the virtual category $V(X)$, cf. \cite[Section 9.3]{DRR1}. We write the latter in the form
\begin{displaymath}
    0\to \Lambda^{r}E\to\Lambda^{r-1}E\to\cdots\to E\to\Ocal_{X}\to 0,
\end{displaymath}
where $\Lambda^{p}E$ is placed in degree $-p$. 

In the following lemma, we suppose our schemes are divisorial, for the purpose of having pushforward functors between virtual categories of vector bundles. 

\begin{lemma}\label{lemma:KoszulKoszulKoszul}
In the setting of Proposition \ref{prop:doubleinclusion} \eqref{doubleinclusion-item3}, supposing in addition that $X$ is a divisorial scheme, the  diagram

\begin{displaymath}
    \xymatrix{ {(i \circ i')}_! \Ocal_Z \ar[r] \ar[d] & \lambda_{-1}(N^{\vee}(-1))  \ar[d] \\ i_! (\lambda_{-1}(N^{\prime\vee}(-1))) \ar[r] &\lambda_{-1}(N^{\prime\vee}(-1)) \otimes  \lambda_{-1}(N^{\bis\vee}(-1))}
\end{displaymath}
commutes in the rational virtual category $V(\PBbb(N^\vee \oplus 1))_\QBbb$. Here:

\begin{enumerate}
    \item The upper horizontal map comes from the Koszul resolution of $\Ocal_Z$ on $\PBbb(N^\vee \oplus 1).$
    \item The leftmost vertical map comes from the Koszul resolution of $\Ocal_Z$ on $\PBbb(N^{\prime\vee} \oplus 1)$, and then applying $i_!$.
    \item The lower horizontal map comes from the Koszul resolution of $\Ocal_{\PBbb(N^{\prime\vee} \oplus 1)}$ on $\PBbb(N^\vee \oplus 1)$. 
    \item The rightmost vertical map comes from the multiplicativity of $\lambda_{-1}$ on short exact sequences.
\end{enumerate} 
\end{lemma}
    
\begin{proof}
 We begin with the case when the sequence of vector bundles is split, so that $ N = N' \oplus N^{\bis}$. We will consider the following two filtrations on $K(\sigma)$: 
    \begin{itemize}
        \item $F^k K(\sigma)^{p} =\Imag\left( \Lambda^{k}(N^{\prime\vee}(-1))\otimes \Lambda^{p-k}(N^{\vee}(-1))\to \Lambda^{p}(N^{\vee}(-1))\right) $. The graded quotients are given by
            $$\Gr_F^k K(\sigma)^{p}=\Lambda^{k}(N^{\prime\vee}(-1)) \otimes \Lambda^{p-k} (N^{\bis\vee}(-1)),$$
                    and hence
        \begin{displaymath}
            \Gr^k_F(K(\sigma)) = \Lambda^{k}({N^{\prime \vee}}(-1)) \otimes K(\sigma^{\bis}) [- k] .
        \end{displaymath}
        \item $G^{k}K(\sigma)^{p}=K(\sigma)^{p}$ if $p\leq k$ and $0$ otherwise. The graded pieces are given by
        \begin{displaymath}
            \Gr_G^k K(\sigma)^{p}=
            \begin{cases}
                     K(\sigma)^{k}      &\text{if}\ k=p,\\
                     0                  &\text{if}\ k\neq p.
            \end{cases}
        \end{displaymath}
    \end{itemize}
Second, we consider the following diagram:
\begin{displaymath}
    \xymatrix{
    {(i \circ i')}_! (\Ocal_Z) \ar[d] \ar[r]\ar[dr]   & [K(\sigma)] \ar[r] \ar[d] & \lambda_{-1}(N^\vee (-1)) \ar[dd] \\
    i_! [K(\sigma')] \ar[r] \ar[d] & \ar[d][K(\widetilde{\sigma}')] \otimes [K(\sigma'')] &  \\
    i_! \lambda_{-1}(N^{\prime \vee}(-1)) \ar[r] &  \lambda_{-1}(N^{\prime \vee}(-1)) \otimes [K(\sigma'')]  \ar[r] & \lambda_{-1}({N^{\prime\vee}}(-1)) \otimes \lambda_{-1}(N^{\bis\vee}(-1)).  }
\end{displaymath}
In the diagram above, we denote by $\widetilde{\sigma}^{\prime}$ the section of $N'(1)$ on $\PBbb(N^{\vee}\oplus 1)$ obtained by projecting the section $\sigma$ of $N(1)$ onto the factor $N^{\prime}(1)$. This extends the section $\sigma^{\prime}$, and $\sigma = \widetilde{\sigma}^{\prime} \oplus \sigma^{\bis}$. In this situation, there is an isomorphism of Koszul complexes $K(\sigma)\simeq K(\widetilde{\sigma}^{\prime}) \otimes K(\sigma^{\bis})$, which provides a resolution of $i_\ast K(\sigma')$ . These facts imply that the upper left diagrams commute. The lower left diagram commutes since the isomorphism $[K(\widetilde{\sigma}^{\prime})]\to\lambda_{-1}(N^{\prime\vee}(-1))$ restricts to $[K(\sigma^{\prime})]\to\lambda_{-1}(N^{\prime\vee}(-1))$ through $i$, and since the horizontal maps correspond to the projection formula for $i$, which is a natural transformations of functors. The rightmost diagram commutes by applying \cite[Section 4.7]{Deligne-determinant} to the two filtrations $F^\bullet$ and $G^\bullet$, which states that for a biadmissible filtration, the result of successively applying both filtrations is independent of the order. To appeal to this formula, we remark that the double-graded quotients are locally free, and hence the filtration is biadmissible. 

To reduce the general case to the split case, we introduce the transgression exact sequence, which interpolates between the general exact sequence and the split case, over $\PBbb^1$, cf. \cite[Section 2.2]{Eriksson-Freixas-Wentworth}. This can be written as 
\begin{displaymath}
    0 \to \widetilde{N}'(1) \to \widetilde{N}(1) \to \widetilde{N}^{\bis}(1) \to 0.
\end{displaymath}
We also write down the corresponding inclusion of projective bundles:
\begin{displaymath}
    \PBbb^1_Z \to \PBbb(\widetilde{N}^{\prime\vee} \oplus 1) \to \PBbb(\widetilde{N}^{\vee} \oplus 1). 
\end{displaymath}
We can consider the statement of the lemma. At infinity, the transgression sequence is split, and the corresponding section $\sigma$ identifies with $\widetilde{\sigma}^\prime \oplus \sigma^{\bis}$ as above. The failure of the commutativity of the diagram is measured by an automorphism $g$ of an object in $V(\PBbb(\widetilde{N}^{\vee} \oplus 1))_\QBbb$. As explained in \textsection \ref{intro:virtualcategories}, this identifies as an element of $K_1(\PBbb(\widetilde{N}^{\vee} \oplus 1))_\QBbb,$ and the statement of actual commutativity is that this is trivial. For $t \in \PBbb^1$, we denote by $i_{t}^{\ast}g$ the restriction of $g$ along $t$. Then, by the above argument, we have that $i_{\infty}^{\ast} g$ measures the commutativity in the split case, and hence $i_{\infty}^{\ast} g= 0$. The fact that $i_{0}^{\ast}g=0$ is proven in the lemma below. This completes the proof.
\end{proof}

\begin{lemma}\label{lemma:splittonotsplit}
In the situation above, $i_{0}^{\ast}g=0$ if and only if $i_{\infty}^{\ast}g=0$.
\end{lemma}
\begin{proof}
   We claim that there is a commutative diagram 
    \begin{displaymath}
        \xymatrix{
        & K_1(\PBbb(\widetilde{N}^{\vee} \oplus 1)) \ar[dl]_{i_{0}^{\ast}} \ar[dr]^{i_{\infty}^{\ast}}  & \\
        K_1(\PBbb({N}^{\vee} \oplus 1)) \ar[rr]^{\sim} & & K_1 (\PBbb((N' \oplus N^{\bis})^{\vee} \oplus 1)).}
    \end{displaymath}
     First, we introduce some notation for the $K_{0}$ classes of the $\Ocal(1)$-type bundles of the several involved projective bundles: $\Ocal(1)$ for $\PBbb^{1}_{Z}$, $\widetilde{\Ocal}(1)$ for $\PBbb(\widetilde{N}^{\vee} \oplus 1)$, $\Ocal_{0}(1)$ for $\PBbb({N}^{\vee} \oplus 1)$, and finally $\Ocal_{\infty}(1)$ for $\PBbb((N' \oplus N^{\bis})^{\vee} \oplus 1)$. By the projective bundle formula in higher algebraic $K$-theory, cf. \cite[Theorem 4.1]{ThomasonTrobaugh}, an element of the $K_1(\PBbb(\widetilde{N}^{\vee} \oplus 1))$ is of the form 
     \begin{displaymath}
        \sum a_{i,j} \otimes \Ocal(i) \otimes \widetilde{\Ocal}(j),
     \end{displaymath}
     with $a_{i,j}$ are elements of $ K_{1}(Z)$, and the indices are restricted to $-1\leq i\leq 0$ and $-\rk N+1\leq j\leq 0$. This is mapped to the elements 
     \begin{displaymath}
        \sum a_{i,j}\otimes \Ocal_{0}(j),\quad\text{resp. } \linebreak \sum a_{i,j} \otimes \Ocal_{\infty}(j),
    \end{displaymath}
    by the restrictions $i_{0}$, resp. $i_{\infty}$. On the other hand, by the same projective bundle formula, $K_1(\PBbb(N^\vee \oplus 1))$ is freely generated as $K_1(Z)$-module by $1, \Ocal_0(-1), \ldots, \Ocal_0(-\rk N+1),$ and likewise for $K_1(\PBbb((N' \oplus N^{\bis})^{\vee} \oplus 1))$, and we define the lower horizontal arrow accordingly. The formula for the restrictions shows that the diagram commutes. Then, the statement follows from the commutative diagram tensored by $\QBbb$.
\end{proof}

\section{Construction for regular closed immersions}\label{section:closed-immersions}
In this section we construct the DRR isomorphism in the case of regular closed immersions. This proceeds in two steps. First, we address the linearized case, that is the case of the zero section of a projective bundle $\PBbb(N^{\vee}\oplus 1)\to Y$, and then we perform a deformation to the normal cone to reduce to this situation. We also establish the compatibility of the construction with the projection formula and the composition of closed immersions. As a byproduct of our arguments, we can also propose a characterization of the construction. Throughout, we assume that all the $S$-schemes satisfy the condition $(C_{n})$, for some $n$, except possibly for base changes $S^{\prime}\to S$. 

\subsection{Construction using a resolution} 
In this subsection, we describe a DRR-type isomorphism associated to a regular closed immersion $i\colon Y\to X$ of $S$-schemes, which is cut out by the zero locus of a regular section $\sigma$ of a vector bundle $\Ecal$ on $X$ of constant rank. We then elaborate on the particular case of Cartier divisors.

\subsubsection{The general construction} Let $i\colon Y\to X$ be as above. We allow the exceptional case where the section $\sigma$ is nowhere-vanishing and $Y$ is empty, and in this situation we apply the conventions from \textsection\ref{sec:notations-conventions}. We will follow a construction in \cite[Corollary 9.7]{DRR1} to provide, in this setting, the following definition.

\begin{construction-definition}\label{construction-definition-RRsec}
Suppose first that $F$ is a virtual vector bundle on $X$, and let $E=i^{\ast}F$. In this case, the DRR isomorphism associated with the vector bundle $\Ecal$ and section $\sigma$ is constructed as the following composition :
\begin{equation}\label{eq:BorelSerre}
    \begin{split}
       \hspace{1cm} [\chfrak(i_{!} E)]_{X/S} &\overset{(1)}{\simeq} \chfrak(F) \cdot [\chfrak(i_! \Ocal_Y)]_{X/S}\\
            &\overset{(2)}{\simeq} \chfrak(F) \cdot i_\ast [\tdfrak(N)^{-1}]_{Y/S} \overset{(3)}{=} i_\ast [\chfrak(E) \cdot \tdfrak(N)^{-1}]_{Y/S},
    \end{split}\tag{$\RR_{sec}(F, \sigma)$}
\end{equation}
where $N$ is the normal bundle of $Y\to X$. Concretely, these isomorphisms are defined as follows:
\begin{enumerate}
    \item The first isomorphism is given by the projection formula from Proposition \ref{prop:proj-for-RR-dist} \eqref{item:proj-for-RR-dist-1}, using that $E=i^{\ast}F$. That is, it combines the projection formula at the level of virtual categories \cite[Proposition 4.12 (2)]{DRR1} with the multiplicativity of the categorical Chern character with respect to the tensor product \cite[Proposition 9.1 (1)]{DRR1}.
    \item The second isomorphism $[\chfrak(i_! \Ocal_Y)]_{X/S} \simeq i_\ast [\tdfrak(N)^{-1}]_{Y/S}$ is \cite[Corollary 9.7]{DRR1}, and it is obtained in two steps:
    \begin{enumerate}
        \item First, rewriting the direct image of the trivial sheaf in terms of a Koszul complex, whose categorical Chern character can be rearranged in a way which makes the top Chern class of $\Ecal$ appear, by the Borel--Serre isomorphism from \cite[Theorem 9.5]{DRR1} and recalled in \eqref{eq:Borel-Serre-prelim} above.
        \item Second, the Chern categorical version of the cycle representation $c_{\rm{top}}(\Ecal) = [Z(\sigma)]$ in Theorem \ref{thm:cor86} \eqref{item:prop-int-dist-restriction}. 
    \end{enumerate}
    \item The last isomorphism is given by the projection formula from Proposition \ref{prop:proj-for-RR-dist} \eqref{item:proj-for-RR-dist-2}, using that $i^\ast F = E$.
\end{enumerate}
By Lemma \ref{lemma:extension-RRsec} below, the construction extends to virtual perfect complexes on $X$.
\end{construction-definition}

Several remarks are in order. We note that in \cite[Corollary 9.7]{DRR1}, it was assumed that the immersion $i$ admits a retraction $X\to Y$, but supposing that $E$ extends to $X$ is actually enough. We also comment on two extreme cases. If $\sigma$ is the zero section of the zero bundle, then $\RR_{sec}(F,\sigma)$ is the identity automorphism of $[\chfrak(F)]_{X/S}$. If $\sigma$ is nowhere-vanishing, the conventions in \textsection \ref{sec:notations-conventions} stipulate that $\RR_{sec}(F,\sigma)$ is the identity automorphism of the constant functor $\Ocal_{S}$. This is compatible with the empty case of the Borel--Serre isomorphism \cite[Theorem 9.5 (3)]{DRR1}, on which the above construction is based. 

\begin{lemma}\label{lemma:extension-RRsec}
The construction $F\mapsto\RR_{sec}(F, \sigma)$ induces an isomorphism of functors of commutative Picard categories $V(X)\to\Dcal(X/S)$, compatible with base change and with bounded denominators. By Proposition \ref{prop:base-reduction}, this extends to an isomorphism of functors of commutative Picard categories $V(\Pcal_{X})\to\Dcal(X/S)$, compatible with base change and with bounded denominators.
\end{lemma}
\begin{proof}
By Proposition \ref{prop:base-reduction}, we can restrict to divisorial base schemes and work with virtual vector bundles instead of virtual perfect complexes. Then, we just need to remark that pulling back virtual bundles by $i$ defines a functor of commutative Picard categories, and the projection formulas used in the definition of $\RR_{sec}(F, \sigma)$ are functorial too, by Proposition \ref{prop:proj-for-RR-dist}. We also need to observe that all the involved constructions are already compatible with base change, and have bounded denominators depending only on the denominators of $\chfrak$ and $\tdfrak$ up to a given order, which is bounded by the relative dimensions of $Y\to S$ or $X\to S$ plus one. 
\end{proof}

\subsubsection{The case of Cartier divisors}\label{subsubsec:RRimmersionD} The case of a closed immersion given by a relative effective Cartier divisor $i\colon D\to X$ is of particular importance, and we briefly discuss it in detail. 

The sheaf $i_{\ast}\Ocal_{D}$ admits a resolution given by the natural map $\Ocal_{X}(-D)\to\Ocal_{X}$. This induces an isomorphism $i_{!}\Ocal_{D}\simeq \Ocal_{X}-\Ocal_{X}(-D)$ in the virtual category, and hence
\begin{equation}\label{eq:RRimmersionD-1}
    [\chfrak(i_{!}\Ocal_{D})]_{X/S}\simeq [\chfrak(\Ocal_{X}-\Ocal_{X}(-D))]_{X/S}.
\end{equation}
Using the rank triviality isomorphism from Theorem \ref{thm:cor86}, the isomorphism $[\cfrak_{1}(\Ocal_{X}(-D))]\simeq -[\cfrak_{1}(\Ocal(D))]$ from \cite[Proposition 8.11]{DRR1} and a formal rearrangement of power series, we obtain the Borel--Serre isomorphism in this case:
\begin{equation}\label{eq:RRimmersionD-2}
    [\chfrak(\Ocal_{X}-\Ocal_{X}(-D))]_{X/S}\simeq [\cfrak_{1}(\Ocal(D))\cdot\tdfrak(\Ocal(D))^{-1}]_{X/S}.
\end{equation}
By the restriction isomorphism from Theorem \ref{thm:cor86}, we find
\begin{equation}\label{eq:RRimmersionD-3}
    [\cfrak_{1}(\Ocal(D))\cdot \tdfrak(\Ocal(D))^{-1}]_{X/S}\simeq [\tdfrak(N_{D/X})^{-1}]_{D/S}.
\end{equation}
Concatenating \eqref{eq:RRimmersionD-1}--\eqref{eq:RRimmersionD-3} recovers $\RR_{sec}(\Ocal_{X},1)$, where $1$ is the canonical section of $\Ocal_{X}(D)$.

\subsection{Reduction to the case of a resolution}
 In the general case of a regular closed immersion of a subscheme, by the previous section on the deformation to the normal cone, the construction of a Deligne--Riemann--Roch isomorphism can be reduced to the setting of \eqref{eq:BorelSerre}. Below, we follow the notation of Section \ref{sec:def-normal-cone}.

Let $Y$ be a closed subscheme of $X$, and suppose that the morphism $Y\to S$ satisfies the condition $(C_{n-c})$, and that the inclusion $i\colon Y\to X$ is a regular closed immersion of codimension $c$ and normal bundle $N$. Let $\pi\colon M\to\PBbb^{1}_{X}$ be the associated deformation to the normal cone construction. For $t \in \PBbb^1$, denote by $k_t: M_t \to M$ the natural inclusion. We denote by $k_\infty^{1}: \PBbb(N^\vee \oplus 1)\to M$ the restriction of $k_{\infty}$ to $\PBbb(N^\vee \oplus 1)$, and by $k_{\infty}^{2}\colon \widetilde{X}\to M$ the restriction to the blowup $\widetilde{X}$. The following lemma introduces the necessary rational equivalence reducing to the linear case. 

\begin{lemma}\label{lemma:rationalequivalence}
For any virtual perfect complex $E$ on $Y$ there is an isomorphism of line distributions
    \begin{displaymath}
        k_{0, \ast } [\chfrak(i_! E)]_{X/S} \simeq k_{\infty, \ast }^{1} [\chfrak(j_{\infty, !} E)]_{\PBbb(N^\vee \oplus 1)/S} 
    \end{displaymath}
  over $M \to S$. This induces an isomorphism of functors of commutative Picard categories $V(\Pcal_Y) \to \Dcal(M/S)$, compatible with base change and with bounded denominators.
\end{lemma}
\begin{proof}
By Proposition \ref{prop:base-reduction}, we can restrict to divisorial schemes and work with virtual categories of vector bundles. 

By \cite[Corollary 8.8, Corollary 8.9]{DRR1}, there is a canonical isomorphism of line distributions for $M\to S$:

\begin{equation}\label{eq:rationalequivalencelinedistr}
    \delta_{X/S}= \delta_{M_0/S} \simeq \delta_{M_\infty/S} \simeq \delta_{\PBbb(N^\vee \oplus 1)/S} + \delta_{\widetilde{X}/S}.
\end{equation}
We multiply this by $\chfrak(j_! \widetilde{E}),$ and consider the associated line distribution over $M\to S$. The intersection of $\PBbb^{1}_{Y}$ with both components $\PBbb(N^\vee \oplus 1)$ and $\widetilde{X}$ is Tor-independent. This implies that the pullback of $j_! \widetilde{E}$ to $\PBbb(N^\vee \oplus 1)$ is isomorphic to $j_{\infty, !} E$ and the pullback to $\widetilde{X}$ is isomorphic to 0. 

For the functoriality, the key points are that $E\mapsto j_{!}\widetilde{E}=j_{!}p^{\ast}E$ induces a functor of commutative Picard categories and the fact that the deformation to the normal cone construction commutes with base changes $S^{\prime}\to S$, cf. Proposition \ref{prop:Deformationcone} \eqref{item:Deformationcone-2}. The boundedness of denominators is clear from the construction. The denominators depend only on the relative dimensions of $X\to S$ and $Y\to S$, and they can be extracted from the denominators of the power series defining $\chfrak$.
\end{proof}
\begin{proposition-definition}\label{prop:RR-closed-immersions}
Let $i\colon Y \to X$ be a regular closed immersion of $S$-schemes with normal bundle $N$. Then, there is a canonical isomorphism 
\begin{displaymath}
    \RR_{i}(E)\colon [\chfrak(i_! E)]_{X/S} \to i_\ast \left[ \chfrak(E) \cdot \tdfrak(N)^{-1} \right]_{Y/S},
\end{displaymath}
which induces an isomorphism of functors $V(\Pcal_{Y}) \to \Dcal(X/S)$ of commutative Picard categories, compatible with base change and with bounded denominators. Moreover, it is functorial in $i$, in the sense of Definition \ref{def:compatibilitymorphisms}, and it is the identity isomorphism if $i$ is the identity map.\end{proposition-definition}
\begin{proof}
First, by Proposition \ref{prop:isomor-of-mor} and by imposing the compatibility with isomorphisms in the sense of Definition \ref{def:compatibilitymorphisms}, we can reduce the construction to the case where $Y$ is actually a closed subscheme of $X$. We are thus in the setting above, and we maintain the notation therein. 

Consider the projection $q: M \to \PBbb^1_X \to X$, for which we have $q \circ k_0 = \id_X$ and $q \circ k_\infty \circ j_\infty = i$. Taking direct images along $q$ in the isomorphism of Lemma \ref{lemma:rationalequivalence}, we find an isomorphism of line distributions from $X \to S:$
\begin{equation}\label{eq:rationalequivalence}
   [\chfrak(i_! E)]_{X/S} \simeq q_\ast k_{\infty, \ast} [\chfrak(j_{\infty, !} E)]_{\PBbb(N^\vee \oplus 1)/S}.
\end{equation}
Since the regular immersion $Y \to \PBbb(N^\vee \oplus 1)$ is cut out by the zero locus of the regular section $\sigma_{0}$ of the tautological quotient bundle $Q$ as in \eqref{eq:tautologicalsequenceP}, we can next apply the isomorphism \eqref{eq:BorelSerre} with $F = p^\ast E$ and the regular section $\sigma_0$. This yields
\begin{displaymath}
    [\chfrak(j_{\infty, !} E)]_{\PBbb(N^\vee \oplus 1)/S} \simeq j_{\infty, \ast} [\chfrak(E) \cdot \tdfrak(N)^{-1}]_{Y/S}.
\end{displaymath}
Composing this isomorphism with \eqref{eq:rationalequivalence}, and using that $q_\ast k_{\infty, \ast} j_{\infty, \ast} = i_\ast$, we conclude. The functoriality and compatibility with base change are clear by construction, as well as the boundedness of denominators. 

For the functorial dependence on $i$, by the very construction, one can easily reduce to the following setting:
\begin{displaymath}
\xymatrix{
        Y^{\prime}\ar[r]^{i^{\prime}}\ar[d]_{\varphi}^{\sim}  &X^{\prime}\ar[d]^{\psi}_{\sim}\\
        Y\ar[r]^{i} &X,
    }
\end{displaymath}
where the immersions are defined by closed subschemes. Also by the deformation to the normal cone, one is led to treat the linearized case, for which the DRR isomorphisms are given by Construction/Definition \ref{construction-definition-RRsec}. The problem further amounts to studying the compatibility of \cite[Corollary 9.7]{DRR1} with isomorphisms of schemes, which in turns hinges on the analogous question for the Borel--Serre isomorphism \eqref{eq:Borel-Serre-prelim} and the restriction isomorphism from Theorem \ref{thm:cor86} \eqref{item:prop-int-dist-restriction}. These are compatible with the birational invariance isomorphism from Theorem \ref{thm:cor86} \eqref{item:prop-int-dist-birational}, and in particular with isomorphisms of schemes. Indeed, this was observed in \textsection\ref{subsubsec:mult-chern-Borel-Serre} for the Borel--Serre isomorphism, and it is part of Theorem \ref{thm:cor86} for the restriction isomorphism. 

Finally, the fact that $\RR_{i}$ is the identity isomorphism if $i$ is the identity map is trivial. It also follows from Proposition \ref{prop:RRres=RR} in the case of the zero section of the zero vector bundle.
\end{proof}

Note that in case where $Y\to X$ is cut out by the zeros of a regular section, both constructions \eqref{eq:BorelSerre} and Proposition \ref{prop:RR-closed-immersions} apply. The following proposition ensures the resulting isomorphisms agree.
\begin{proposition}\label{prop:RRres=RR}
        Let $i: Y \to X$ be a regular closed immersion cut out by the zeros of a regular section $\sigma$ of a vector bundle. Also, let $F$ (resp. $E$) be a virtual perfect complex on $X$ (resp. $Y$), with an isomorphism $\psi: E \simeq i^\ast F$. Then, the following diagram commutes, where the vertical arrows are induced by $\psi$ :
        \begin{displaymath}
            \xymatrix{[\chfrak(i_! E)]_{X/S} \ar[rr]^-{\RR_i(E)}\ar[d] & & i_\ast  [\chfrak(E) \cdot \tdfrak(N)^{-1}]_{Y/S} \ar[d] \\
            [\chfrak(i_! i^\ast F)]_{X/S} \ar[rr]^-{\RR_{sec}( F,\sigma)} & & i_\ast  [\chfrak(i^\ast F) \cdot \tdfrak(N)^{-1}]_{Y/S}    }         
        \end{displaymath}
        In particular, $\RR_i(i^\ast F) = \RR_{sec}(F, \sigma)$, which only depends on the immersion $i$ and the restriction $i^{\ast}F$.   \end{proposition}
\begin{proof}
To prove the commutativity, we may assume that the base schemes are divisorial and work with virtual categories of vector bundles, cf. Proposition \ref{prop:base-reduction}. Since $\RR_{i}$ is functorial with respect to isomorphisms in the bundle, we can reduce to the case $E= i^\ast F$, which we assume from now on. Therefore, we must show $\RR_i(i^\ast F) = \RR_{sec}(F, \sigma)$.

We then consider the following diagram, whose commutativity is proven below:

\begin{displaymath}
\resizebox{\textwidth}{!}{
\xymatrix@R=2em@C=10em{    k_{0,\ast} [\chfrak(i_! E)] \ar[d] \ar[r]^-{k_{0, \ast} \RR_{sec}(F, \sigma)} \ar@/_10pc/[dddddd]_{\varphi} & \ar@/^10pc/[dddddd]^{\phi} k_{0, \ast} i_\ast [\chfrak(E) \cdot \tdfrak(N)^{-1}] \ar[d] \\
    [\chfrak(j_! \widetilde{E})] \cdot \delta_0 \ar[d] \ar[r]^-{\RR_{sec}(\widetilde{F}, \widetilde{\sigma}) \cdot \delta_0} & j_\ast [\chfrak(\widetilde{E})\cdot \tdfrak(N_j)^{-1}] \cdot \delta_0 \ar[d] \\
    [\chfrak(j_! \widetilde{E})] \cdot \delta_\infty \ar[d] \ar[r]^-{\RR_{sec}(\widetilde{F}, \widetilde{\sigma}) \cdot \delta_\infty} & j_\ast [\chfrak(\widetilde{E})\cdot \tdfrak(N_j)^{-1}] \cdot \delta_{\infty} \ar[d] \\
    k_{\infty, \ast} [\chfrak(j_{\infty, !} E)] \ar[r]^-{k_{\infty, \ast}\RR_{sec}(\widetilde{F}|_{M_\infty}, \widetilde{\sigma}|_{M_{\infty}})} \ar[d] & k_{\infty, \ast } j_{\infty, \ast}[\chfrak(E) \cdot \tdfrak(N)^{-1}] \ar[d] \\
    k_{\infty, \ast}^{1} [\chfrak(j_{\infty,!}^{1} E)] \ar@{}[d]|-{+}    &    \ar[l]+<40pt,-22pt>;+<-60pt,-22pt>^{k_{\infty, \ast}^1 \RR_{sec}(\widetilde{F}|_{\PBbb(N^\vee \oplus 1)}, \sigma_0)}_{\overset{+}{k_{\infty, \ast}^2 \RR_{sec}(\widetilde{F}|_{\widetilde{X}}, \widetilde{\sigma}|_{\widetilde{X}} )}}   k_{\infty, \ast}^{1}j_{\infty, \ast }^{1} [\chfrak(E) 
    \cdot \tdfrak(N)^{-1}]  \ar@{}[d]|-{+}\\   
    k_{\infty, \ast}^{2}[\chfrak(j_{\infty, !}^{2} 0)]\ar[d]  &k_{\infty, \ast}^{2} j_{\infty, \ast }^{2} [\chfrak(0) \cdot \tdfrak(N|_{\emptyset})^{-1}] \ar[d]\\
    k_{\infty, \ast}^{1} [\chfrak(j_{\infty, !}^{1} E)] \ar[r]^-{k_{\infty, \ast} \RR_{sec}(\widetilde{F}|_{\PBbb(N^\vee \oplus 1)}, \sigma_0)} &  k_{\infty, \ast}^{1}j_{\infty, \ast }^{1} [\chfrak(E) 
    \cdot \tdfrak(N)^{-1}].
}
}
\end{displaymath}
In the diagram, the notation is as follows:
\begin{itemize}
    \item We write $\delta_{0}$ and $\delta_{\infty}$ for the distributions $\delta_{M_{0}/S}$ and $\delta_{M_{\infty}/S}$.
    \item We denote by $\widetilde{F}$ the pullback of $F$ to $M$ and $\widetilde{E}$ the pullback of $E$ to $\PBbb^{1}_{Y}$. 
    \item We introduce the closed immersions $k_{\infty}^{1}\colon \PBbb(N^{\vee}\oplus 1)\to M$, $k_{\infty}^{2}\colon\widetilde{X}\to M$, $j_{\infty}^{1}\colon Y\to \PBbb(N^{\vee}\oplus 1)$ and $j_{\infty}^{2}\colon\emptyset\to\widetilde{X}$.
\end{itemize}
Moreover, the vertical arrows are defined similarly to Lemma \ref{lemma:rationalequivalence}, and we in particular use that the constructions restricted to $\widetilde{X}$ are canonically trivial.

The fact that these diagrams actually commute relies on several facts. One is the actual description of $\RR_{sec}$ from Construction/Definition \ref{construction-definition-RRsec}, and the restriction morphism along the zeros of a regular section. The claim as such is that the vertical isomorphisms preserve this datum. For the rational equivalence isomorphism $\delta_{0}\simeq\delta_{\infty}$ (cf. \cite[Corollary 8.8]{DRR1}), this follows since the definition is in terms of multiplication by $\cfrak_1(\Ocal_{\PBbb^{1}}(1))$ and restriction along the divisors of the standard sections $X_{0}$ and $X_{1}$ of $\Ocal_{\PBbb^{1}}(1)$, together with the commutativity of the restriction isomorphism for multiple sections. For the birational invariance isomorphism $\delta_{M_\infty/S} \simeq \delta_{\PBbb(N^\vee \oplus 1)/S} + \delta_{\widetilde{X}/S}, $ (cf. \cite[Corollary 8.7]{DRR1}), this follows from the description in terms of restriction to the components $\PBbb(N^{\vee}\oplus 1)$ and $\widetilde{X}$ (cf. \cite[Proposition 6.13]{DRR1}), and the commutativity of the birational invariance isomorphism and the restriction along the zeros of regular sections. The claimed commutativity of isomorphisms of line distributions is ensured by Theorem \ref{thm:cor86}.

The fact that the lower diagram commutes follows from the empty case in \cite[Corollary 9.6]{DRR1}, which asserts that both $[\chfrak(j_{\infty,!}^{2}0)]$ and $j_{\infty, \ast }^{2} [\chfrak(0) \cdot \tdfrak(N|_{\emptyset})^{-1}]$ are canonically trivialized, and that the isomorphism $\RR_{sec}$ interchanges these trivializations. 

Having established this, let us now note that by the universal property of the virtual category, to prove the proposition we can assume that our initial $F$ is actually a vector bundle. By Proposition \ref{prop:GS-normal-cone} \eqref{item:BFM-GS-2}, in this case $\widetilde{F}|_{\PBbb(N^\vee \oplus 1)} \simeq p^\ast i^\ast F$. Applying the functor $q_\ast$ to $k_{\infty,\ast}\RR_{sec}(p^\ast i^\ast F,\sigma_{0}) \circ \varphi$ is exactly the definition of $\RR_i(i^\ast F)$. Since the diagram commutes, this is equal to $(q_\ast \phi) \circ q_\ast k_{0, \ast} \RR_{sec}(F, \sigma).$ Since $q k_0 = \id_X$ and $q k_{\infty} j_{\infty} = i$, we find that the isomorphism $q_\ast \phi$ naturally identifies with the identity, and $q_\ast k_{0, \ast} \RR_{sec}(F, \sigma) \simeq \RR_{sec}(F, \sigma)$. In other words, the commutativity of the diagram ensures that $\RR_i(i^\ast F)$ coincides with $\RR_{sec}(F, \sigma)$. 

\end{proof}
The assertion of Proposition \ref{prop:RRres=RR} can be seen to imply that the DRR isomorphism is compatible with the projection formula.  The following corollary asserts that this also holds in general. 

\begin{corollary}\label{cor:projformula}
    If $F$ (resp. $E$) is a virtual vector bundle on $X$ (resp. virtual perfect complex on $Y$), the following diagram commutes: 
    \begin{displaymath}
        \xymatrix{\chfrak(F) \cdot [\chfrak(i_! E)]_{X/S} \ar[rr]^-{\chfrak(F) \cdot \RR_i(E)} \ar[d] &  &\chfrak(F) \cdot i_\ast [\chfrak(E)\cdot \tdfrak(N)^{-1}]_{Y/S} \ar[d] \\ 
        [\chfrak(i_!(i^\ast F \otimes E))]_{X/S} \ar[rr]^-{\RR_{i}(i^\ast F \otimes E)} &   &i_\ast [\chfrak(i^\ast F \otimes E)\cdot \tdfrak(N)^{-1}]_{Y/S}.
        }
    \end{displaymath}
    Here, the vertical arrows are given by the projection formula for the Riemann--Roch distributions from Proposition \ref{prop:proj-for-RR-dist}. 
\end{corollary}

\begin{proof}
     In the linearized case of a zero section $Y\to\PBbb(N^{\vee}\oplus 1)$, by Proposition \ref{prop:RRres=RR}, the isomorphism $\RR_{i}$ amounts to a computation by $\RR_{sec}$, which is defined in terms of the projection formula, using as extension of $E$ the bundle $p^\ast E$. Since $\widetilde{F}|_{\PBbb(N^\vee \oplus 1)} = p^\ast i^\ast F$ the diagram which appears in this comparison automatically commutes. 

     In general, one reduces to this case by a deformation to the normal cone argument, which is part of the construction of $\RR_{i}$. The details are omitted.
\end{proof}

\subsection{Composition of immersions: the linearized case}
 In this subsection we show that the DRR isomorphism for regular closed immersions is compatible with composition of immersions, in the linearized case. In the subsequent subsection we will show how to reduce to this case.

\begin{proposition}
\label{prop:linearcase}
Let $Z$ be an $S$-scheme, and suppose that $N' \subseteq N$ is an inclusion of vector bundles of constant rank, such that $N/N'$ is locally free. Consider the corresponding regular closed immersions $Z \overset{i'}{\to} \PBbb(N^{\prime\vee} \oplus 1) \overset{i}{\to} \PBbb(N^\vee \oplus 1)$. Then, $\RR_{i'}\RR_{i}=\RR_{ii'}$.   
\end{proposition}
\begin{proof} 
Without loss of generality, we may assume that our schemes are divisorial and work with virtual categories of vector bundles, cf. Proposition \ref{prop:base-reduction}. 

First, any vector bundle on $Z$ will extend to $\PBbb(N^\vee \oplus 1)$. Then, by Corollary \ref{cor:projformula}, we can reduce to study the compatibility with the composition for the trivial vector bundle on $Z$. 

Second, by Proposition \ref{prop:RRres=RR}, we can use the description of the Riemann--Roch isomorphism provided by $\RR_{sec}$. Indeed, by Proposition \ref{prop:doubleinclusion} \eqref{doubleinclusion-item3} all three inclusions are cut out by regular sections, in a way compatible with each other in the sense accounted for in the proposition.

We now consider the diagram below:

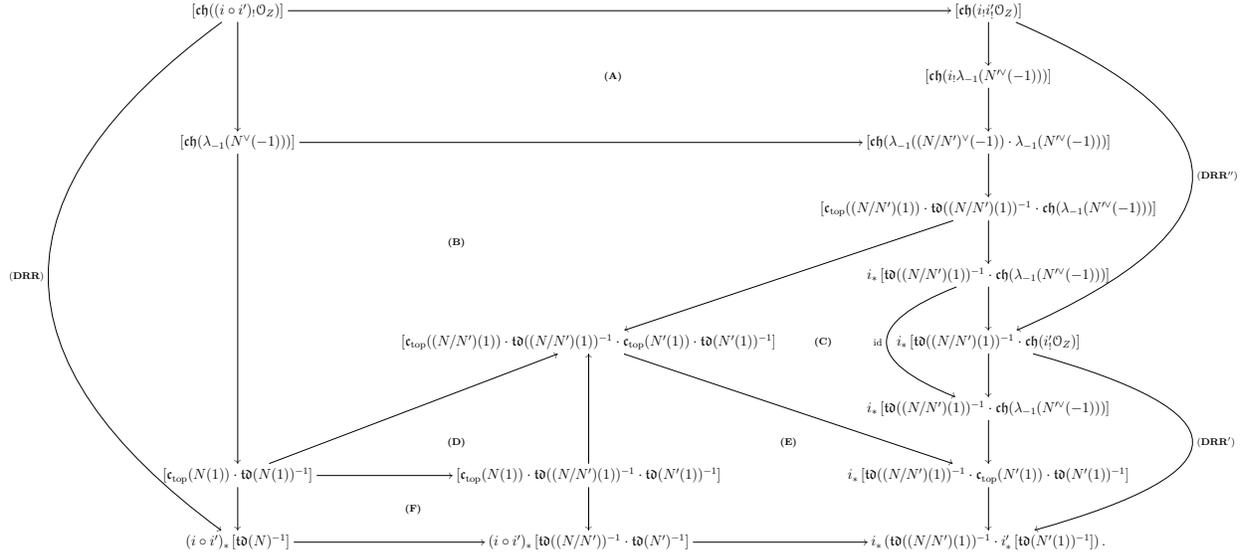
\begin{displaymath}
\resizebox{\textwidth}{!}{
\xymatrix@R=3em@C=2.5em{
     \left[\chfrak((i \circ i')_! \Ocal_Z)\right]  
     \ar@/_13pc/[dddddddd]_{\mathbf{(DRR)}} 
     \ar@{}[ddrrr]|{\textbf{(A)}} 
     \ar[dd] 
     \ar[rrr] 
     & & & 
     \left[\chfrak(i_! i^{\prime}_! \Ocal_Z)\right]  
     \ar[d] 
     \ar@/^14pc/[ddddd]^{\mathbf{(DRR'')}} \\
    & & & 
    \left[\chfrak(i_! \lambda_{-1}(N^{\prime\vee}(-1)))\right] 
    \ar[d] \\
    \left[\chfrak(\lambda_{-1}(N^\vee(-1)))\right] 
    \ar[ddddd] 
    \ar[rrr] 
    \ar@{}[dddrr]|{\hspace{2.5cm}\textbf{(B)}} 
    & & & 
    \left[\chfrak(\lambda_{-1}((N/N')^\vee(-1)) \cdot \lambda_{-1}(N^{\prime\vee}(-1)))\right] 
    \ar[d] \\
     &  &  & 
     \left[\cfrak_{\operatorname{top}}((N/N')(1)) \cdot \tdfrak((N/N')(1))^{-1} \cdot \chfrak(\lambda_{-1}(N^{\prime\vee}(-1)))\right] 
     \ar[d] 
     \ar[ddl] \\
    &  & &  
    i_\ast \left[\tdfrak((N/N')(1))^{-1} \cdot \chfrak(\lambda_{-1}(N^{\prime\vee}(-1)))\right]  
    \ar[d] 
    \ar@/_7pc/[dd]_{\operatorname{id}} \\
    \ar@{}[dddrr]|{\hspace{2.5cm}{\textbf{(D)}}} 
    & & 
    \ar@{}[dddr]|{\textbf{(E)}}   
    \left[\cfrak_{\operatorname{top}}((N/N')(1)) \cdot \tdfrak((N/N')(1))^{-1} \cdot \cfrak_{\operatorname{top}}(N'(1)) \cdot \tdfrak({N'}(1))^{-1}\right]  
    \ar[ddr] 
    \ar@{}[r]|{\hspace{2cm}\textbf{(C)}} 
    & i_\ast \left[\tdfrak((N/N')(1))^{-1} \cdot \chfrak({i}'_! \Ocal_Z)\right] 
    \ar[d] 
    \ar@/^14pc/[ddd]^{\mathbf{(DRR')}}\\
    & & & 
    i_\ast \left[\tdfrak((N/N')(1))^{-1} \cdot \chfrak(\lambda_{-1}(N^{\prime\vee}(-1)))\right]   
    \ar[d] \\
    \left[\cfrak_{\operatorname{top}}(N(1)) \cdot \tdfrak(N(1))^{-1}\right] 
    \ar@{}[drr]|{\textbf{(F)}} 
    \ar[d] 
    \ar[rr] 
    \ar[uurr] 
    & & 
    \left[\cfrak_{\operatorname{top}}(N(1)) \cdot \tdfrak((N/N')(1))^{-1} \cdot \tdfrak(N'(1))^{-1}\right] 
    \ar[d] \ar[uu]
    & i_\ast \left[\tdfrak((N/N')(1))^{-1} \cdot \cfrak_{\operatorname{top}}(N'(1)) \cdot \tdfrak({N'}(1))^{-1}\right]    
    \ar[d] \\
    (i \circ i')_\ast \left[ \tdfrak(N)^{-1}\right]   
    \ar[rr] 
    &  &  
    (i \circ i')_\ast \left[\tdfrak((N/N') )^{-1} \cdot \tdfrak({N'})^{-1}\right]   
    \ar[r]  
    & i_\ast \left(\tdfrak((N/N')(1))^{-1} \cdot i^{\prime}_{\ast} \left[ \tdfrak(N'(1))^{-1}\right]\right).  
}
}
\end{displaymath}
We first explain the meaning of the arrows $\mathbf{(DRR)}$, $\mathbf{(DRR')}$ and $\mathbf{(DRR'')}$:
\begin{itemize}
    \item[$\mathbf{(DRR)}$] is given by $\RR_{sec}(\Ocal_{\PBbb(N^{\prime\vee} \oplus 1)}, \sigma),$ and hence by Proposition \ref{prop:RRres=RR} realizes $\RR_{i \circ i'}(\Ocal_Z)$.
    \item[$\mathbf{(DRR^{\prime})}$] is given by the isomorphism $i_\ast\left[ \tdfrak((N/N')(1))^{-1} \cdot \RR_{i'}(\Ocal_Z) \right]$, where we have realized $\RR_{i'}(\Ocal_Z)$ as $\RR_{sec}(\Ocal_{\PBbb(N^{\prime\vee}\oplus 1)}, \sigma')$.
    \item[$\mathbf{(DRR^{\prime\prime})}$] is given by $\RR_{i}({i}^{\prime}_{!} \Ocal_Z)$, realized using $\RR_{sec}(\lambda_{-1}(N^{\prime\vee}(-1)),\sigma'')$, where we have used the isomorphism $i'_! \Ocal_Z \simeq \lambda_{-1}(N^{\prime\vee}(-1))$ in the virtual category, and the fact that $\lambda_{-1}(N^{\prime\vee}(-1))$ canonically extends to $\PBbb(N^\vee \oplus 1)$.
\end{itemize}
Hence, the outer contour of the above diagram is the comparison stated in the lemma. We will show that this outer contour commutes, by proving the commutativity of all the interior diagrams.

The diagram $\mathbf{(A)}$ commutes by applying $\chfrak$ to the diagram in Lemma \ref{lemma:KoszulKoszulKoszul}. Because we work with rational coefficients, the fact that it only commutes up to sign is innocuous. The diagram $\mathbf{(B)}$ amounts to the compatibility of the Borel--Serre isomorphism with short exact sequences from \cite[Theorem 9.5(1)]{DRR1}. The diagram $\mathbf{(C)}$ commutes because it is the restriction of the Borel--Serre isomorphism, which commutes with restriction, as observed in \textsection \ref{subsubsec:mult-chern-Borel-Serre}. The diagram $\mathbf{(D)}$ commutes by definition, since the diagonal arrow is given by the two others. The diagram $\mathbf{(E)}$ commutes because of \cite[Proposition 7.14]{DRR1}, applied to the sections involved in the exact sequence \eqref{eq:Fultonexactsequence} of Proposition \ref{prop:doubleinclusion}. This states that in this particular situation, the restriction isomorphisms are compatible with each other in the asserted sense.  The diagram $\mathbf{(F)}$ commutes because the Whitney isomorphism is compatible with the restriction isomorphism, as stated in Theorem \ref{thm:cor86}.

\end{proof}

\subsection{Composition of immersions: the general case}
We trace out the steps needed to reduce the composition of immersions to the already treated linearized case. The main result of this subsection is hence the following:

\begin{proposition}\label{prop:compositionofinclusions}
Let $i: Y \to X$ and $i': Z \to Y$ be regular closed immersions. Then, $\RR_{i'}\RR_{i}=\RR_{ii'}$.    
\end{proposition}

To be able to prove the proposition, we will need a technical result on the compatibility of the DRR for closed immersions with some birational morphisms. Let us introduce the setting. We consider a Cartesian diagram of $S$-schemes
\begin{equation}\label{eq:cartesian-birational-invariance}
    \xymatrix{
        Y^{\prime}\ar[r]^{i^{\prime}}\ar[d]_{\varphi}   &X^{\prime}\ar[d]^{\psi}\\
        Y\ar[r]^{i}       &X.
    }
\end{equation}
We make the following assumptions on the various morphisms:
\begin{enumerate}
    \item\label{item:all-mor-same-dim} $X\to S$ and $X^{\prime}\to S$ have the same (constant) relative dimension, and similarly for $Y\to S$ and $Y^{\prime}\to S$.
    \item The morphism $\psi$ has open image and it induces a birational morphism onto its image, in the sense of Theorem \ref{thm:cor86} \eqref{item:prop-int-dist-birational}, and $\varphi$ inherits the analogous structure.
    \item The maps $i$ and $i^{\prime}$ are regular closed immersions.
\end{enumerate}
For clarity, we elaborate on the meaning of the second assumption. The map $\psi$ has open image by assumption, and it is necessarily proper, since all the $S$-schemes are. Hence it has closed image too. We can thus factor $X=X_{1}\sqcup X_{2}$, where $X_{1}=\psi(X^{\prime})$ is open-and-closed, and so is $X_{2}$. By the condition \eqref{item:all-mor-same-dim}, we see that necessarily $X_{2}\cap Y=\emptyset$. We hence just require that $\varphi$ is birational in the sense of Theorem \ref{thm:cor86} \eqref{item:prop-int-dist-birational}. 

The assumptions above have several implications:
\begin{enumerate}
    \item The closed immersions necessarily have the same codimension, by the dimension condition imposed on the $S$-schemes in the diagram. In particular, their normal bundles have the same rank.
    \item There is a natural isomorphism
    \begin{equation}\label{eq:iso-normal-bundles}
        \varphi^{\ast}N_{i}\overset{\sim}{\to} N_{i'}.
    \end{equation}
    Indeed, let $\Jcal_{i}$ and $\Jcal_{i'}$ be the ideals defining the closed immersions. Since the diagram is Cartesian, we have a natural surjective map $\psi^{\ast}\Jcal_{i}\to\Jcal_{i'}$, and hence a surjective map $\varphi^{\ast}N_{i}\to N_{i'}$. This must be an isomorphism, since the normal bundles have the same rank by the previous point.
    \item The fact that the regularity and the codimension of the immersion are preserved under the Cartesian product entails that the diagram is Tor-independent by \cite[Lemma 3.2]{Thomason:excess}. Conversely, if the diagram is Tor-independent, then the regularity and the codimension of the immersion are preserved, since locally over $X$ the latter can be expressed as the exactness of a Koszul complex, which remains exact after pulling back to $X^{\prime}$.
\end{enumerate}

\begin{lemma}\label{lemma:ignorepart}
Let the setting be as above. Let $E$ denote a virtual perfect complex on $Y$. Then:
\begin{enumerate}
    \item\label{item:ignore-1} There exist natural isomorphisms 
    \begin{equation}\label{eq:statement-ignore-1}
        \psi_{\ast}[\chfrak(i^{\prime}_{!}\varphi^{\ast} E)]_{X^{\prime}/S}\to [\chfrak(i_{!}E)]_{X/S}
    \end{equation}
    and
    \begin{equation}\label{eq:statement-ignore-2}
        \psi_{\ast}i^{\prime}_{\ast}[\chfrak(\varphi^{\ast}E)\cdot\tdfrak(N_{i^{\prime}})^{-1}]_{Y^{\prime}/S}\to i_{\ast}[\chfrak(E)\cdot\tdfrak(N_{i})^{-1}]_{Y/S}
    \end{equation}
    inducing isomorphisms of functors of commutative Picard categories $V(\Pcal_{Y})\to\Dcal(X/S)$, compatible with base change and the projection formulas from Proposition \ref{prop:proj-for-RR-dist}, and with bounded denominators.
     \item\label{item:ignore-2}  The diagram
    \begin{equation}\label{eq:birational-inv-RR-diagram}
         \xymatrix{  
             \psi_{\ast}[\chfrak(i^{\prime}_{!}\varphi^{\ast} E)]_{X^{\prime}/S} \ar[rr]^-{\psi_{\ast}\RR_{i^{\prime}}(\varphi^{\ast}E)}\ar[d]_{\eqref{eq:statement-ignore-1}}     &  &\psi_{\ast}i^{\prime}_{\ast}[\chfrak(\varphi^{\ast}E)\cdot\tdfrak(N_{i^{\prime}})^{-1}]_{Y^{\prime}/S}\ar[d]^{\eqref{eq:statement-ignore-2}}\\
              [\chfrak(i_{!}E)]_{X/S}\ar[rr]^-{\RR_{i}(E)}    &   &i_{\ast}[\chfrak(E)\cdot\tdfrak(N_{i})^{-1}]_{Y/S}
         }
    \end{equation}
   commutes.
\end{enumerate}
\end{lemma}

\begin{proof}
As usual, we may assume that our schemes are divisorial, and we can work with virtual vector bundles instead of perfect complexes. 

We can factor $\psi$ as a birational morphism in the sense of Theorem \ref{thm:cor86} \eqref{item:prop-int-dist-birational}, and an open-and-closed immersion, inducing an analogous factorization for $\varphi$. To prove the lemma, we can reduce to dealing separately with the case of birational morphisms and the case of open-and-closed immersions. 

Suppose next that $\psi$ and $\varphi$ are birational morphisms. In this case the isomorphism $\psi_{\ast}[\chfrak(i^{\prime}_{!}\varphi^{\ast} E)]_{X^{\prime}/S}\simeq [\chfrak(i_{!}F)]_{X/S}$ is induced by Tor-independent base change and birational invariance of the intersection distributions, in the following manner. We already observed that $\psi$ and $i$ are Tor-independent morphisms, and hence there is an isomorphism  $i^{\prime}_{!}\varphi^{\ast}E\simeq \psi^{\ast}i_{!} E$. Thus, we have a sequence of natural isomorphisms
\begin{displaymath}
    \psi_{\ast}[\chfrak(i^{\prime}_{!}\varphi^{\ast} E)]\simeq \psi_{\ast}[\psi^{\ast}\chfrak(i_{!}E)]\simeq [\chfrak(i_{!}E)],
\end{displaymath}
where the last identification is given by birational invariance isomorphism from Theorem \ref{thm:cor86} \eqref{item:prop-int-dist-birational}. This defines an isomorphism of functors between commutative Picard categories, the key point being that the Whitney isomorphism is compatible with the birational invariance, as stated in Theorem \ref{thm:cor86}. Moreover, it is compatible with base changes $S^{\prime}\to S$, since this is Tor-independent with our $S$-schemes. The denominators of all these isomorphisms are bounded, depending only on the dimension of $X\to S$ and the denominators of $\chfrak$. For the compatibility with the projection formula, we first observe that at the level of the virtual categories, the corresponding projection formula in \cite[Proposition 4.12]{DRR1} is compatible with Tor-independent base change, as seen by inspecting the proof. The result then follows from the construction in Proposition \ref{prop:proj-for-RR-dist} \eqref{item:proj-for-RR-dist-1} and the compatibility of \eqref{eq:chern-tensor-product} with the birational invariance isomorphism, as noted in \emph{loc. cit.}. The isomorphism \eqref{eq:statement-ignore-2} is obtained in a similar manner, using the canonical isomorphism $\varphi^{\ast}N_{i}\simeq N_{i'}$ from \eqref{eq:iso-normal-bundles}.

For the commutativity of \eqref{eq:birational-inv-RR-diagram}, by Proposition/Definition \ref{prop:RR-closed-immersions}, we may assume that $i$ and $i'$ are given by closed subschemes. Let then $M$ and $M'$ denote the deformation to the normal cone constructions for $Y\to X$ and $Y^{\prime}\to X^{\prime}$. The map $\psi$ induces a birational morphism $\widetilde{\psi}\colon M^{\prime}\to M$, which also satisfies the assumptions of Theorem \ref{thm:cor86} \eqref{item:prop-int-dist-birational}. We then observe that \eqref{eq:birational-inv-RR-diagram} is the fiber at 0 of a similar diagram, performed at the deformation to the normal cone level. For \eqref{eq:birational-inv-RR-diagram} to commute, it suffices to check the commutativity of the diagram for $\widetilde{\psi}$ on the fiber at infinity. Indeed, evaluating along a Chern power series, the defect of commutativity for $\widetilde{\psi}$ is given by an invertible function on $\PBbb^{1}_{S}$, which is actually an invertible function on $S$, and hence its value is determined either by restricting to 0 or infinity. Now, $M_{\infty}=\widetilde{X}\cup\PBbb(N_{i}^{\vee}\oplus 1)$ and $M_{\infty}^{\prime}=\widetilde{X}^{\prime}\cup\PBbb(N_{i^{\prime}}^{\vee}\oplus 1)$, and the birational morphism $M_{\infty}^{\prime}\to M_{\infty}$ restricts to a birational morphisms $\widetilde{X}^{\prime}\to\widetilde{X}$ and $\PBbb(N_{i^{\prime}}^{\vee}\oplus 1)\to\PBbb(N_{i}^{\vee}\oplus 1)$. By the construction of the DRR isomorphism in Proposition \ref{prop:RR-closed-immersions}, one reduces to treating separately the corresponding cases. That is, on the one hand, the immersion of the empty scheme in $\widetilde{X}^{\prime}$ and $\widetilde{X}$, and the birational morphism $\widetilde{X}^{\prime}\to\widetilde{X}$. On the other hand, the immersion of the zero sections in $\PBbb(N_{i^{\prime}}^{\vee}\oplus 1)$ and $\PBbb(N_{i}^{\vee}\oplus 1)$ gives us the diagram
\begin{equation}\label{eq:zero-sections-birational}
    \xymatrix{
         Y^{\prime}\ar[r]\ar[d]  &\PBbb(N_{i^{\prime}}^{\vee}\oplus 1)\ar[d]\\
        Y\ar[r]       &\PBbb(N_{i}^{\vee}\oplus 1).
    }
\end{equation}
This diagram fulfills the assumptions of the lemma. In the first case of the embedding of the empty scheme, the commutativity of \eqref{eq:birational-inv-RR-diagram} is trivial. In the case of \eqref{eq:zero-sections-birational}, the DRR isomorphisms are given by Construction/Definition \ref{construction-definition-RRsec}. This construction imposes the compatibility with the projection formula, and the the isomorphisms \eqref{eq:statement-ignore-1}--\eqref{eq:statement-ignore-2} are compatible with it. One hence reduces to treat the birational invariance of the isomorphism $[\chfrak(i_! \Ocal_Y)]_{X/S} \simeq i_\ast [\tdfrak(N_{i})^{-1}]_{Y/S}$ from \cite[Corollary 9.7]{DRR1}. Given the construction of the latter in \emph{op. cit.}, recalled in Construction/Definition \ref{construction-definition-RRsec}, this amounts to the following two facts:
\begin{enumerate}
    \item The restriction isomorphism and the birational invariance isomorphisms are compatible, as asserted in the conclusion of Theorem \ref{thm:cor86}.
    \item The Borel--Serre isomorphism from \cite[Theorem 9.5]{DRR1} is compatible with the birational invariance isomorphism, as observed in \textsection\ref{subsubsec:mult-chern-Borel-Serre}.\end{enumerate}
We conclude that $[\chfrak(i_! \Ocal_Y)]_{X/S} \simeq i_\ast [\tdfrak(N_{i})^{-1}]_{Y/S}$ exhibits the sought birational invariance.

If $\psi$ and $\varphi$ are open-and-closed immersions, then we can identify $X=X_{1}\sqcup X_{2}$, where the $X_{i}$ are open-and-closed subschemes of $X$, and $X_{2}\cap Y=\emptyset$. In this case, the isomorphism in \eqref{item:ignore-1} is part of the very construction of intersection bundles \cite[Section 7.1]{DRR1}, which imposes a compatibility with open partitions of the schemes. Concretely, if $P$ is a Chern power series on $X$, then we have a sequence of natural isomorphisms:

\begin{displaymath}
    \begin{split}
         [\chfrak(i_{!}E)]_{X/S}(P)=\langle \chfrak(i_{!}E)\cdot P\rangle_{X/S}\simeq &\langle \chfrak(i_{!}E|_{X_{1}})\cdot P\rangle_{X_{1}/S}\otimes \langle \chfrak(i_{!}E|_{X_{2}})\cdot P\rangle_{X_{2}/S}\\
         &\simeq \langle \chfrak(i_{!} E|_{X_{1}})\cdot P\rangle_{X_{1}/S}\otimes \langle \chfrak(0)\cdot P\rangle_{X_{2}/S}\\
         &\hspace{0.4cm}\simeq  [\chfrak(i_{!}^{\prime}\varphi^{\ast} E)]_{X_{1} /S}(h^{\ast} P).
    \end{split}
\end{displaymath}
The right-hand side Riemann--Roch distribution satisfies a similar sequence of isomorphisms. With this understood, the proof in this case follows the same pattern as in the birational case, and we omit the details.
\end{proof}

\begin{proof}[Proof of Proposition \ref{prop:compositionofinclusions}]
   By Proposition \ref{prop:RR-closed-immersions}, we may assume that all the immersions are given by closed subschemes. We will consider the deformation to the normal cone for both $i'$ and $i \circ i',$ denoting them by $M$ and $M'.$ By Proposition \ref{prop:doubleinclusion} \eqref{doubleinclusion-item1} there is an induced map 
    \begin{displaymath}
        \PBbb^1_Z \overset{j'}{\to} M_{i'} \overset{j}{\to} M_{i \circ i'},
    \end{displaymath}
    where $j$ and $j'$ are regular closed immersions. For later use, we note that by a similar argument as for $M_{i'}\to M_{i\circ i'}$, the morphism between blowups $\widetilde{Y}\to\widetilde{X}$ is also a regular closed immersion of $S$-schemes.
    
    As in the proof of Lemma \ref{lemma:ignorepart}, evaluating the difference between $\RR_{j'}\RR_j$ and $\RR_{j'  j}$ along a Chern power series defines an invertible function on $\PBbb^1_S$, i.e. on $S$. This constant function on $\PBbb^1_S$ can be specialized to either 0 or $\infty$. The value at 0 is the comparison of the two DRR isomorphisms we started with. 
    
    We aim to show that the value at $\infty$ is 1, in which case the sequence of inclusions to consider is 
\begin{displaymath}
        Z\overset{i'}{\longrightarrow}\widetilde{Y}\cup\PBbb(N_{i'}^{\vee}\oplus 1)\overset{i}{\longrightarrow} \widetilde{X}\cup\PBbb(N_{ii'}^{\vee}\oplus 1).
    \end{displaymath}
    Here, by an abuse of notation, we continue to denote the immersions by $i$ and $i^{\prime}$. Towards our goal, we consider the diagram 
    \begin{equation}\label{eq:cartesian-diagrams-ignorepart}
    \xymatrix{
        Z \ar[d] \ar[r] & \widetilde{Y} \sqcup \PBbb(N_{i'}^{\vee} \oplus 1) \ar[d]^\varphi \ar[r]^{k} & \widetilde{X}\sqcup\PBbb(N_{ii'}^{\vee}\oplus 1) \ar[d]^\psi  \\ 
        Z \ar[r] &  \widetilde{Y} \cup \PBbb(N_{i'}^{\vee} \oplus 1) \ar[r]^i & \widetilde{X}\cup\PBbb(N_{ii'}^{\vee}\oplus 1),
        }
    \end{equation}
    where $\varphi$ and $\psi$ denote the natural maps. They are birational in the sense of Theorem \ref{thm:cor86}\eqref{item:prop-int-dist-birational}. Moreover, as remarked above, all the upper and lower closed immersions are regular. 
    
    We want to compare the DRR isomorphisms of the lower and the upper compositions. Assume for the time being that the leftmost, rightmost and outer squares of the diagram \eqref{eq:cartesian-diagrams-ignorepart} fulfill the requirements of Lemma \ref{lemma:ignorepart}. A straightforward application of the lemma then shows that the statement of the current proposition for the lower composition follows from the analogous statement for the upper composition. Hence, the statement to be proven is now reduced to the composition
    \begin{displaymath}
        Z \to \widetilde{Y} \sqcup \PBbb(N_{i'}^{\vee} \oplus 1) \to \widetilde{X}\sqcup\PBbb(N_{ii'}^{\vee}\oplus 1).
    \end{displaymath}
    Since the image of $Z$ does not intersect $\widetilde{Y}$ and $\widetilde{X}$, we can apply Lemma \ref{lemma:ignorepart} once again, and we can disregard $\widetilde{Y}$ and $\widetilde{X}$. We are hence in the setting of Proposition \ref{prop:linearcase}, which thus allows to conclude. 

    It remains to show that we can apply Lemma \ref{lemma:ignorepart} to the squares in \eqref{eq:cartesian-diagrams-ignorepart}. The only point which requires an explanation is that the squares are Cartesian. For the leftmost square, there is nothing to prove, since $Z$ does not intersect $\widetilde{Y}$. For the rightmost square, we recall from Proposition \ref{prop:Deformationcone} that the schemes $\widetilde{Y} \cup \PBbb(N_{i'}^\vee \oplus 1)$ and $\widetilde{X}\cup\PBbb(N_{ii'}^{\vee}\oplus 1)$ can be realized as pushouts in the category of schemes. We can then apply \cite[\href{https://stacks.math.columbia.edu/tag/0ECK}{0ECK}]{stacks-project} to conclude that the rightmost square is Cartesian. It also follows that the outer square is Cartesian.
\end{proof}

        \subsection{General statement and characterization}

We now summarize the goals we have achieved in the previous subsections, to the effect of proving that our construction is a DRR isomorphism, defined as in Definition \ref{def:DRRisoforclassofmorphisms}. Recall our running assumptions that $S$-schemes are assumed to satisfy the condition $(C_{n})$, for some $n$.

\begin{theorem}\label{thm:DRRi-general}
There exists a unique DRR isomorphism for regular closed immersions $i:Y \to X$ of $S$-schemes
    \begin{displaymath}
        \RR_{i}(E)\colon [\chfrak(i_! E)]_{X/S} \to i_\ast [\chfrak(E) \cdot \tdfrak(N_{i})^{-1}]_{Y/S},
    \end{displaymath}
    satisfying the following properties:
    \begin{enumerate}  
        \item\label{item:DRR-4} Birational invariance in the sense of Lemma \ref{lemma:ignorepart}, when $\varphi$ in \eqref{eq:cartesian-birational-invariance} is the identity and $\psi$ is an isomorphism in a neighborhood of $i(Y)$.
        \item\label{item:DRR-5} If $Y \to X$ is a relative effective Cartier divisor, the isomorphism $\RR_{i}(\Ocal_{Y})$ is given by $\RR_{sec}(\Ocal_{X},1)$ of Construction/Definition \ref{construction-definition-RRsec} (see also \textsection\ref{subsubsec:RRimmersionD}), where $1$ is the canonical section of $\Ocal_{X}(Y)$.
    \end{enumerate}
\end{theorem}

\begin{proof}
    The existence statement has been established in the previous sections. Now for the uniqueness claim. Without loss of generality, we can restrict to divisorial schemes and work with virtual vector bundles. A deformation to the normal cone argument as in the proof of Lemma \ref{lemma:ignorepart} shows that any DRR isomorphism for closed immersions is determined by the linearized case $Y\to\PBbb(N^{\vee}\oplus 1)$. Here $N$ is a vector bundle of constant rank on $Y$. This reduction uses the compatibility with base change and with the projection formula, and the birational invariance in \eqref{item:DRR-4}. In this case, $E$ extends to $\PBbb(N^{\vee}\oplus 1)$. By the splitting principle, we may assume that $N$ admits a complete flag: $N=N_{r}\supset\cdots\supset N_{1}\supset 0$. We can then decompose the closed immersion as a composition
    \begin{displaymath}
        Y \overset{i_{1}}{\longrightarrow} \PBbb(N_1^\vee \oplus 1) \overset{i_{2}}{\longrightarrow} \cdots\overset{i_{r}}{\longrightarrow} \PBbb(N_{r}^\vee \oplus 1),
    \end{displaymath}
    where every step is an inclusion by a relative effective Cartier divisor. By induction, because these are all linear embeddings, we see that for  $1\leq k\leq r-1$, the virtual vector bundle $(i_{k} i_{k-1} \cdots i_{1})_{!}E$ extends to $\PBbb(N_{k+1}^{\vee}\oplus 1)$. By the compatibility with composition, we can then reduce to the case where $Y\to X$ is a closed immersion of a relative effective Cartier divisor and $E$ extends to $X$. By the compatibility with the projection formula, we may assume that $E$ is the trivial rank one bundle. In this case, the isomorphism is then fixed by \eqref{item:DRR-5}.
\end{proof}

In Proposition \ref{prop:crazy-diagrams}, based on the above characterization, we will prove that the DRR isomorphism for regular closed immersions is compatible with Grothendieck duality.

\section{Construction for projective bundles}\label{section:projective-bundles}
In this section, we construct and study the DRR isomorphism for projective bundles $\PBbb(\Ecal)\to X$. We begin by stating the main result:

\begin{theorem}\label{thm:RR-P(E)}
Let $X\to S$ be a morphism of schemes satisfying the condition $(C_{n})$. Let $\Ecal$ be a vector bundle of constant rank $r\geq 1$ on $X$, and let $\pi\colon \PBbb(\Ecal) \to X$ be the associated projective bundle. Then, there exists a DRR isomorphism for $\pi$
\begin{equation}\label{eq:RR-iso-PE}
    \RR_{\pi}(E): [\chfrak(\pi_! E)]_{X/S} \to \pi_{\ast} \left[\chfrak(E) \cdot \tdfrak(T_\pi) \right]_{\PBbb(\Ecal)/S}
\end{equation}
satisfying the following properties:
 \begin{enumerate}  
        \item\label{item:RRPE-1} If $\pi$ is the identity map, then $\RR_{\pi}$ is the identity isomorphism.
        \item\label{item:RRPE-2} Compatibility with the projection formula.
        \item\label{item:RRPE-3} Compatibility with composition of projective bundles.
        \item\label{item:RRPE-4} Functoriality with respect to isomorphisms of projective bundles, namely commutative diagrams of the form
        \begin{equation}\label{eq:diagram-proj-bundles}
            \xymatrix{
                \PBbb(\Ecal^{\prime})\ar[r]^{\pi^{\prime}}\ar[d]^{\sim}_{\varphi}   &X^{\prime}\ar[d]^{\psi}_{\sim}\\
                \PBbb(\Ecal)\ar[r]^{\pi}        &X,
            }
        \end{equation}
        where the vertical arrows are isomorphisms.
    \end{enumerate}
\end{theorem}
Recall from \textsection \ref{sec:notations-conventions} that we identify $\PBbb(\Ocal_{X})$ with $X$ and view $\pi$ as the identity map. This is only a notational convenience: without this identification, property \eqref{item:RRPE-1} would instead require $\RR_{\pi}$ to coincide with the DRR isomorphism for closed immersions in Theorem \ref{thm:DRRi-general} when $\pi$ is an isomorphism.

The construction leading to the proof of the theorem is addressed in the following subsection.

\subsection{The construction}\label{subsec:the-construction}
Before proceeding with the construction of \eqref{eq:RR-iso-PE}, we record the following lemma: 

\begin{lemma}\label{lemma:projbundle}
Let $X$ be a quasi-compact scheme, and let $\Ecal$ be a vector bundle of constant rank $r$ on $X$. Then, there is an equivalence of commutative Picard categories
    \begin{equation}\label{eq:equivalence-virtual-projective}
        \bigoplus_{i=0}^{r-1} V(X) \simeq V(\PBbb(\Ecal)),
    \end{equation}
    induced by 
    \begin{equation}\label{eq:equiv-virtual-proj}
        (G_0, \ldots, G_{r-1}) \mapsto \sum \pi^{\ast} G_{i}\otimes\Ocal(-i),
    \end{equation}
    which is compatible with pullback by morphisms of quasi-compact schemes $X^{\prime}\to X$. Moreover, one can choose inverses to the functors \eqref{eq:equivalence-virtual-projective} compatibly with pullback functoriality.
\end{lemma}

\begin{proof}
It is straightforward to verify that \eqref{eq:equiv-virtual-proj} induces a functor of commutative Picard categories which is compatible with pullback functoriality.

To show that \eqref{eq:equivalence-virtual-projective} defines an equivalence of categories, by \cite[Lemma 2.2]{DRR1} it suffices to verify that it induces an isomorphism on $\pi_0$ and $\pi_1$; see also the reminder on virtual categories in \textsection\ref{intro:virtualcategories}. This in turn corresponds to the analogous maps $\bigoplus_{i=0}^{r-1} K_i(X) \to K_i(\PBbb(\Ecal))$, for $i=0,1$. These are isomorphisms by the projective bundle formula, cf.  \cite[Proposition 4.3]{Quillen:K-theory-I}. 

For the last claim, we note that $V(X)$ and $V(\PBbb(\Ecal))$ define categories fibered in groupoids over quasi-compact $X$-schemes. Here we follow the notation convention fixed at the beginning of \textsection\ref{subsubsec:prelim-on-base-change}. Since \eqref{eq:equivalence-virtual-projective} is compatible with pullback, by \cite[\href{https://stacks.math.columbia.edu/tag/003Z}{003Z}]{stacks-project} we can construct an inverse compatible with pullbacks as well. 
\end{proof}

We go back to the setting of Theorem \ref{thm:RR-P(E)}. In order to construct the DRR isomorphism, we can first make several simplifications. By Proposition \ref{prop:base-reduction}, we can restrict to working with divisorial schemes and virtual categories of vector bundles. In particular, our schemes are now quasi-compact. Then, in view of Lemma \ref{lemma:projbundle} and by the additivity of the categorical Chern character, the above lemma allows us to reduce the construction of \eqref{eq:RR-iso-PE} to objects $E$ of the form $ \pi^\ast G \otimes \Ocal(-i)$, for $0 \leq i \leq r-1$ and $G$ some vector bundle on $X$. By imposing the compatibility with the projection formulas from Proposition \ref{prop:proj-for-RR-dist}, we finally reduce to the cases $\Ocal(-i)$ for $i = 0, \ldots, r-1$. We will tackle those by rendering both sides of \eqref{eq:RR-iso-PE} explicit, starting with the left-hand side. 

The forthcoming statements no longer require a divisorial or quasi-compactness assumption on the base schemes, although we may as usual reduce to this setting in the proofs. 

\begin{proposition}\label{prop:RR-proj-left}
There are natural isomorphisms of line distributions,

\begin{equation}\label{eq:determinantprojectivebundle}
     [\chfrak(\pi_!(  \Ocal(-i)))]_{X/S}  \simeq \begin{cases}
         \delta_{X/S} ,  & \hbox{ if }\; i=0, \\
          \Ocal_S, & \hbox{ if }\; 0 < i \leq r-1,
     \end{cases}
\end{equation}
where we consider $\Ocal_S$ as a constant functor. These are compatible with base change, and have bounded denominators.

\end{proposition}
\begin{proof}
In order for $\pi_{!}$ to be defined at the level of virtual categories, we reduce to the divisorial and hence quasi-compact setting, by Proposition\ref{prop:base-reduction}. By the very definition of $\pi_!$ in \eqref{eq:pushfwd-1}, the object $\pi_! \Ocal(-i)$ is given by the complex $R\pi_{\ast}\Ocal(-i)$. This is acyclic for $0<i\leq r-1$, and is naturally isomorphic to $\Ocal_{X}$ for $i=0$. In the case $i=0$, we use \cite[Proposition 9.1 (2)]{DRR1}, which asserts that $[\chfrak(\Ocal_{X})]\simeq 1$. The case $0<i\leq r-1$ amounts to the additivity of the categorical Chern character. The commutativity with base change follows from the Tor-independent base-change property of virtual categories, and the boundedness of denominators is clear.

\end{proof}

For the second part of the construction, we will need to adapt, to the functorial setting, the classical rewriting of Chern classes on the right-hand side of the Grothendieck--Riemann--Roch theorem for projective bundles. Our argument is inspired by that of Franke \cite{Franke}, which in turn goes back to a computation of Howe, cf. \cite[Lemma 2.3]{FultonLang}. Concretely, we are led to simplify the following expression: 
\begin{equation}\label{eq:Eulerconsequence} 
   [\chfrak(\Ocal(-i)) \cdot \tdfrak(\pi^{\ast}\Ecal^{\vee}(1))]_{\PBbb(\Ecal)/S}.
\end{equation}
First, we consider the Chern classes $\cfrak_{k}(\pi^{\ast}\Ecal^{\vee}(1))$. For $k>r$, by the rank triviality property in Theorem \ref{thm:cor86} \eqref{item:prop-int-dist-rank}, we have a canonical isomorphism
 \begin{displaymath}
    [\cfrak_k(\pi^\ast \Ecal^{\vee}(1))]_{\PBbb(\Ecal)/S} \simeq \Ocal_S.
\end{displaymath}
For $k\leq r$, by \cite[Proposition 8.12]{DRR1}, there is a canonical isomorphism 
\begin{equation} \label{eq:binomialdevelopment1}
[\cfrak_k(\pi^\ast \Ecal^{\vee} (1))]_{\PBbb(\Ecal)/S} \simeq \sum_{j=0}^k  {r-k+j \choose j}[\cfrak_{k-j}(\pi^\ast \Ecal^{\vee}) \cdot \cfrak_1(\Ocal(1))^j ]_{\PBbb(\Ecal)/S},
\end{equation}
and in particular
\begin{equation} \label{eq:binomialdevelopment2}
    [\cfrak_r(\pi^\ast \Ecal^{\vee}(1))]_{\PBbb(\Ecal)/S} \simeq \sum_{j=0}^r  [\cfrak_{r-j}(\pi^\ast \Ecal^{\vee}) \cdot \cfrak_1(\Ocal(1))^j ]_{\PBbb(\Ecal)/S}.
\end{equation}
 Also, since the Euler sequence equips $\pi^\ast \Ecal^{\vee}(1)$ with an everywhere non-vanishing canonical section, one finds by Theorem \ref{thm:cor86}\eqref{item:prop-int-dist-restriction} that 
 \begin{equation}\label{eq:vanishingsectiontrivial}
    [\cfrak_r(\pi^\ast \Ecal^{\vee}(1))]_{\PBbb(\Ecal)/S} \simeq \Ocal_S.
\end{equation}
Combining this with \eqref{eq:binomialdevelopment2}, we infer inductively that for $k\geq r$, the line distribution $[\cfrak_{1}(\Ocal(1))^k]_{\PBbb(\Ecal)/S}$ is expressed as a universal polynomial of degree at most $r-1$ in $\cfrak_{1}(\Ocal(1))$, whose coefficients are themselves polynomials in $\cfrak_{j}(\pi^{\ast}\Ecal)$, for $j\leq r$, and the rank $r$. 

Finally, we use the canonical isomorphism $[\chfrak(\Ocal(-i))]_{\PBbb(\Ecal)/S}\simeq [\exp(-i\cfrak_{1}(\Ocal(1)))]_{\PBbb(\Ecal)/S}$, which uses the rank triviality in Theorem \ref{thm:cor86} \eqref{item:prop-int-dist-rank}, we expand $\tdfrak(\pi^{\ast}\Ecal^{\vee}(1))$ as a Chern power series in the $\cfrak_{k}(\pi^{\ast}\Ecal^{\vee}(1))$, and we apply the above procedure to simplify the latter. We conclude that the expression \eqref{eq:Eulerconsequence} is canonically isomorphic to 
\begin{equation}\label{eq:premierePIL}
     \sum_{\ell=0}^{r-1} [P_{i, \ell}(\pi^\ast \Ecal) \cdot \cfrak_1(\Ocal(1))^\ell]_{\PBbb(\Ecal)/S},
\end{equation}
for some universal polynomial $P_{i, \ell}(\Ecal)$ in the Chern classes of $\Ecal$.

The expression $P_{i, r-1}(\Ecal)$ was computed by R. Howe, see \cite[Lemma 2.3]{FultonLang}:
\begin{equation}\label{eq:Pir-1}
    P_{i,r-1}(\Ecal)= 
    \begin{cases}
        1, &\hbox{  if }\;  i=0,\\ 
        0, &\hbox{  if }\;  0 < i \leq r-1.
    \end{cases}
\end{equation}
The statement in \emph{op. cit.} is given in terms of Chern roots, and then one obtains \eqref{eq:Pir-1} by the splitting principle. The fact that this polynomial identity can be lifted to a canonical isomorphism follows from the fact that all the operations above commute, as stated in Theorem \ref{thm:cor86}. 

\begin{corollary}\label{corollary:RHS-RR-PB}
   There is a canonical isomorphism of line distributions
    \begin{displaymath}
        [\chfrak(\Ocal(-i)) \cdot \tdfrak(\pi^{\ast}\Ecal^{\vee}(1))]_{\PBbb(\Ecal)/S}\simeq \sum_{\ell=0}^{r-1} [P_{i, \ell}(\pi^\ast \Ecal) \cdot \cfrak_1(\Ocal(1))^\ell]_{\PBbb(\Ecal)/S},
    \end{displaymath}
    where $P_{i,r-1}(\Ecal)$ is given by \eqref{eq:Pir-1}.  
\end{corollary}
\qed

With this in mind, we are now ready to complete the analogue of Proposition \ref{prop:RR-proj-left} for the right-hand side of the DRR isomorphism for projective bundles. 
\begin{proposition}\label{prop:RR-proj-right}
    There are natural isomorphisms of line distributions 
\begin{equation}\label{eq:RHS-projectivebundle}
     \pi_{\ast}[\chfrak( \Ocal(-i)) \cdot \tdfrak(T_{\pi})]_{\PBbb(\Ecal)/S}  \simeq \begin{cases}
         \delta_{X/S} ,  & \hbox{ if }\; i=0, \\
          \Ocal_S, & \hbox{ if }\; 0 < i \leq r-1,
     \end{cases}
\end{equation}
which are compatible with base change and have bounded denominators.
\end{proposition}
\begin{proof}
  
First, by the Euler sequence \begin{equation}\label{eq:Euler}
    0\to \Ocal_{\PBbb(\Ecal)}\to \pi^{\ast}\Ecal^{\vee}(1)\to T_{\pi}\to 0
\end{equation}
and the multiplicativity of the categorical Todd genus, we have a canonical isomorphism
\begin{displaymath}
    [\chfrak( \Ocal(-i)) \cdot \tdfrak(T_{\pi})]\simeq [\chfrak(\Ocal(-i)) \cdot \tdfrak(\pi^{\ast}\Ecal^{\vee}(1))].
\end{displaymath}
The latter is described in Corollary \ref{corollary:RHS-RR-PB}, as a sum of terms $[P_{i, \ell}(\pi^\ast \Ecal) \cdot \cfrak_{1}(\Ocal(1))^\ell]$. By the projection formula of Theorem \ref{thm:cor86} \eqref{item:cor86-1}, we obtain that 
\begin{equation} \label{eq:projformulaPil}
    \pi_\ast [P_{i, \ell}(\pi^\ast \Ecal) \cdot \cfrak_{1}(\Ocal(1))^\ell]_{\PBbb(\Ecal)/S} \simeq \Ocal_S, 
\end{equation}
unless $\ell = r-1$, in which case it is isomorphic to 
\begin{displaymath}
        [P_{i,r-1}(\Ecal)]_{X/S}.
\end{displaymath}
The claimed isomorphism now follows from the explicit evaluation provided by \eqref{eq:Pir-1}. The compatibility with base change follows from the compatibility of all the involved intermediate constructions, and similarly for the boundedness of the denominators. 
\end{proof}

We are now in a position to construct the DRR isomorphism for projective bundles.

\begin{proof}[Proof of Theorem \ref{thm:RR-P(E)} \eqref{item:RRPE-1}--\eqref{item:RRPE-2}] We define $\RR_\pi(\Ocal(-i))$ as the isomorphism obtained by the conjunction of the isomorphisms in Proposition \ref{prop:RR-proj-left} and Proposition \ref{prop:RR-proj-right}. As explained after Lemma \ref{lemma:projbundle}, this is extended to the general case by imposing compatibility with the projection formula. The rest of the properties of the statement of Theorem \ref{thm:RR-P(E)} are addressed in the following subsections, together with some additional properties needed in Section \ref{sec:construction-DRR}. Concretely, the compatibility with the composition of projective bundles is the content of Proposition \ref{prop:projprojprojproj}, and the compatibility with isomorphisms of projective bundles is considered in Corollary \ref{cor:invariance-isom-proj-bundles}.
\end{proof}
\subsection{Invariance of projective bundles} 

We now consider the invariance of the DRR isomorphism for projective bundles under twisting by line bundles. We maintain the assumptions and notation of the previous subsections. In particular, we work with divisorial schemes.

Let $L$ be a line bundle on $X$. Then, for $\Ecal'=\Ecal\otimes L$, we consider $\pi': \PBbb(\Ecal \otimes L) \to X$, and note that there is a natural $X$-isomorphism $p\colon\PBbb(\Ecal)\to\PBbb(\Ecal^{\prime})$. We denote the respective tautological line bundles by $\Ocal(1)$ and $\Ocal^{\prime}(1)$. They are related by a natural isomorphism $p^{\ast}\Ocal^{\prime}(1)\simeq \Ocal(1) \otimes \pi^{\ast} L^{\vee}$.  

Since $\pi = \pi' \circ p$, by the projection formula we have 
\begin{displaymath}  [\chfrak (\pi_! p^\ast E^{\prime})]_{X/S} \simeq [\chfrak (\pi'_! E')]_{X/S}.
\end{displaymath}
Since $p^\ast T_{\pi'} \simeq T_{\pi}$, we also have
\begin{displaymath} \pi_\ast [\chfrak (p^\ast E^{\prime}) \cdot \tdfrak(T_\pi)]_{\PBbb(\Ecal)/S} \simeq  \pi'_\ast [\chfrak (E^{\prime}) \cdot \tdfrak(T_{\pi'})]_{\PBbb(\Ecal')/S}.
\end{displaymath}
It hence makes sense to ask if these isomorphisms interchange the DRR isomorphisms. 
\begin{proposition}\label{prop:vectorbundleinvariance}
With the notation as above, the pullback along $p$ induces a natural identification of the DRR isomorphisms $\RR_{\pi} $ and $\RR_{\pi'}$. 
\end{proposition}
 
This will follow from Lemma \ref{lemma:lhsRRprojindependence} and Lemma \ref{lemma:rhsRRprojindependence} below, which assert that the isomorphisms used in the construction satisfy the analogous properties. 

\begin{lemma} \label{lemma:lhsRRprojindependence}
    The following diagram commutes:
\[
\begin{tikzcd}
    {[\chfrak(\pi_!(   \Ocal(-i)) \cdot \chfrak(  L^{i})]_{X/S}} 
        \arrow[dr, " \eqref{eq:determinantprojectivebundle}\cdot \chfrak(L^{i}) "]
        \arrow[d, "\simeq" '] 
    & \\ {[\chfrak(\pi_!(  p^\ast \Ocal^{\prime}(-i))]_{X/S}}  \arrow[d, "\simeq" ']  & \quad {\begin{cases}
         \delta_{X/S} ,\quad\    \hbox{ if }\; i=0, \\
          \Ocal_S,\quad\quad  \hbox{ if }\; 0 < i \leq r-1.
     \end{cases}}
        \\
    {[\chfrak(\pi'_!(  \Ocal^{\prime}(-i)) ]_{X/S}} 
    \arrow[ur, swap, "\eqref{eq:determinantprojectivebundle}"] & 
\end{tikzcd}
\]
Here: the upper vertical arrow is induced by the relationship $p^{\ast}\Ocal^{\prime}(1)\simeq \Ocal(1) \otimes \pi^{\ast} L^{\vee}$ and the projection formula from Proposition \ref{prop:proj-for-RR-dist} \eqref{item:proj-for-RR-dist-1}; the lower vertical arrow is induced by pushforward functoriality $\pi_{!}\simeq\pi^{\prime}_{!}\circ p_{!}$ and the projection formula \eqref{eq:proj-formula-virtual-cat}.

\end{lemma}

\begin{proof}
    In the case $i=0$, this follows from \cite[Proposition 9.1 (2)]{DRR1}, which states that $[\chfrak(\Ocal)]\simeq 1$, in a way which is compatible with the multiplicative structure of $\chfrak$ with respect to the tensor product. See also the reminder in \textsection\ref{subsubsec:mult-chern-Borel-Serre}. The case $i\neq 0$ follows a similar logic, instead relying on \cite[Proposition 9.1 (1)]{DRR1}, which states that the multiplicativity of the Chern character is bimonoidal, and further states that it is compatible with multiplication by 0.
\end{proof}
Similarly to Lemma \ref{lemma:lhsRRprojindependence}, we consider the following compatibility with the construction in Proposition \ref{prop:RR-proj-right}. In the formulation, we use the natural isomorphism $p^\ast T_{\pi'} \simeq T_{\pi}$.
\begin{lemma} \label{lemma:rhsRRprojindependence}
    The following diagram commutes:

\begin{displaymath}
\begin{tikzcd}
    {{\pi}_\ast[\chfrak(\Ocal(-i)) \cdot \chfrak({\pi^\ast L^{i}}) \cdot \tdfrak(T_{\pi})]_{\PBbb(\Ecal)/S}} 
        \arrow[dr, " \eqref{eq:RHS-projectivebundle}\cdot \chfrak(L^{i}) "]
        \arrow[d, "\simeq" '] 
    & \\    {{\pi}_\ast[\chfrak(p^\ast 
 \Ocal^{\prime}(-i)) \cdot \tdfrak(p^{\ast} T_{\pi^{\prime}})]_{\PBbb(\Ecal)/S}} \arrow[d, "\simeq" ']  & \quad {\begin{cases}
         \delta_{X/S} ,\quad\    \hbox{ if }\; i=0, \\
          \Ocal_S,\quad\quad  \hbox{ if }\; 0 < i \leq r-1.
     \end{cases}}
        \\
    {\pi'_\ast[\chfrak(\Ocal^{\prime}}(-i)) \cdot \tdfrak(T_{\pi'})]_{\PBbb(\Ecal')/S}
    \arrow[ur, swap, "\eqref{eq:RHS-projectivebundle}"] & 
\end{tikzcd}
\end{displaymath}
Here: the upper vertical arrow is given by the relationships $p^{\ast}\Ocal^{\prime}(1)\simeq\Ocal(1) \otimes \pi^{\ast} L^{\vee}$, $p^\ast T_{\pi'} \simeq T_{\pi}$ and the multiplicativity of $\chfrak$ under tensor product; the lower vertical arrow is induced by the identity $\pi_{\ast}=\pi^{\prime}_{\ast}\circ p_{\ast}$ and the projection formula for intersection distributions \eqref{eq:proj-for-int-dist}.  
\end{lemma}

\begin{proof}
We need to compare the constructions leading to Proposition \ref{prop:RR-proj-right} for $\Ecal$ and $\Ecal^{\prime}$. We will write $\pi^\ast \Ecal^{\vee}(1)$ for $\pi^{\ast}\Ecal^{\vee}\otimes \Ocal(1)$ and ${\pi'}^{\ast}\Ecal^{\prime\vee}(1)$ for ${\pi'}^\ast \Ecal^{\vee}\otimes \Ocal^{\prime}(1)$. 

Note that we have a natural isomorphism
\begin{equation}\label{eq:pullbackO1s}
p^{\ast} {\pi'}^\ast \Ecal^{\prime\vee}(1) \simeq \pi^{\ast}(\Ecal^{\vee} \otimes L^{\vee}) \otimes (\Ocal(1) \otimes \pi^{\ast} L)\simeq \pi^\ast \Ecal^{\vee}(1),
\end{equation}
and that the pullback of the Euler sequence \eqref{eq:Euler} on $\PBbb(\Ecal')$ identifies with that on $\PBbb(\Ecal).$ This also means that the natural non-vanishing section of ${\pi'}^\ast \Ecal^{\prime\vee}(1)$ pulls back to that of $\pi^\ast \Ecal^{\vee}(1)$. Since by Theorem \ref{thm:cor86} the trivialization determined by the non-vanishing section commutes with the projection formula, this means that the trivialization of \eqref{eq:vanishingsectiontrivial} commutes with pullback along $p$. Thus, the only possible difference in the construction arises in how one develops \eqref{eq:binomialdevelopment1} according to the various bundles in \eqref{eq:pullbackO1s}. We claim that for line bundles $M$ and $N$, developing \eqref{eq:binomialdevelopment1} by $\Ecal^{\vee} \otimes M \otimes N$ first as  $(\Ecal^{\vee} \otimes M) \otimes N$, and then by developing $\Ecal^{\vee} \otimes M$ is the same as developing $\Ecal^{\vee} \otimes (M \otimes N)$ and then expanding according to the isomorphism $[\cfrak_1( M \otimes N)] \simeq [\cfrak_1(M) + \cfrak_1(N)]$, cf. equation \eqref{eq:chern-tensor-product-line}. This follows from the construction of \eqref{eq:binomialdevelopment1} in \cite[Proposition 8.12]{DRR1}, which proceeds by reduction to the line bundle case via the splitting principle. This statement proves that in the rearrangements leading to \eqref{eq:premierePIL}, we naturally have the following identifications of the coefficients:
\begin{displaymath}
    [\chfrak(L^{i}) \cdot P_{i, \ell}( \Ecal)]_{X/S}\simeq  [P_{i, \ell}(p^{\ast} \Ecal')]_{X/S}.
\end{displaymath}
This in turn implies the lemma. 
    
\end{proof}

\subsection{Composition of projective bundles}

Consider two vector bundles $\Ecal, \Ecal'$ on $X$ of constant rank. The natural projection $f: \PBbb(\Ecal) \times_X \PBbb(\Ecal') \to X$ factors in two ways as a composition of projective bundles: writing $\PBbb=\PBbb(\Ecal) \times_X \PBbb(\Ecal')$, we have a commutative diagram
\begin{equation}\label{eq:productdiagram}
    \xymatrix{
                &\PBbb(\Ecal)\ar[rd]^{\pi}     & \\
        \PBbb\ar[ru]^{q'}\ar[rd]_{q}\ar[rr]^{f} &   & X.\\
            &\PBbb(\Ecal^{\prime})\ar[ru]_{\pi^{\prime}}   &
    }
\end{equation}
The following proposition shows that the composition of the DRR isomorphisms according to the two possible factorizations of $f$ coincide.

\begin{proposition}\label{prop:projprojprojproj}
We have the identity of composed DRR isomorphisms $\RR_{\pi} \RR_{q'}=\RR_{\pi'} \RR_{q} $.
\end{proposition}
\begin{proof}
    Without loss of generality, we may assume that the base schemes are divisorial. We are hence in the setting of the construction of the DRR isomorphisms, in \textsection \ref{subsec:the-construction}. By Lemma \ref{lemma:projbundle} and the compatibility of the DRR isomorphisms with the projection formula, we are reduced to study the interaction of the constructions in Proposition \ref{prop:RR-proj-left} and Proposition \ref{prop:RR-proj-right}, for $\pi$ and $\pi^{\prime}$. For this, it will be useful to fix the following conventions:
    \begin{enumerate}
        \item We will write the DRR isomorphisms symbolically as the composition 
    \begin{displaymath}
        [\chfrak(\pi_! \Ocal(-i))] \simeq \delta_i \simeq \pi_\ast [\chfrak(\Ocal(-i))\cdot\tdfrak(T_\pi)],
    \end{displaymath}
    where we use the shorthand $\delta_{i}$ for the right-hand side distribution in \eqref{eq:determinantprojectivebundle}. Hence, $\delta_i$ is the distribution $[1]=\delta_{X/S}$ if $i=0$ and $[0]=\Ocal_{S}$ otherwise, i.e. the distribution associated with the Kronecker delta. 
        \item The multiplication of any line distribution $T$ with $\delta_{i}$ can be defined by setting $\delta_{i}\cdot T=T$ if $i=0$, and $\Ocal_{S}$ otherwise. Equivalently, here we interpret $\delta_{i}$ as the Chern power series associated with the Kronecker delta. We follow a similar convention for multiplication from the right. 
        \item We will write $\pi_{\ast}(\chfrak(\Ocal(-i))\cdot\tdfrak(T_\pi))$ for the Chern power series representing the distribution $\pi_\ast [\chfrak(\Ocal(-i))\cdot\tdfrak(T_\pi)]$, after performing the reductions in the proof of Proposition \ref{prop:RR-proj-right} leading to \eqref{eq:premierePIL}, for which Lemma \ref{lemma:basechangesmalldegree} \eqref{item:basechangesmalldegree-1} applies.
    \end{enumerate}
With this in mind, we find that the two possible compositions of DRR isomorphisms can naturally be written as the two outer contours of the diagram below: \smallskip

\begin{displaymath}
\resizebox{\textwidth}{!}{
\xymatrix@C-30pt@R=30pt{
& {\left[ \chfrak(f_! (q^{\prime\ast} \Ocal(-i) \otimes q^{\ast} \Ocal^{\prime}(-j))) \right]} 
  \ar@/^2pc/[d]
  \ar@/_2pc/[d]\ar@{}[d]|-{\textbf{(A)}} & \\
& {\left[ \chfrak(\pi_! \Ocal(-i)) \cdot \chfrak(\pi_{!}^{\prime}\Ocal^{\prime}(-j)) \right]}  
  \ar[dl] \ar[dr] & \\
{\delta_i \cdot  \chfrak(\pi_{!}^{\prime}\Ocal^{\prime}(-j)) } 
  \ar[d] \ar[dr] & & 
{ \chfrak(\pi_! \Ocal(-i)) \cdot \delta_j} 
  \ar[d] \ar[dl] \\
{\pi_\ast \left[ \chfrak(\Ocal(-i)) \cdot \tdfrak(T_\pi) \right] \cdot \chfrak(\pi'_!\Ocal^{\prime}(-j))} 
  \ar[d] & 
\delta_i\cdot \delta_j & 
{\chfrak(\pi_!\Ocal(-i)) \cdot \pi'_\ast \left[ \chfrak(\Ocal^{\prime}(-j)) \cdot \tdfrak(T_{\pi'}) \right]} 
  \ar[d] \\
{\pi_\ast \left[ \chfrak(\Ocal(-i)) \cdot \tdfrak(T_\pi) \right] \cdot \delta_j} 
  \ar[d] \ar[ur] & & 
{\delta_i \cdot \pi_\ast^{\prime} \left[ \chfrak(\Ocal^{\prime}(-j)) \cdot \tdfrak(T_{\pi'}) \right]} 
  \ar[d] \ar[ul] \\
{\pi_\ast \left(\chfrak(\Ocal(-i)\cdot \tdfrak(T_\pi)) \cdot q'_! \left[\chfrak(q^\ast \Ocal^{\prime}(-j))\cdot \tdfrak(q^\ast T_{\pi'})\right]\right)} 
  \ar[dr] & & 
{\pi'_\ast \left(\chfrak(\Ocal^{\prime}(-j)\cdot \tdfrak(T_{\pi'})) \cdot q_! \left[\chfrak({q'}^\ast \Ocal(-i))\cdot \tdfrak({q'}^\ast T_{\pi})\right]\right)} 
  \ar[dl] \\
& {\left[ \pi_\ast(\chfrak(\Ocal(-i)) \cdot\tdfrak(T_\pi)) \cdot \pi'_\ast(\chfrak(\Ocal^{\prime}(-j))\cdot \tdfrak(T_{\pi'}))  \right]} 
  \ar@/^2pc/[d]
  \ar@/_2pc/[d]\ar@{}[d]|-{\textbf{(B)}} & \\
& {f_\ast \left[ \chfrak(q^{\prime\ast} \Ocal(-i) \otimes q^\ast \Ocal^{\prime}(-j) ) \cdot\tdfrak(T_f) \right].} & 
}
}
\end{displaymath}
The proposition amounts to the commutativity of the diagram above. The unlabeled inner diagrams commute because the involved product is bimonoidal and the involved operations all commute with each other. The meaning of the diagrams of isomorphisms of K\"unneth-type in \textbf{(A)} and \textbf{(B)} will be further elucidated in the proof of Lemma \ref{lemma:diagramA} below, which ensures the diagrams commute.

\end{proof}

\begin{lemma}\label{lemma:diagramA}
The diagrams \emph{\textbf{(A)}} and \emph{\textbf{(B)}} commute.
\end{lemma}

\begin{proof}
We first address the diagram \textbf{(A)}. Consider the following diagram

\begin{tikzcd}[column sep=-30pt, row sep=30pt]    & \left[\chfrak(f_! \left({q'}^\ast \Ocal(-i) \otimes q^{\ast} \Ocal^{\prime}(-j)\right)) \right] \ar[dr, "f = \pi  q' "] \ar[dl, "f = \pi' q" '] \\
    \left[\chfrak(\pi'_! {q}_! \left({q'}^\ast \Ocal(-i) \otimes q^{\ast} \Ocal^{\prime}(-j)\right)\right] \ar[d, "{\substack{\text{projection}\\ \text{formula}}}" ']
    & &  \left[\chfrak(\pi_! q'_! \left({q'}^\ast \Ocal(-i) \otimes q^{\ast} \Ocal^{\prime}(-j)\right))\right] \ar[d, "{\substack{\text{projection}\\ \text{formula}}}"] \\
    \left[\chfrak(\pi'_!( {q}_! {q'}^\ast \Ocal(-i) \otimes  \Ocal^{\prime}(-j)))\right] \ar[d, "{\substack{\text{base}\\ \text{change}}}" ']
    & & \left[\chfrak(\pi_!( \Ocal(-i) \otimes q^{\prime}_! q^{\ast}\Ocal^{\prime}(-j)))\right] \ar[d, "{\substack{\text{base}\\ \text{change}}}"] \\
    \left[\chfrak(\pi'_!( {\pi'}^\ast {\pi}_! \Ocal(-i) \otimes  \Ocal^{\prime}(-j)))\right] \ar[dr]
    & &  \left[\chfrak(\pi_!(\Ocal(-i) \otimes \pi^{\ast}\pi^{\prime}_! \Ocal^{\prime}(-j)))\right] \ar[dl] \\ & \left[\chfrak(\pi_! \Ocal(-i)) \cdot \chfrak( {\pi'_!} \Ocal^{\prime}(-j))\right].
\end{tikzcd}

\smallskip

\noindent This diagram is obtained as a composition of operations in the virtual category, induced by the analogous operations in the derived category, and an application of $\chfrak$ and its multiplicative structure. The upper and lower compositions are instances of the K\"unneth formula \cite[\href{https://stacks.math.columbia.edu/tag/0FLN}{0FLN}]{stacks-project}, hence the diagram commutes. 

For diagram \textbf{(B)}, we proceed analogously, applying instead the K\"unneth-type isomorphism of Lemma \ref{lemma:basechangesmalldegree} instead. Recall that the isomorphism in \emph{loc. cit.} holds under a restriction of the degrees of the involved Chern power series. In the current setting, we recall that $\pi_\ast(\chfrak(\Ocal(-i)) \cdot\tdfrak(T_\pi))$ and its $\pi'$ counterpart are understood after the reductions in the proof of Proposition \ref{prop:RR-proj-right} leading to \eqref{eq:premierePIL}, for which we can indeed apply Lemma \ref{lemma:basechangesmalldegree} \eqref{item:basechangesmalldegree-1}.
\end{proof}

\subsection{Birational invariance}
Let $h: X' \to X$ be a birational morphism as in Theorem \ref{thm:cor86} \eqref{item:prop-int-dist-birational}, and $\Ecal$ a vector bundle on $X$ of constant rank. Consider the Cartesian square
\begin{displaymath}
    \xymatrix{
    \PBbb(h^\ast \Ecal) \ar[r]^-{h'} \ar[d]_{\pi'} & \PBbb(\Ecal) \ar[d]^{\pi} \\
        X^{\prime} \ar[r]^{h} & X.
    }
\end{displaymath}
Then, there is a diagram of morphisms that compares the birational invariance and the isomorphisms of Proposition \ref{prop:RR-proj-left}: 
\begin{displaymath}
    \xymatrix{
        h_* [\chfrak(\pi'_! {h'}^\ast \Ocal(-i))] \ar[d] \ar[r]^{\substack{\text{base}\\ \text{change}\\ \textcolor{white}{hello}} }& h_* [\chfrak(h^\ast \pi_! \Ocal(-i))]  \ar[r]^-{\substack{\text{birational}\\ \text{invariance}\\ \textcolor{white}{hello}}} &  [\chfrak(\pi_! \Ocal(-i))] \ar[d] \\ 
     h_* \delta_i \ar[rr]^-{\substack{\text{birational}\\ \text{invariance}}} & & \delta_i.
     }
\end{displaymath}
Here we recall that $\delta_{i}$ is a shorthand for the right-hand side distribution in \eqref{eq:determinantprojectivebundle}. This diagram is easily seen to commute, by keeping track of the definitions. There is a similarly defined diagram instead utilizing Proposition \ref{prop:RR-proj-right}. This also commutes, reducing to the fact that the operations in Theorem \ref{thm:cor86} commute with each other. By the construction of the DRR isomorphism for projective bundles, this easily implies the following lemma:
\begin{lemma}\label{lemma:birinvarianceprojective}
 The birational invariance isomorphism intertwines the two DRR isomorphisms for the two projective bundles $\pi$ and $\pi'$. 
\end{lemma}
\qed

\begin{corollary}\label{cor:invariance-isom-proj-bundles}
The DRR isomorphism for projective bundles is functorial with respect to isomorphisms of the projective bundles as in \eqref{eq:diagram-proj-bundles}, in the sense of Definition \ref{def:compatibilitymorphisms}. Moreover, $\RR_{\pi}$ is the identity isomorphism if $\pi$ is the identity map.
\end{corollary}
\begin{proof}
The first claim follows from Proposition \ref{prop:vectorbundleinvariance} and Lemma \ref{lemma:birinvarianceprojective}, where the latter is applied to the case of birational morphisms which are actually isomorphisms. For the second claim, the very construction of $\RR_{\pi}$ for $\Ecal=\Ocal_{X}$ readily yields the identity isomorphism. 
\end{proof}

\section{Construction of the Deligne--Riemann--Roch isomorphism}\label{sec:construction-DRR}
 In this section, we construct and characterize the DRR isomorphism for local complete intersection morphisms, thus proving the main contributions of this article, cf. Theorem \ref{thm:general-DRR} and Theorem \ref{thm:DRR-det-coh}. Several of the considerations are parallel to the strategy of proof of the classical Grothendieck--Riemann--Roch formula, as treated, for instance, in \cite{FultonLang}. However, in the functorial setting, the classical proofs do not directly carry over, since one has to keep track of numerous isomorphisms and intermediate choices, as the previous sections already illustrate. From the characterization in Theorem \ref{thm:DRR-det-coh}, we deduce that the DRR isomorphism for the determinant of the cohomology is compatible with Grothendieck duality, cf. Theorem \ref{thm:Groth-duality}. Throughout, we assume that all the $S$-schemes satisfy the condition $(C_{n})$, for some $n$, except possibly for base changes $S^{\prime}\to S$.

\subsection{Compatibilities between closed immersions and projective bundles}
The compatibilities exhibited in this subsection are compatibilities that are expected from any DRR isomorphism. In later subsections the very construction of the DRR isomorphism in the lci case will build upon these basic compatibilities. 
\begin{proposition}\label{prop:projimmersionimmersionproj}
Let $i\colon Y\to X$ be a regular closed immersion of $S$-schemes, and let $\Ecal$ be a vector bundle of constant rank on $X$. Consider the following Cartesian diagram: 
\begin{displaymath}
    \xymatrix{
        \PBbb(i^{\ast}\Ecal) \ar[d]_{p'} \ar[r]^{i'}& \PBbb(\Ecal) \ar[d]^{p} \\ 
         Y \ar[r]^{i} & X.
    }
\end{displaymath}
Then, $\RR_{i'} \RR_p = \RR_{p'} \RR_i$.
\end{proposition}

\begin{proof}
We may assume that the base schemes are divisorial, and work with virtual vector bundles instead of perfect complexes, by Proposition \ref{prop:base-reduction}. Also, we may assume that $Y$ is a closed subscheme of $X$, by Proposition \ref{prop:RR-closed-immersions}. An argument similar to that of Lemma \ref{lemma:rationalequivalence} and Proposition \ref{prop:RRres=RR} allows us to perform a deformation to the normal cone argument with respect to $Y \to X$. We can hence replace $X$ by the fiber at infinity of the deformation to the normal cone, given by $M_{\infty}=\PBbb(N^\vee \oplus 1)\cup \widetilde{X}$, so that the closed immersion $Y\to X$ is replaced by the zero section of $\PBbb(N^\vee \oplus 1)$.

    Since the zero section does not meet $\widetilde{X}$, we can apply Lemma \ref{lemma:ignorepart}, and conclude that the DRR isomorphism for the immersion of the zero section is compatible with the birational transformation $\pi: \PBbb(N^\vee \oplus 1) \sqcup \widetilde{X} \to M_{\infty}$. By Lemma \ref{lemma:birinvarianceprojective}, the same holds for $\RR_p$. This shows that we can reduce to studying two situations:
    \begin{itemize}
        \item The exceptional case where $Y = \emptyset$, corresponding to restricting to $\widetilde{X}$.
        \item The linearized case $Y \to \PBbb(N^\vee \oplus 1)$. 
    \end{itemize}
    The first case is a straightforward unraveling of the definitions, and we omit the details. In the second case, both $Y$ and also $\PBbb(i^{\ast}\Ecal)$ are cut out by a regular section, in which case  the DRR isomorphism for closed immersions is described by Proposition \ref{prop:RRres=RR}. Hence, we may assume that the virtual bundle on $\PBbb(i^{\ast}\Ecal)$ is of the form $i^{\prime\ast} (p^\ast E \otimes \Ocal(-k))$. By the compatibility with the projection formula, we may assume that $E$ is trivial. The statement then reduces to the compatibility between the projection formula and the restriction isomorphism, stated in Theorem \ref{thm:cor86}. 
\end{proof}

The following relates another compatibility between closed immersions and projective bundles, in the case of hyperplane sections. More precisely, we suppose we are given a vector bundle $\Ecal$ of rank $r$ on an $S$-scheme $X$, and a quotient bundle $\Ecal'$ of rank $r-1$. Consider the associated linear embedding $i: \PBbb(\Ecal')\to \PBbb(\Ecal)$, and denote by $\pi': \PBbb(\Ecal') \to X$ and $\pi: \PBbb(\Ecal) \to X$ the natural projections. 

\begin{proposition}\label{prop:hyperplanereduction}
The DRR isomorphism for projective bundles is compatible with the DRR isomorphism for closed immersions defined by hyperplane sections. Namely, with the above notation, we have $ \RR_i\RR_{\pi} = \RR_{\pi^{\prime}}$.
\end{proposition}
\begin{proof}
By the splitting principle and by Proposition \ref{prop:vectorbundleinvariance}, we may assume that $\Ecal=\Ecal^{\prime}\oplus\Ocal$, and the immersion $\PBbb(\Ecal^{\prime})\to\PBbb(\Ecal)$ corresponds to the hyperplane at infinity. The normal bundle $N_i$ of $i$ is then identified with $\Ocal^{\prime}(1).$

We consider the exact sequence on $\PBbb(\Ecal)$
\begin{equation}\label{eq:Oksequence}
    0 \to \Ocal(-k-1) \to \Ocal(-k) \to i_\ast \Ocal^{\prime}(-k) \to 0
\end{equation}
which identifies with the standard Koszul resolution 
\begin{equation}\label{eq:Osequence}
    0 \to \Ocal(-1) \to \Ocal \to i_\ast \Ocal \to 0,
\end{equation}
twisted by $\Ocal(-k)$. Since both sides are compatible with the projection formula, it is enough to prove the statement for bundles of the form $\Ocal^{\prime}(-k), 0 \leq k \leq r-2$. By the compatibility of the DRR isomorphisms for regular closed immersions and projective bundles with the projection formula, the proposition ultimately amounts to the commutativity of the diagram below, where we use the notation $\delta_{k}$ for the distribution associated with the Kronecker delta:

\begin{equation}\label{eq:OkOk-1}
\resizebox{\textwidth}{!}{
\xymatrix{  
    [\chfrak({\pi}_! i_! \Ocal^{\prime}(-k))] \ar[r] \ar[d]^{\eqref{eq:Oksequence}} \ar@/_13pc/[dddd]_{\RR_{\pi}}& \ar[dd] [\chfrak(\pi'_! \Ocal^{\prime}(-k))]\ar@/^5pc/[dddddddd]^{\RR_{\pi^{\prime}}} \\ 
    [\chfrak(\pi_! (\Ocal(-k) - \Ocal(-k-1)))]  \ar[d] &  \\
    \delta_k - \delta_{k+1} \ar[d]\ar[r] & \delta_k \ar[dddddd] \\
    \pi_{\ast} [\chfrak(\Ocal(-k) - \Ocal(-k-1)) \cdot\tdfrak(T_\pi)] \ar[d]^{\eqref{eq:Oksequence}}  \ar@/^3pc/[dddrdd]^{\textbf{(A)}}\ar@{}[r]|-{\text{Lemma \ref{lemma:resres=res}}}&  \\
    \pi_{\ast} [\chfrak(i_! \Ocal^{\prime}(-k)) \cdot\tdfrak(T_\pi)]  \ar[d]^{\text{projection formula}}\ar@/_13pc/[dddd]_{\RR_{i}} &  \\
         \pi_{\ast} [\chfrak(i_! \Ocal )\cdot \chfrak(\Ocal(-k)) \cdot\tdfrak(T_\pi)]\ar[d]^{\eqref{eq:Osequence}}  &  \\ 
       \pi_{\ast} [\chfrak(\Ocal-\Ocal(-1) ) \cdot\chfrak(\Ocal(-k)) \cdot\tdfrak(T_\pi)] \ar[d]^{\text{Borel--Serre}}  &  \\
        \pi_{\ast} [\cfrak_1(\Ocal(1)) \cdot \tdfrak(\Ocal(1))^{-1} \cdot \chfrak(\Ocal(-k)) \cdot\tdfrak(T_\pi)] \ar[d]^{\text{restriction}} &  \\ 
        \pi_\ast i_\ast [\chfrak(\Ocal^{\prime}(-k))\cdot \tdfrak( \Ocal^{\prime}(1))^{-1} \cdot \tdfrak(i^\ast T_\pi)] \ar[r]^-{\text{Whitney}} &  \pi'_{\ast} [\chfrak(\Ocal^{\prime}(-k))\cdot\tdfrak(T_{\pi'})].} 
}
\end{equation}
\normalsize
The outer vaulted diagrams are simply the definitions of the corresponding DRR isomorphisms. Here we have used Proposition \ref{prop:RRres=RR} to obtain the explicit description of $\RR_i$ as $\RR_{sec}$. We proceed to prove that the square diagrams commute.

The first upper square diagram naturally commutes, by the very definition in terms of the computation of direct images and the exact sequence \eqref{eq:Oksequence}. Note here that $\delta_{k+1}$ is actually the constant distribution $\Ocal_{S}$, because $k$ is in the range $0\leq k\leq r-2$, and the arrow $\delta_k-\delta_{k+1}\to\delta_{k}$ identifies with the identity map $\delta_{k}\to\delta_{k}$. 

The lower diagram is more subtle and requires some preliminary reductions. First, we claim that we can replace $T_{\pi}$ (resp. $T_{\pi^{\prime}}$) by $\pi^{\ast} \Ecal^{\vee}(1)$ (resp. $\pi^{\prime \ast} \Ecal^{\prime\vee}(1)$) in \eqref{eq:OkOk-1}. For the arrows involving $\tdfrak(T_{\pi})$, this is immediate from the Euler sequence \eqref{eq:Euler} and the Whitney isomorphism for the categorical Todd genus, which yield 
\begin{displaymath} 
    [\tdfrak(T_{\pi})]\simeq [\tdfrak(\Ocal)\cdot\tdfrak(\pi^{\ast}\Ecal^{\vee}(1))]\simeq [\tdfrak(\pi^{\ast}\Ecal^{\vee}(1))].
\end{displaymath} 
The only non-trivial point is thus the compatibility between such isomorphisms for $T_{\pi}$ and $T_{\pi^{\prime}}$ and the lowest horizontal Whitney isomorphism. For this, we consider the following diagram of bundles on $\PBbb(\Ecal')$, where the two vertical exact sequences are the Euler sequences, and the lower horizontal diagram is that of the exact sequence of tangent bundles:

\begin{equation}\label{eq:Eulersequencesinclusion}
    \xymatrix{
      \underset{\ }{\Ocal} \ar@{>->}[d] \ar@{=}[r] & \underset{\ }{\Ocal} \ar@{>->}[d] & \\
     \pi^{\prime\ast} \Ecal^{\prime\vee}(1)\ \ \ar@{->>}[d] \ar@{>->}[r] & \pi^{\prime\ast} \Ecal^\vee(1) \ar@{->>}[r] \ar@{->>}[d] & \Ocal^{\prime}(1) \ar@{=}[d] \\ 
     T_{\pi'}\ \ \ar@{>->}[r] & i^\ast T_\pi \ar@{->>}[r] & \Ocal^{\prime}(1).
    }
\end{equation}
By \cite[Lemme 4.8]{Deligne-determinant} and the multiplicativity of the categorical Todd genus, the exact sequence of exact sequences in \eqref{eq:Eulersequencesinclusion} implies the commutativity of the diagram of Whitney isomorphisms
\begin{displaymath}
    \xymatrix{&  [\tdfrak(\Ocal^{\prime}(1))^{-1}\cdot\tdfrak(\pi^{\prime\ast} \Ecal^\vee(1))] \ar[d] \ar[r] &[ \tdfrak(\pi^{\prime \ast} \Ecal^{\prime \vee}(1))]  \ar[d] \\ 
    & [\tdfrak(\Ocal) \cdot\tdfrak(\Ocal^{\prime}(1))^{-1}\cdot \tdfrak(i^\ast T_\pi)] \ar[r] \ar[d] & [\tdfrak(\Ocal) \cdot \tdfrak( T_{\pi'})] \ar[d] \\
    & [\tdfrak(\Ocal^{\prime}(1))^{-1} \cdot \tdfrak(i^\ast T_\pi)] \ar[r] & [\tdfrak( T_{\pi'})].
    }
\end{displaymath}
This allows us to replace $T_{\pi}$ (resp. $T_{\pi^{\prime}}$) by $\pi^{\ast} \Ecal^{\vee}(1)$ (resp. $\pi^{\prime \ast} \Ecal^{\prime\vee}(1)$) in the lowest Whitney isomorphism in \eqref{eq:OkOk-1}, as claimed. With this in mind, the commutativity of the lower square diagram in \eqref{eq:OkOk-1} is a consequence of Lemma \ref{lemma:resres=res} below, and the projection formula as applied to deduce Proposition \ref{prop:RR-proj-right} from Corollary \ref{corollary:RHS-RR-PB}.
\end{proof}

\begin{lemma}\label{lemma:resres=res}
    The following diagram commutes:
\footnotesize{
    \begin{displaymath}
        \xymatrix{
        [(\chfrak(\Ocal(-k)) - \chfrak(\Ocal(-k-1))) \cdot \tdfrak(\pi^{\ast} \Ecal^\vee (1))] \ar[dd]^{\eqref{eq:OkOk-1}\ \mathbf{(A)}} \ar[r] & \sum_{\ell=1}^{r-1} [\cfrak_1(\Ocal(1))\cdot \pi^\ast (P_{k, \ell}(\Ecal) - P_{k+1, \ell}(\Ecal)) \cdot \cfrak_1(\Ocal(1))^{\ell-1}] \ar[d]^{\mathrm{restriction}} \\
         & \sum_{\ell=1}^{r-1} i_\ast  [\pi^{\prime \ast} (P_{k, \ell}(\Ecal) - P_{k+1, \ell}(\Ecal)) \cdot \cfrak_1(\Ocal^{\prime}(1))^{\ell-1}] \ar[d]^{\mathrm{Whitney}} \\
         i_\ast [\chfrak(\Ocal^{\prime}(-k))\cdot \tdfrak(\pi^{\prime\ast} \Ecal^{\prime\vee} (1))] \ar[r] &  \sum_{\ell=0}^{r-2} i_\ast  [\pi^{\prime \ast} P_{k, \ell}(\Ecal') \cdot\cfrak_1(\Ocal^{\prime}(1))^{\ell}].
         }
    \end{displaymath}
}
\normalsize
Here the upper horizontal (resp. lower) arrow is given by  Corollary \ref{corollary:RHS-RR-PB}, associated with the canonical everywhere non-vanishing section of $\pi^{\ast}\Ecal^\vee(1)$ (resp. of $\pi^{\prime \ast}\Ecal^{\prime\vee}(1)$).
\end{lemma}
\begin{proof}  
The canonical section $\sigma$ of $\pi^{\ast}\Ecal^{\vee}(1)$ sits in an exact sequence
    \begin{displaymath}
        \xymatrix{& & \Ocal \ar[d]^\sigma \ar[dr]^{\sigma''} & \\
        0 \ar[r] & \pi^{\ast} \Ecal^{\prime \vee}(1) \ar[r] & \pi^{\ast} \Ecal^{ \vee}(1) \ar[r] & \Ocal(1) \ar[r]  & 0}
    \end{displaymath}
    where $\sigma''$ is simply the section of $\Ocal(1)$ induced by $\sigma$ under the natural surjection. Then, $\sigma''$ cuts out the hyperplane $\PBbb(\Ecal')$ of $\PBbb(\Ecal)$, and the restriction of $\sigma$  to $\PBbb(\Ecal')$ induces the canonical section $\sigma'$ of $\pi^{\prime \ast} \Ecal^{\prime \vee} (1)$ as in \eqref{eq:Eulersequencesinclusion}.

The horizontal arrows in the statement of the lemma are explained in the discussion preceding Corollary \ref{corollary:RHS-RR-PB}. Beyond standard manipulations such as rank triviality, the isomorphism $[\chfrak(\Ocal(-k))]\simeq [\exp(-k\cfrak_{1}(\Ocal(1))]$ and the binomial expansion \eqref{eq:binomialdevelopment1}, the key point of the construction relies on the trivializations $[\cfrak_{r}(\pi^\ast \Ecal^\vee(1))] \simeq 0$ and $[\cfrak_{r-1}(\pi^{\prime\ast} \Ecal^{\prime\vee}(1))] \simeq 0$ associated with the sections $\sigma$ and $\sigma^{\prime}$, respectively. We thus have to relate these. By \cite[Proposition 7.14]{DRR1}, they correspond to each other through the following composition of Whitney and restriction isomorphisms:
\begin{displaymath}
        [\cfrak_{r}(\pi^{\ast} \Ecal^{\vee}(1))]\simeq [\cfrak_{1}(\Ocal(1))\cdot\cfrak_{r-1}(\pi^{\ast} \Ecal^{\prime\vee}(1))]\simeq i_{\ast}[\cfrak_{r-1}(\pi^{\prime\ast} \Ecal^{\prime\vee}(1))].
\end{displaymath}
This leads to the following diagram involving the binomial expansion in the form of \eqref{eq:binomialdevelopment2}:
\begin{equation}\label{eq:cr-restriction-cr-1}
    \xymatrix{
        [\cfrak_{1}(\Ocal(1))^{r}]\ar[d]_{\text{restriction}} \ar[r]  &-\sum_{j=0}^{r-1}[\cfrak_{r-j}(\pi^{\ast}\Ecal^{\vee})\cdot\cfrak_{1}(\Ocal(1))^{j}]\ar[d]^{\text{Whitney + restriction}} \\
        i_{\ast}[\cfrak_{1}(\Ocal(1))^{r-1}] \ar[r] &-\sum_{j=0}^{r-2}i_{\ast}[\cfrak_{r-1-j}(\pi^{\prime\ast}\Ecal^{\prime\vee})\cdot\cfrak_{1}(\Ocal(1))^{j}].  
    }
\end{equation}
We assert that this diagram commutes. This is a consequence of the compatibility of the trivializations $[\cfrak_{r}(\pi^\ast \Ecal^\vee(1))] \simeq 0$ and $[\cfrak_{r-1}(\pi^{\prime\ast} \Ecal^{\prime\vee}(1))] \simeq 0$ already addressed above, and the fact that \eqref{eq:binomialdevelopment2} is compatible with the Whitney isomorphism by the construction of the latter given in \cite[Proposition 8.12]{DRR1}, and with the restriction isomorphism trivially.

To conclude, we unravel the definition of the arrow \textbf{(A)} in \eqref{eq:OkOk-1}. We distinguish two parts. The first part is the composition of the vertical arrows which take place at the level of $\PBbb(\Ecal)$, namely those with the labels \eqref{eq:Oksequence}, projection formula, \eqref{eq:Osequence} and Borel--Serre. The composition of these arrows is a simple rearrangement of the involved Chern power series. This reduces to the fact that the exact sequence \eqref{eq:Oksequence} identifies with \eqref{eq:Osequence} tensored by $\Ocal(-k)$. The second part is the composition of the lowest restriction and Whitney isomorphisms. To sum up, up to rearrangement, the arrow \textbf{(A)} is a composition of restriction and Whitney isomorphisms. This parallels the operations defining the right vertical arrows in the statement of the lemma. The commutativity of the diagram then reduces to the very construction preceding Corollary \ref{corollary:RHS-RR-PB} and the commutativity of \eqref{eq:cr-restriction-cr-1}.

\end{proof}

\begin{theorem}\label{thm:elkikfranke}
Let $X$ be an $S$-scheme, and let $\Ecal$ be a vector bundle of constant rank on $X$. Suppose that $\pi: \PBbb(\Ecal) \to X$ is a projective bundle, with a section $s: X \to \PBbb(\Ecal)$. Then, $\RR_{s} \RR_{\pi} = \id.$
\end{theorem}
\begin{proof}
    First, we note that if $E$ is a vector bundle on $X$ and $Q$ a Chern power series on $X$, we can interpret $(\RR_{s}\RR_{\pi}(E))(Q)$ as an automorphism of the line functor $E \mapsto \langle\chfrak(E)\cdot Q\rangle_{X/S}$. Hence, while the line functor itself does not depend on $\Ecal$, the automorphism a priori depends on it. The same applies to the identity to be proven. As such, they are amenable to the splitting principles for isomorphisms of line functors in $\Ecal$, similarly to the proof of \cite[Corollary 5.44]{DRR1}. 
    
    In the case of rank one, $\pi$ and $s$ are mutually inverse isomorphisms. In this case, the claim of the theorem is trivial. Actually, by the functoriality of $\RR_{\pi}$ and $\RR_{s}$ in the morphism, we may assume that $\pi$ and $s$ are the identity map. If the rank of $\Ecal$ is two, this immediately follows from Proposition \ref{prop:hyperplanereduction}. In general, we suppose the section $s$ corresponds to a  surjection $\Ecal \twoheadrightarrow L$ onto a line bundle $L$. By Proposition \ref{prop:vectorbundleinvariance} we can even suppose $L$ is trivial. By the splitting principle \cite[Theorem 2.10]{Eriksson-Freixas-Wentworth}, we can assume that the inclusion $L^\vee \subseteq \Ecal^\vee$ can be completed into a full flag $L^\vee = \Ecal_1^\vee \subset \cdots \subset \Ecal_{r}^\vee = \Ecal^\vee$, which in turn induces a sequence of hyperplane inclusions 
    \begin{displaymath}
        X \stackrel{s_1}{\longrightarrow} \PBbb(\Ecal_2) \stackrel{s_2}{\longrightarrow}  \cdots \stackrel{s_{r-2}}{\longrightarrow} \PBbb(\Ecal_{r-1})   \stackrel{s_{r-1}}{\longrightarrow}   \PBbb(\Ecal),
    \end{displaymath}
    whose composition is the section $s$. 
    
    In this case, note that by Proposition \ref{prop:compositionofinclusions} we have $\RR_s = \RR_{s_1} \cdots \RR_{s_{r-1}}$. If we denote by $\pi_j: \PBbb(\Ecal_j) \to X$ the natural projection, then by  Proposition \ref{prop:hyperplanereduction} we have $\RR_{s_j} \RR_{\pi_{j+1}} = \RR_{\pi_j}$. The statement is then proven by induction.
    \end{proof}

\begin{corollary}\label{cor:diagonal}
Consider a commutative diagram of $S$-schemes
\begin{displaymath}
    \xymatrix{
        Y\ar[r]^{j}\ar[rd]_{i}     &\PBbb(\Ecal)\ar[d]^{\pi}\\
            &X,
    }
\end{displaymath}
where $\Ecal$ is a vector bundle of constant rank on $X$, and $i$ and $j$ are regular closed immersions. Then, $\RR_{i}=\RR_{j}\RR_{\pi}$.
\end{corollary}
\begin{proof}
From the assumptions, we can form a new commutative diagram:
\begin{displaymath}
    \xymatrix{
        \PBbb(i^{\ast}\Ecal)\ar[rr]^-{i^{\prime}} \ar@/^1pc/[d]^{\pi^{\prime}} &  &\PBbb(\Ecal)\ar[d]^{\pi}\\
        Y\ar[rr]^i \ar@/^1pc/[u]^{s}\ar[urr]^{j}  &   &X.
    }
\end{displaymath}
Here, $i^{\prime}$ is the regular closed immersion deduced from $i$ by base change by $p$, $\pi^{\prime}$ is the restriction of $\pi$ to $Y$, and $s$ is the section $(j,\id_{Y})\colon Y\to\PBbb(\Ecal)\times_{X}Y$. By Proposition \ref{prop:compositionofinclusions}, we find
\begin{displaymath}
    \RR_{i}=\RR_{s}\RR_{i^{\prime}}.
\end{displaymath}   
By Proposition \ref{prop:projimmersionimmersionproj}, we also have
\begin{displaymath}
    \RR_{i^{\prime}}\RR_{\pi}=\RR_{\pi^{\prime}}\RR_{i}.
\end{displaymath}
Finally, by Theorem \ref{thm:elkikfranke} we know that $\RR_{s}\RR_{\pi^{\prime}}=\id$. We conclude by combining all these facts.
\end{proof}

\subsection{Assembling the Deligne--Riemann--Roch isomorphism}
 We now consider decompositions of lci morphisms into closed immersions and projective bundles, and we prove that the composition of the associated DRR isomorphisms does not depend on the particular factorization. This is the key result behind the construction of the DRR isomorphism.

\begin{proposition}\label{prop:independence-factorizations}
 Let $X\to Y$ be an lci morphism of $S$-schemes. Suppose given two factorizations $X\stackrel{i}{\to}\PBbb(\Ecal)\stackrel{\pi}{\to} Y$ and $X\stackrel{i^{\prime}}{\to} \PBbb(\Ecal^{\prime})\stackrel{\pi^{\prime}}{\to} Y$ into closed immersions and projective bundles, where $\Ecal$ and $\Ecal^{\prime}$ are vector bundles of constant rank on $Y$. Then, $\RR_{i}\RR_{\pi}=\RR_{i^{\prime}}\RR_{\pi^{\prime}}$.
\end{proposition}
\begin{proof}
Given the tools developed so far, the proof of the proposition now follows the standard approach in \cite{FultonLang}. We provide the details for completeness. 

Any two factorizations can be compared as in the diagram below:
\begin{displaymath}
    \xymatrix{& & & \PBbb(\Ecal)  \ar[dr]^{\pi} & \\ 
    X \ar@/^1pc/[urrr]^i \ar[rr]^-{\Delta}
    \ar@/_1pc/[drrr]_{i'} & &  \PBbb(\Ecal) \times_Y \PBbb(\Ecal') \ar[ur]^{p'} \ar[dr]_{p} & & Y. \\ 
    & & & \PBbb(\Ecal') \ar[ur]_{\pi'}
    }
\end{displaymath}
By Corollary \ref{cor:diagonal}, we have
\begin{equation}\label{eq:kindofElkikFranke} \RR_i =\RR_{\Delta} \RR_{p'}.
\end{equation}
We infer that 
\begin{displaymath}
    \RR_{i} \RR_{\pi} =  \RR_\Delta \RR_{p'} \RR_{\pi} = \RR_\Delta \RR_{p'}  \RR_{\pi},
\end{displaymath}
which by Proposition \ref{prop:projprojprojproj} identifies with 
\begin{displaymath}
    \RR_\Delta \RR_{p'} \RR_{\pi} = \RR_{\Delta} \RR_{p}  \RR_{\pi'}.
\end{displaymath}
Another application of the claim \eqref{eq:kindofElkikFranke} identifies this with $\RR_{i'}\RR_{\pi'}$. 

\end{proof}
Thanks to Proposition \ref{prop:independence-factorizations}, it now makes sense to give the following definition. 
\begin{construction-definition}\label{construction-definition:DRR}
 Let $f: X \to Y $ be an lci morphism of $S$-schemes.
 \begin{enumerate}
    \item\label{enum:def-RRf-1} If $f$ admits a projective bundle factorization $X \stackrel{i}{\to}\PBbb(\Ecal)\stackrel{\pi}{\to} Y$, for some vector bundle $\Ecal$ on $Y$ of constant rank, then we define the DRR isomorphism of $f$ as $\RR_f = \RR_i \RR_\pi$. 
    \item\label{enum:def-RRf-2} In general, suppose that $S$ is quasi-compact. Since $f$ admits a projective bundle factorization over a finite open cover of $S$, we can define $\RR_{f}$ by gluing the corresponding locally defined DRR isomorphisms as in \eqref{enum:def-RRf-1}. 
 \end{enumerate}
\end{construction-definition}
In the construction above, the quasi-compactness assumption of the second point is needed to ensure a uniform control on the denominators of the Riemann--Roch distributions, which is necessary in the gluing step. Note that the property that $f$ admits a projective bundle factorization over a finite open cover of $S$ is stable under arbitrary base change $S^{\prime}\to S$, with $S^{\prime}$ not necessarily compact.

We next establish the existence of a DRR isomorphism for lci projective morphisms in the sense of Definition \ref{def:DRRisoforclassofmorphisms}, along with a characterization of the construction. We recall our convention that all the $S$-schemes are assumed to satisfy the condition $(C_{n})$, except possibly for base changes over $S$.

\begin{theorem}\label{thm:general-DRR}
Let $S$ be a scheme. There exists a unique DRR isomorphism
    \begin{displaymath}
        \RR_f(E): [\chfrak(Rf_{\ast} E)]_{Y/S} \xrightarrow{\sim} f_\ast [\chfrak(E) \cdot \tdfrak(T_f)]_{X/S},
    \end{displaymath}
for lci projective morphisms $f:X \to Y$ of $S$-schemes (resp. lci morphisms of $S$-schemes if $S$ is quasi-compact) such that:
\begin{enumerate}
    \item In the case of regular closed immersions, it coincides with the DRR isomorphism of Theorem \ref{thm:DRRi-general}.
    \item In the case of a projective bundle $\pi\colon\PBbb(\Ecal)\to X$ with $\Ecal$ of rank $r\geq 1$, $\RR_{\pi}(\Ocal(-i))$, for $0\leq i\leq r-1$, is given by the constructions in Proposition \ref{prop:RR-proj-left} and Proposition \ref{prop:RR-proj-right}.  
\end{enumerate}
\end{theorem}
\begin{proof}
For the existence, Theorem \ref{thm:DRRi-general} and Theorem \ref{thm:RR-P(E)} prove that Construction/Definition \ref{construction-definition:DRR} yields an isomorphism which satisfies the required properties for a DRR isomorphism, except possibly the compatibility with the composition. The treatment of the latter follows the same pattern as in the theory of generalized holomorphic analytic torsion classes \cite[Theorem 7.7, pp. 42--43]{BFL2}, formally replacing the latter by the DRR isomorphisms. Due to its relevance, we next explain how to adapt the argument in our setting. 

Let $f\colon X\to Y$ and $g\colon Y\to Z$ be lci morphisms of $S$-schemes. By Proposition \ref{prop:base-reduction}, we may assume that $S$ is divisorial and that the morphisms are projective. We begin by choosing projective factorizations
\begin{displaymath}
    \xymatrix{
        X\ar@{^{(}->}[r]^-{i}\ar[dr]_-{gf}   &\PBbb(\Ecal)\ar[d]^-{p}\\
        &Z
    }
    \quad\quad
    \xymatrix{
        Y\ar@{^{(}->}[r]^-{j}\ar[dr]_-{g}   &\PBbb(\Fcal)\ar[d]^-{q}\\
        &Z.
    }
\end{displaymath}
From the first one, we deduce a projective factorization for $f$, namely
\begin{displaymath}
    \xymatrix{
        X\ar@{^{(}->}[r]^-{k}\ar[dr]_{f}   &\PBbb(g^{\ast}\Ecal)\ar[d]^-{r}\\
        &Y.
    }
\end{displaymath}
We can further factor the immersion $i$ as
\begin{displaymath}
    \xymatrix{
        X\ar@{^{(}->}[r]^-{k}\ar@{^{(}->}[dr]_-{i}      &\PBbb(g^{\ast}\Ecal)\ar[d]^-{g^{\prime}}\ar@{^{(}->}[r]^-{\ell}   &\PBbb(\Ecal)\times_{Z}\PBbb(\Fcal)\ar[d]^{q^{\prime}}\\
                &\PBbb(\Ecal)\ar@{=}[r]           &\PBbb(\Ecal).
    }
\end{displaymath}
Taking these diagrams into account, together with Corollary \ref{cor:diagonal} and Proposition \ref{prop:compositionofinclusions}, we have
\begin{displaymath}
    \RR_{gf}=\RR_{i}\RR_{p}=\RR_{k}\RR_{\ell}\RR_{q^{\prime}}\RR_{p}=\RR_{k}\RR_{g^{\prime}}\RR_{p}.
\end{displaymath}
To conclude, it suffices to show that $\RR_{g^{\prime}}\RR_{p}=\RR_{r}\RR_{g}$, because $\RR_{k}\RR_{r}=\RR_{f}$ by construction. For this, we write down the diagram
\begin{displaymath}
    \xymatrix{
        \PBbb(g^{\ast}\Ecal)\ar@{^{(}->}[r]^-{\ell}\ar[d]^-{r}        &\PBbb(\Ecal)\times_{Z}\PBbb(\Fcal) \ar[r]^-{q^{\prime}}\ar[d]^-{p^{\prime}}       &\PBbb(\Ecal)\ar[d]^-{p}\\
        Y\ar@{^{(}->}[r]^-{j}                           &\PBbb(\Fcal)\ar[r]^-{q}                              &Z,                                          
    }   
\end{displaymath}
with Cartesian squares. The upper composition is the morphism $g^{\prime}$, while the lower composition is $g$. Applying Proposition \ref{prop:projprojprojproj} first and then Proposition \ref{prop:projimmersionimmersionproj}, we find
\begin{displaymath}
    \begin{split}
        \RR_{g^{\prime}}\RR_{p}=\RR_{\ell}\RR_{q^{\prime}}\RR_{p}=&\RR_{\ell}\RR_{p^{\prime}}\RR_{q}\\
       & =\RR_{r}\RR_{j}\RR_{q}=\RR_{r}\RR_{g},
    \end{split}
\end{displaymath}
as was to be shown. This completes the proof of the existence.

For the uniqueness, by Proposition \ref{prop:base-reduction} we may assume that our base schemes are divisorial, and our morphisms projective. Then, we are led to separately treat the case of closed immersions and projective bundles. For closed immersions, there is nothing to say. For the case of projective bundles, by the compatibility of the projection formula and by Lemma \ref{lemma:projbundle}, we see that the DRR isomorphism is fixed by the second condition in the statement. 
\end{proof}

\subsection{Deligne--Riemann--Roch isomorphism for the determinant of the cohomology}
In the particular case of morphisms $f\colon X \to S$, Theorem \ref{thm:general-DRR} implies Theorem \ref{thm:A} in the Introduction, except for the compatibility with Grothendieck duality, which is provided by Theorem \ref{thm:Groth-duality} in \textsection \ref{subsec:compat-Grothendieck-duality} below. In fact, compared with Theorem \ref{thm:general-DRR}, a stronger characterization can be given which does not require any condition on projective bundles. To achieve this, one is led to consider all DRR isomorphisms over variable base schemes $S$.

In the discussion below, we follow the conventions in \textsection \ref{subsubsec:base-change-conventions} regarding base change, where instead of line distributions we deal with $\QBbb$-line bundles. This is legitimate, since $\QBbb$-line bundles over $S$ can be interpreted as line distributions in $\Dcal(S/S)$. Also, we state the theorem for general, or quasi-compact base schemes, but we note that the same result can be proven if one restricts to the category of Noetherian schemes. The Noetherian setting will be used in Proposition \ref{thm:Groth-duality} below on Grothendieck duality.

\begin{theorem}\label{thm:DRR-det-coh}
There is a unique way to associate, with any quasi-compact scheme $S$ (resp. any scheme $S$) and any lci morphism (resp. lci projective morphism) $f\colon X\to S$ satisfying the condition $(C_{n})$, an isomorphism of functors of commutative Picard categories $V(\Pcal_{X})\to\Picfr(S)_{\QBbb}$ 
\begin{displaymath}
   E\mapsto \RR_{f}(E)\colon \lambda_{f}(E)\to\langle\chfrak(E)\cdot\tdfrak(T_{f})\rangle_{X/S},
\end{displaymath}
compatible with base change and with bounded denominators, and satisfying:
\begin{enumerate}
    \item Compatibility with isomorphisms of $S$-schemes.
    \item Compatibility with the projection formula.
    \item\label{item:DRR-det-coh-3} Restriction property. Suppose that $i\colon Y\to X$ is a regular closed immersion of codimension $c$, such that $Y\to S$ satisfies the condition $(C_{n-c})$. Let $E$ be a virtual perfect complex on $Y$. Then, the diagram
    \begin{displaymath}
       \xymatrix{
         \lambda_{f}(i_{!}E) \ar[rr]^-{\RR_{f}(i_{!}E)} \ar[d]_{(f|_{Y})_{!}\simeq f_{!}\circ i_{!}}  &    & \langle \chfrak(i_{!}E) \cdot \tdfrak(T_{f}) \rangle_{X/S} \ar[d]^{\RR_{i}(E)+\mathrm{Whitney}} \\
         \lambda_{f|_{Y}}(F) \ar[rr]^-{\RR_{f|_{Y}}(E)} &  & \langle \chfrak(E) \cdot \tdfrak(T_{f|_{Y}}) \rangle_{Y/S}
}
    \end{displaymath}
    commutes. Here: the left vertical arrow is induced by pushforward functoriality; the right vertical arrow is induced by $\RR_{i}(E)$ from Theorem \ref{thm:DRRi-general}, and the Whitney-type isomorphism $[\tdfrak(T_{f|_{Y}})]\simeq [\tdfrak(i^{\ast}T_{f})\cdot\tdfrak(N_{i})^{-1}]$ (cf. Definition \ref{def:compositionofmorphisms}).
\end{enumerate}

\end{theorem}
\begin{proof}
The existence is obtained by specializing Theorem \ref{thm:general-DRR} to the case of lci morphisms $f\colon X\to S$, and then evaluating the line distributions along the Chern power series $1$.  Indeed, on the one hand, by Theorem \ref{thm:cor86} \eqref{item:prop-int-dist-c1-det} and \eqref{eq:det-coh-prelim}, and the definition of intersection distributions, we have
\begin{displaymath}
    [\chfrak(f_{!}E)]_{S/S}(1)=[\cfrak_{1}(f_{!}E)]_{S/S}\simeq \det f_{!}E=\lambda_{f}(E).
\end{displaymath}
On the other hand, again by definition of the intersection distributions, we find
\begin{displaymath}
    f_{\ast}[\chfrak(E)\cdot\tdfrak(T_{f})]_{X/S}(1)=\langle \chfrak(E)\cdot\tdfrak(T_{f})\rangle_{X/S}.
\end{displaymath}
We thus obtain a canonical isomorphism of $\QBbb$-line bundles
\begin{displaymath}
    \lambda_{f}(E)\to \langle \chfrak(E)\cdot\tdfrak(T_{f})\rangle_{X/S}.
\end{displaymath}
This isomorphism satisfies the stated properties, as follows from the properties satisfied by the general DRR isomorphism for lci morphisms that we have constructed. In the particular case of the projection formula, one needs to apply the projection formula for the Riemann--Roch distributions from Proposition \ref{prop:proj-for-RR-dist} with the projection formula from Theorem \ref{thm:cor86}.

For the uniqueness, suppose that we are given two such constructions, denoted by $\RR^{(1)}_{f}$ and $\RR^{(2)}_{f}$. Since they are compatible with base change, for the sake of comparison, we can argue locally over the base schemes. In particular, we can reduce to working with divisorial schemes, virtual categories of vector bundles, and projective morphisms (cf. Proposition \ref{prop:base-reduction}). Moreover, we can also suppose that our projective morphisms factor through relative projective spaces of the form $\PBbb^{N}_{S}\to S$, for some $N\geq 1$. 

Let $p\colon\PBbb^{N}_{S}\to S$ be a relative projective space. By the compatibility with the projection formula and by Lemma \ref{lemma:projbundle}, in order to compare  $\RR^{(1)}_{p}$ and $\RR^{(2)}_{p}$ we are reduced to comparing of $\RR^{(1)}_{p}(\Ocal(-k))$ and $\RR^{(2)}_{p}(\Ocal(-k))$, for $k$ an integer. Now, $p$ is a base change of the structure map $q\colon \PBbb^{N}_{\ZBbb}\to\Spec\ZBbb$. By the compatibility with base change, we reduce to comparing $\RR^{(1)}_{q}(\Ocal(-k))$ and $\RR^{(2)}_{q}(\Ocal(-k))$. These are isomorphisms between the same $\QBbb$-line bundles over $\Spec\ZBbb$, so that they necessarily coincide. Indeed, after taking an appropriate power, they both become isomorphisms of line bundles, which differ by a unit in $\ZBbb$, i.e. by $\pm 1$. We can then take an additional power of 2 to kill the sign ambiguity.

Consider now a factorization of a flat projective lci morphism $f\colon X\to S$ of constant relative dimension $n$, into a closed immersion $i\colon X\to\PBbb^{N}_{S}$ and the natural projection $p\colon\PBbb^{N}_{S}\to S$. Since $f$ is lci and $p$ is smooth, the immersion $i$ is regular of codimension $N-n$, cf. \cite[\href{https://stacks.math.columbia.edu/tag/069G}{069G}]{stacks-project}. Let $E$ be a virtual vector bundle over $X$. Then, we write down the following diagrams, for $k=1,2$:
\begin{equation}\label{eq:comparison-RR1-RR2}
       \xymatrix{
         \lambda_{f}(E) \ar[rr]^-{\RR_{f}^{(k)}(E)} \ar[d]  &    & \langle \chfrak(E) \cdot \tdfrak(T_{f}) \rangle_{X/S} \ar[d]^{\RR_{i}(E)+\text{Whitney}} \\
         \lambda_{p}(i_{!}E) \ar[rr]^-{\RR_{p}^{(k)}(i_{!}E)} &  & \langle \chfrak(i_{!}E) \cdot \tdfrak(T_{p}) \rangle_{P/S}.
}
    \end{equation}
The left vertical arrow is induced by pushforward functoriality $f_{!}\simeq p_{!}\circ i_{!}$. The right vertical arrow is induced by the isomorphism $\RR_{i}(E)$ from Theorem \ref{thm:DRRi-general} and by the Whitney isomorphism $[\tdfrak(T_{f})]\simeq [\tdfrak(i^{\ast} T_{p})\cdot\tdfrak(N_{i})^{-1}]$. By the assumed restriction property, the diagram commutes. Moreover, we have already shown that $\RR_{p}^{(1)}(i_{!}E)=\RR_{p}^{(2)}(i_{!}E)$. We conclude that $\RR_{i}^{(1)}(E)=\RR_{i}^{(2)}(E)$, as required. 
\end{proof}

\subsection{Deligne--Riemann--Roch isomorphism and Grothendieck duality}\label{subsec:compat-Grothendieck-duality}
In this subsection we apply the characterization from Theorem \ref{thm:DRR-det-coh} and we prove that the DRR isomorphism for the determinant of the cohomology is compatible with Grothendieck duality. The argument is partly inspired by an analogous result in the theory of generalized holomorphic analytic torsion \cite[Section 9]{BFL2}.

\subsubsection{Dual Chern power series and line distributions}
We begin with some preliminary constructions on Chern power series, that we later need to formulate the compatibility of the DRR isomorphism with Grothendieck duality. 

Let $P$ be a Chern power series on a scheme $X$. We define the dual Chern power series $P^{\ast}$ of $P$ by
\begin{displaymath}
    (P^{\ast})^{(k)}=(-1)^{k}P^{(k)},
\end{displaymath}
where the index $k$ denotes the degree-$k$ part. This operation induces a categorical graded ring endomorphism of the Chern category $\CHfrak(X)_{\QBbb}$, whose formation is compatible with pullback functoriality. 

Suppose now that $X\to S$ is a morphism of schemes over $S$, of relative dimension $n$. Then, if $P$ is a Chern power series on $X$, we clearly have
\begin{equation}\label{eq:dist-ast-1}
    \langle P^{\ast}\rangle_{X/S}=\langle P\rangle_{X/S}^{(-1)^{n+1}},
\end{equation}
because the intersection bundle only takes into account the degree-$(n+1)$ part. If $Q$ is another Chern power series, then
\begin{equation}\label{eq:star-jumps}
    \langle P^{\ast}\cdot Q\rangle_{X/S}=\langle P\cdot Q^{\ast}\rangle_{X/S}^{(-1)^{n+1}}. 
\end{equation}
If $T$ is a line distribution for $X\to S$, equation \eqref{eq:star-jumps} motivates the definition of the dual line distribution $T^{\ast}$ by
\begin{equation}\label{eq:dist-ast-2}
    T^{\ast}(P)=T(P^{\ast})^{(-1)^{n+1}}. 
\end{equation}
This construction induces an autoequivalence of categories on $\Dcal(X/S)$. 

We gather several formal properties of these constructions in a lemma, whose proof we omit.

\begin{lemma}\label{lemma:properties-ast}
Let $X\to S$ and $Y\to S$ be $S$-schemes of relative dimension $n$ and $m$, respectively. Let $f\colon X\to Y$ be an $S$-morphism. The dual Chern power series and line distributions satisfy the following properties:
\begin{enumerate}
    \item If $P$ is a Chern power series on $X$, then $[P]^{\ast}_{X/S}=[P^{\ast}]_{X/S}$.
    \item If $Q$ a Chern power series on $X$, then
    \begin{displaymath}
    (Q\cdot T)^{\ast}=Q^{\ast}\cdot T^{\ast}.
\end{displaymath}
    In particular, for every Chern power series $P$ on $X$, we have
    \begin{displaymath}
         (Q\cdot[P]_{X/S})^{\ast}=Q^{\ast}\cdot[P^{\ast}]_{X/S}.
    \end{displaymath}
    \item\label{item:properties-ast-3} If $T$ is a line distribution for $X\to S$, then
\begin{displaymath}
    (f_{\ast}T)^{\ast}=f_{\ast}(T^{\ast})^{(-1)^{d}},
\end{displaymath}
where $d=n-m$. In particular, for every Chern power series $P$ on $X$, we have
    \begin{displaymath}
        (f_{\ast}[P]_{X/S})^{\ast}=(-1)^{d}f_{\ast}[P^{\ast}]_{X/S}
    \end{displaymath}
     in additive notation.
\end{enumerate}
\end{lemma}
\qed

Next, we consider the effect of the above duality formalism on categorical characteristic classes, cf. \cite[Section 5.4.2]{DRR1}. If $\phi$ is an additive or multiplicative categorical characteristic class determined by a formal power series $\phi(x)\in\QBbb\llbracket x\rrbracket$, such as $\chfrak$ or $\tdfrak$, then for every virtual vector bundle $E$
\begin{displaymath}
    \phi(E)^{\ast}=\phi^{\ast}(E),
\end{displaymath}
where $\phi^{\ast}(E)$ is the additive or multiplicative categorical characteristic class associated with the power series $\phi^{\ast}(x)=\phi(-x)$. One hence expects a relationship between $\phi^{\ast}(E)$ and $\phi(E^{\vee})$. This is the content of the following lemma.

\begin{lemma}\label{lemma:phi-star}
Let $X\to S$ be an $S$-scheme, and $\phi$ an additive or multiplicative categorical characteristic class determined by a formal power series. Then, for every virtual vector bundle $E$ on $X$, there exists a canonical isomorphism of line distributions
\begin{equation}\label{eq:phiast}
    [\phi^{\ast}(E)]_{X/S}\simeq [\phi(E^{\vee})]_{X/S},
\end{equation}
compatible with the additive or multiplicative structure of $\phi$ and $\phi^{\ast}$. It can be composed in a natural way with the isomorphisms in Theorem \ref{thm:cor86}, and commutes with them.  
\end{lemma}
\begin{proof}
The lemma is a restatement of \cite[Proposition 8.11]{DRR1}. Only the last claim needs some justification. We note that the proof of \cite[Proposition 8.11]{DRR1} is based on the splitting principles, the Whitney isomorphism and the rank triviality isomorphism from Theorem \ref{thm:cor86}, so that one can suppose that $E=L$ is a line bundle and one can write $[\phi^{\ast}(L)]$ as a power series in $\cfrak_{1}(L)$. One then uses the isomorphism $[\cfrak_{1}(L^{\vee})]\simeq -[\cfrak_{1}(L)]$, which can also be derived from Theorem \ref{thm:cor86}. It follows from  the commutativity statement in Theorem \ref{thm:cor86} that the isomorphism \eqref{eq:phiast} can be naturally composed with the isomorphisms in that theorem, and commutes with them. 
\end{proof}
In the particular case of $\chfrak$, it easily follows from the construction in the previous lemma and the proof of \cite[Proposition 9.1]{DRR1} that \eqref{eq:phiast} is compatible with the multiplicativity behavior of $[\chfrak(\bullet)]$ with respect to the tensor product, recalled in \textsection\ref{subsubsec:mult-chern-Borel-Serre}.

Next, we consider the particular case of the duality operator $\ast$ acting on the line distribution associated with the categorical Todd genus.
\begin{lemma}\label{lemma:prelim-iso-duality}
Let $E$ be a virtual vector bundle on $X$. Then, there are canonical isomorphisms
\begin{equation}\label{eq:prelim-iso-duality}
    [\tdfrak^{\ast}(E)]_{X/S}\simeq [\chfrak(\det E^{\vee})\cdot\tdfrak(E)]_{X/S}
\end{equation}
and
\begin{equation}\label{eq:prelim-iso-duality-inverse}
   [\tdfrak^{\ast}(E)^{-1}]_{X/S}\simeq [\chfrak(\det E)\cdot\tdfrak(E)^{-1}]_{X/S}
\end{equation}
of line distributions, which are compatible with the multiplicative behavior with respect to exact sequences. They can be composed in a natural way with the isomorphisms in Theorem \ref{thm:cor86}, and commute with them.
\end{lemma}

\begin{proof}
We prove the first isomorphism, and omit the treatment of the second one, which is similar.

Recall from \textsection\ref{subsubsec:mult-chern-Borel-Serre} that $\chfrak$ behaves multiplicatively with respect to the tensor product, and $\det$ is multiplicative on exact sequences. Also, $\tdfrak$ and $\tdfrak^{\ast}$ are multiplicative on exact sequences. By the splitting principles from \cite[Section 2]{Eriksson-Freixas-Wentworth} and \cite[Section 5.4.2]{DRR1}, we may assume that $F$ is a direct sum of vector bundles, and then by the aforementioned multiplicativity we can reduce to suppose that $E=L$ has rank one. By the rank triviality isomorphism from Theorem \ref{thm:cor86}, we are led to compare Chern power series in $\cfrak_{1}(L)$. In this case, the result follows from the equality of formal power series
\begin{displaymath}
    Q(-x)=e^{-x}Q(x),\quad \text{where}\quad Q(x)=\frac{x}{1-e^{-x}}.
\end{displaymath}
For the last assertion of the lemma, it is enough to say that the so-constructed isomorphism ultimately relies on an applications of the splitting principle together with various isomorphisms in Theorem \ref{thm:cor86}, and hence it can be composed and commute with the latter, by the very same theorem.
\end{proof}

We finally consider the behavior of the right-hand side Riemann--Roch distribution with respect to the duality operator. 
    
\begin{lemma}\label{lemma:prelim-iso-duality-bis}
Let $X\to S$ and $Y\to S$ be $S$-schemes of relative dimensions $n$ and $m$, respectively. Let $f\colon X\to Y$ be an lci morphism over $S$, and define $d=n-m$ and $\omega_{X/Y}= \det T_{f}^{\vee}$. Then:
\begin{enumerate}
    \item\label{item:prelim-iso-duality-1} For every virtual perfect complex $E$ on $X$, there exists a canonical isomorphism
\begin{equation}\label{eq:prelim-iso-duality-bis-bis}
          [\chfrak(E)\cdot\tdfrak(T_{f})]^{\ast}_{X/S}\to [\chfrak(E^{\vee}\otimes\omega_{X/Y})\cdot\tdfrak(T_{f})]_{X/S},
\end{equation}
which induces an isomorphism of functors $V(\Pcal_{X})\to\Dcal(X/S)$ of commutative Picard categories (in $E$), compatible with base change and with bounded denominators. 
    \item The formation of \eqref{eq:prelim-iso-duality-bis-bis} is compatible with the composition of lci morphisms via the Whitney-type isomorphism $[\tdfrak(T_{gf})]\simeq [\tdfrak(T_{f})\cdot f^{\ast}\tdfrak(T_{g})]$ and the multiplicativity of $[\chfrak(\cdot)]$ with respect to the tensor product.
    \item The isomorphism \eqref{eq:prelim-iso-duality-bis-bis} induces an isomorphism
    \begin{equation}\label{eq:RHS-RR-duality}
    (f_{\ast}[\chfrak(E)\cdot\tdfrak(T_{f})]_{X/S})^{\ast}\to f_{\ast}[\chfrak(E^{\vee}\otimes\omega_{X/Y}[d])\cdot\tdfrak(T_{f})]_{X/S}
    \end{equation}
    of functors of commutative Picard categories $V(\Pcal_{X})\to\Dcal(Y/S)$, compatible with base change and with bounded denominators.
    \item\label{item:compat-RHS-RR-ast-projection} The isomorphism \eqref{eq:RHS-RR-duality} is compatible with the projection formula from Proposition \ref{prop:proj-for-RR-dist} \eqref{item:proj-for-RR-dist-2}. Precisely, if $F$ is a virtual vector bundle on $Y$, then there is a commutative diagram
    \begin{displaymath}
        \resizebox{\textwidth}{!}{
            \xymatrix{
                (f_{\ast}[\chfrak(E\otimes f^{\ast} F)\cdot\tdfrak(T_{f})]_{X/S})^{\ast}\ar[d]_-{\mathrm{Proposition} \ref{prop:proj-for-RR-dist} \eqref{item:proj-for-RR-dist-2}}\ar[rr]^-{\eqref{eq:RHS-RR-duality}}      &       &f_{\ast}[\chfrak(E^{\vee}\otimes f^{\ast}F^{\vee}\otimes\omega_{X/Y}[d])\cdot\tdfrak(T_{f})]_{X/S}\ar[d]^-{\mathrm{Proposition} \ref{prop:proj-for-RR-dist} \eqref{item:proj-for-RR-dist-2}}\\
                \chfrak^{\ast}(F)\cdot (f_{\ast}[\chfrak(E)\cdot\tdfrak(T_{f})]_{X/S})^{\ast}\ar[rr]^-{\eqref{lemma:phi-star}+\eqref{eq:RHS-RR-duality}}    &   &\chfrak(F^{\vee})\cdot f_{\ast}[\chfrak(E^{\vee}\otimes\omega_{X/Y}[d])\cdot\tdfrak(T_{f})]_{X/S}.
            }
            }
    \end{displaymath}
    where the lower horizontal arrow combines the isomorphism \eqref{eq:RHS-RR-duality} and the identification $[\chfrak^{\ast}(F)]\simeq [\chfrak(F^{\vee})]$ from Lemma \ref{lemma:phi-star}.
    \item\label{item:prelim-iso-duality-5} The isomorphism \eqref{eq:RHS-RR-duality} is compatible with the composition of lci morphisms. 
\end{enumerate}
\end{lemma}
\begin{proof}
 
By Proposition \ref{prop:base-reduction}, we can work over divisorial schemes and with virtual categories of vector bundles, and suppose that $f$ is projective and $T_{f}$ is a virtual vector bundle. Since $\omega_{X/Y}= (\det T_{f})^{\vee}$, the isomorphism \eqref{eq:prelim-iso-duality-bis-bis} follows from Lemma \ref{lemma:prelim-iso-duality}, together with the multiplicativity of $\chfrak$ with respect to the tensor product and the isomorphism  $[\chfrak(E^{\vee})]_{X/S}\simeq [\chfrak^{\ast}(E)]_{X/S}$ provided by Lemma \ref{lemma:phi-star}. The base change and bounded denominators claims are clear. 

The fact that \eqref{eq:prelim-iso-duality-bis-bis} is compatible with the composition of morphisms is deduced from the fact that the isomorphisms in Lemma \ref{lemma:prelim-iso-duality} commute with the isomorphisms in Theorem \ref{thm:cor86}, and because the construction of these isomorphisms is itself based on Theorem \ref{thm:cor86}. It can also be checked using the splitting principle applied to the vector bundles involved and identifying power series in $\cfrak_{1}$, together with the commutativity claim in Theorem \ref{thm:cor86}.

The isomorphism \eqref{eq:RHS-RR-duality} follows from \eqref{eq:prelim-iso-duality-bis-bis} by the identification $F[d]\simeq (-1)^{d}F$ for any virtual perfect complex $F$ and from Lemma \ref{lemma:properties-ast} \eqref{item:properties-ast-3}. 

The compatibility of \eqref{eq:RHS-RR-duality} with the projection formula can be checked by the splitting principle applied to the vector bundles involved in the diagram, and then using the constructions in Lemma \ref{lemma:phi-star} and Lemma \ref{lemma:prelim-iso-duality} based on identifying power series in $\cfrak_{1}$. 

The compatibility of \eqref{eq:RHS-RR-duality} with the composition of morphisms is a combination of the previous points and the formal properties in Lemma \ref{lemma:properties-ast}.
\end{proof}

To conclude this subsection, we remark that all the above constructions trivially exhibit a compatibility with isomorphisms of schemes, since they rely on constructions with the same property. We do not state this compatibility explicitly. 

\subsubsection{Compatibility with Grothendieck duality}
We are now in a position to state and prove the compatibility of the DRR isomorphism for the determinant of the cohomology with Grothendieck duality. Our sources for Grothendieck duality are \cite{Conrad:duality}, \cite{Grothendieck:residus}, \cite{Hartshorne:residues} and \cite[\href{https://stacks.math.columbia.edu/tag/0DWE}{0DWE}]{stacks-project}, but we mostly follow the latter. To conform with these references, we work over Noetherian base schemes, although this can likely be weakened to quasi-compact and quasi-separated schemes. We maintain the convention that all the $S$-schemes satisfy the condition $(C_{n})$, for some $n$, except possibly for base changes $S^{\prime}\to S$. 

Let $f\colon X\to Y$ be an lci morphism of $S$-schemes, where we assume that $X\to S$ and $Y\to S$ have relative dimensions $n$ and $m$, respectively. The morphism $f$ admits a relative dualizing complex, given in terms of $\omega_{X/Y}=\det T_{f}^{\vee}$ by $\omega_{X/Y}[d]$, where $d=n-m$ is possibly negative, cf. \cite[\href{https://stacks.math.columbia.edu/tag/0AU3}{0AU3}, \href{https://stacks.math.columbia.edu/tag/0BR0}{0BR0}, \href{https://stacks.math.columbia.edu/tag/0BRT}{0BRT}, \href{https://stacks.math.columbia.edu/tag/0ATX}{0ATX}]{stacks-project}. Let $E$ be a perfect complex on $X$ and $F$ a perfect complex on $Y$. They have finite Tor-amplitude by the Noetherian assumption. Grothendieck duality provides a canonical isomorphism of perfect complexes
\begin{equation}\label{eq:remainder-GD-1}
    Rf_{\ast}R\mathcal{H}om_{\Ocal_{X}}(E, Lf^{\ast}F\otimes\omega_{X/Y}[d])\simeq R\mathcal{H}om_{\Ocal_{S}}(Rf_{\ast}E, F)
\end{equation}
in the derived category of $\Ocal_{Y}$-modules, cf. \cite[\href{https://stacks.math.columbia.edu/tag/0A9Q}{0A9Q}, \href{https://stacks.math.columbia.edu/tag/0GEV}{0GEV}]{stacks-project}. This isomorphism is compatible with any base change $S^{\prime}\to S$, cf. \cite[\href{https://stacks.math.columbia.edu/tag/0AU3}{0AU3}, \href{https://stacks.math.columbia.edu/tag/0AAB}{0AAB}, \href{https://stacks.math.columbia.edu/tag/0B6S}{0B6S}, \href{https://stacks.math.columbia.edu/tag/0E2L}{0E2L}, \href{https://stacks.math.columbia.edu/tag/0B6M}{0B6M}, \href{https://stacks.math.columbia.edu/tag/0BZG}{0BZG}]{stacks-project}. To justify the base-change property from these citations, note two facts. First, if $S^{\prime}\to S$ is any morphism, and $Y^{\prime}\to S^{\prime}$ is the base change of $Y\to S$, then $X\to Y$ and $Y^{\prime}\to Y$ are Tor-independent. This uses the flatness of $X\to S$. Second, the complexes $E$, $Rf_{\ast}E$ and $F$ are perfect complexes \cite[Proposition 4.14]{DRR1}, so that they are locally quasi-isomorphic to bounded complexes of vector bundles and one can commute the $R\mathcal{H}om$ in \eqref{eq:remainder-GD-1} with the derived pullback, i.e. the natural base-change map by $X^{\prime}\to X$ for the $R\mathcal{H}om$ above (cf. \cite[\href{https://stacks.math.columbia.edu/tag/08I3}{08I3}]{stacks-project}) is an isomorphism. 

In the particular case where $F=\Ocal_{Y}$, we can write \eqref{eq:remainder-GD-1} as
\begin{equation}\label{eq:remainder-GD-2}
    Rf_{\ast}(E^{\vee}\otimes\omega_{X/Y}[d]) \simeq (Rf_{\ast}E)^{\vee},
\end{equation}
where the duality operator is understood in the derived sense. Since the complexes are perfect, the duals are locally given by the dual perfect complexes. This isomorphism induces a functorial isomorphism at the level of virtual categories, of the form
\begin{equation}\label{eq:GD-virtual}
    f_{!}(E^{\vee}\otimes\omega_{X/Y}[d])\simeq (f_{!}E)^{\vee},
\end{equation}
which is compatible with base change and the projection formula, cf. \eqref{eq:remainder-GD-1} together with \cite[\href{https://stacks.math.columbia.edu/tag/0ATN}{0ATN}]{stacks-project} and \cite[Proposition 4.12 (2)]{DRR1}. The formation of the isomorphism \eqref{eq:GD-virtual} is moreover compatible with the composition of morphisms, which in turn follows from the corresponding feature of Grothendieck duality, recast in \cite[\href{https://stacks.math.columbia.edu/tag/0ATX}{0ATX}, \href{https://stacks.math.columbia.edu/tag/0ATY}{0ATY}]{stacks-project} as the construction of a pseudo-functor $f\mapsto f^{!}$. In our case, we have $f^{!}(\bullet)\simeq Lf^{\ast}(\bullet)\otimes\omega_{X/Y}[d]$, which we have implicitly been using so far.

In the particular case that $Y=S$, further applying to \eqref{eq:GD-virtual} the determinant of the cohomology (cf. \eqref{eq:det-prel-bis} and \eqref{eq:det-coh-prelim}), we obtain a canonical isomorphism of line bundles
\begin{equation}\label{eq:GD-det}
    \lambda_{f}(E^{\vee}\otimes\omega_{X/S})^{(-1)^{n}}\simeq\lambda_{f}(E)^{-1},
\end{equation}
which is functorial in $E$, and whose formation is compatible with base change and the projection formula. 

The main result of this subsection shows that the DRR isomorphism is compatible with the isomorphism \eqref{eq:GD-det}.

\begin{theorem}\label{thm:Groth-duality}
Let $f\colon X\to S$ be an lci morphism over a Noetherian scheme $S$, satisfying the condition $(C_{n})$. Let $E$ be a virtual perfect complex on $X$. Then, the diagram
\begin{equation}\label{eq:groth-duality-diagram}
    \xymatrix{
        \lambda_{f}(E)\ar[rrr]^{\RR_{f}(E)}\ar[d]_{\eqref{eq:GD-det}}      &       &       &\langle\chfrak(E)\cdot\tdfrak(T_{f})\rangle_{X/S}\ar[d]^{\eqref{eq:prelim-iso-duality-bis-bis}}\\
        \lambda_{f}(E^{\vee}\otimes\omega_{X/S})^{(-1)^{n+1}}\ar[rrr]^{\RR_{f}(E^{\vee}\otimes\omega_{X/S})^{(-1)^{n+1}}} &   &     &\langle\chfrak(E^{\vee}\otimes\omega_{X/S})\cdot\tdfrak(T_{f})\rangle^{(-1)^{n+1}}
    }
\end{equation}   
commutes. Here, the left vertical arrow is the isomorphism \eqref{eq:GD-det} induced by Grothendieck duality, and the right vertical arrow is given by Lemma \ref{lemma:prelim-iso-duality-bis} \eqref{item:prelim-iso-duality-1}.
\end{theorem}
\begin{proof}
We define an isomorphism of functors of commutative Picard categories, denoted by $E\mapsto \RR_{f}^{\prime}(E)$, as the composition of the left vertical arrow, the lower horizontal arrow, and the inverse of the right vertical arrow in \eqref{eq:groth-duality-diagram}. We need to check that the construction satisfies the characterization from Theorem \ref{thm:DRR-det-coh}. See the remark concerning the Noetherian assumption preceding that  statement. By a variant of Proposition \ref{prop:base-reduction} in the Noetherian setting (see the introduction of \textsection\ref{subsubsec:base-change-conventions} regarding the assumptions on the base schemes), we can work with Noetherian divisorial schemes and with virtual categories of vector bundles instead of perfect complexes.  

The construction of $\RR^{\prime}_{f}$ is compatible with base change, since Grothendieck duality is compatible with base change in our setting, as recalled above. It clearly has bounded denominators and is compatible with isomorphisms of schemes. For the projection formula, Grothendieck duality is compatible with it too, see \eqref{eq:remainder-GD-1}. As for \eqref{eq:prelim-iso-duality-bis-bis}, the compatibility with the projection formula is established in Lemma \ref{lemma:prelim-iso-duality-bis} \eqref{item:compat-RHS-RR-ast-projection}. It remains to study the compatibility of $\RR_{f}^{\prime}$ with the restriction along regular closed immersions.

Suppose that $Y\to X$ is a regular closed immersion of codimension $r$, such that $Y\to S$ is flat. Hence $Y\to S$ satisfies the condition $(C_{m})$ with $m=n-r$. It admits $N_{i}[-r]$ as a dualizing complex \cite[\href{https://stacks.math.columbia.edu/tag/0B4B}{0B4B}, \href{https://stacks.math.columbia.edu/tag/0BRT}{0BR0}]{stacks-project}. For a virtual vector bundle $E$ on $Y$, we write a commutative diagram of isomorphisms
\begin{equation}\label{eq:det-coh-restriction-duality}
    \xymatrix{
        \lambda_{f}(i_{!}E)\ar[r]\ar[d]     &\lambda_{f|_{Y}}(E)\ar[d]\\
        \lambda_{f}((i_{!}E)^{\vee}\otimes\omega_{X/S})^{(-1)^{n+1}} \ar[r]^{\alpha}          &\lambda_{f|_{Y}}(E^{\vee}\otimes \omega_{Y/S})^{(-1)^{m+1}}.
    }
\end{equation}
Here, the left (resp. right) vertical arrow is given by Grothendieck duality for $f$ (resp. $f|_{Y}$). The upper horizontal arrow is given by pushforward functoriality. The lower horizontal arrow, labeled $\alpha$, is given by Grothendieck duality for the closed immersion $Y\to X$, the projection formula, and pushforward functoriality. Precisely, consider the chain of natural isomorphisms
\begin{equation}\label{eq:chain-isos-Groth-dual}
    \begin{split}
        (i_{!}E)^{\vee}\otimes\omega_{X/S}[n]\overset{\substack{\mathrm{Grothendieck}\\ \mathrm{duality}\\ \mathrm{for}\ i\\ \hspace{1cm}}}{\simeq} & i_{!}(E^{\vee}\otimes N_{i}[-r])\otimes\omega_{X/S}[n]\\
        &\underset{\substack{\hspace{1cm}\\\mathrm{projection}\\ \mathrm{formula}}}{\simeq}i_{!}(E^{\vee}\otimes N_{i}\otimes i^{\ast}\omega_{X/S}[m])
        \simeq i_{!}(E^{\vee}\otimes\omega_{Y/S}[m]).
    \end{split}
\end{equation}
Then, we apply the determinant of the cohomology $\lambda_{f}$, and use $(f|_{Y})_{!}\simeq f_{!}\circ i_{!}$. The diagram \eqref{eq:det-coh-restriction-duality} commutes by the compatibility of Grothendieck duality with the composition of morphisms recalled above. 

Every vertex of \eqref{eq:det-coh-restriction-duality} is connected to a right-hand side Riemann--Roch bundle by the DRR isomorphism, which together with \eqref{eq:det-coh-restriction-duality} gives rise to a cube. We place \eqref{eq:det-coh-restriction-duality} at the rear face. We just saw that this face commutes. The upper face commutes by the restriction property of the DRR isomorphism. The commutativity of the left and right faces are the definitions of $\RR_{f}^{\prime}$ and $\RR_{f|_{Y}}^{\prime}$. It remains to show the commutativity of the lower face and the front face.

The lower face amounts to the following square:
\begin{equation}\label{eq:first-last-diag-duality}
    \xymatrix{
        \lambda_{f}((i_{!}E)^{\vee}\otimes\omega_{X/S})^{(-1)^{n+1}}\ar[d]_{\alpha}\ar[r]     &  
         \langle\chfrak((i_{!}E)^{\vee}\otimes\omega_{X/S})\cdot\tdfrak(T_{f})\rangle_{X/S}^{(-1)^{n+1}}\ar[d]^{\beta}  \\
        \lambda_{f|_{Y}}(E^{\vee}\otimes\omega_{Y/S})^{(-1)^{m+1}}\ar[r]      &\langle\chfrak(E^{\vee}\otimes\omega_{Y/S})\cdot\tdfrak(T_{f|_{Y}})\rangle_{Y/S}^{(-1)^{m+1}},      
    }
\end{equation}
where the horizontal arrows are the corresponding DRR isomorphisms, and the arrow $\beta$ is obtained by applying $[\chfrak(\cdot)]_{X/S}$ to \eqref{eq:chain-isos-Groth-dual}, followed by the DRR isomorphism for closed immersions (cf. Theorem \ref{thm:DRRi-general}) and the Whitney isomorphism $[\tdfrak(T_{f|_{Y}})]\simeq [\tdfrak(i^{\ast}T_{f})\cdot \tdfrak(N_{i})^{-1}]$. The commutativity of \eqref{eq:first-last-diag-duality} is ensured by the functoriality of the DRR isomorphism in the bundle, and by the restriction property. 

The front face involves only the right-hand side Riemann--Roch bundles:
\begin{equation}\label{eq:second-last-diag-duality}
    \xymatrix{
        \langle\chfrak(i_{!}E)\cdot\tdfrak(T_{f})\rangle_{X/S}\ar[r]^{\substack{\RR_{i}(E)\\ +\\ \mathrm{Whitney}}}\ar[d]_{\eqref{eq:prelim-iso-duality-bis-bis}}  &\langle\chfrak(E)\cdot\tdfrak(T_{f|_{Y}})\rangle_{Y/S}\ar[d]^{\eqref{eq:prelim-iso-duality-bis-bis}}\\
        \langle\chfrak((i_{!}E)^{\vee}\otimes\omega_{X/S})\cdot\tdfrak(T_{f})\rangle_{X/S}^{(-1)^{n+1}}\ar[r]^{\beta}     &\langle\chfrak(E^{\vee}\otimes \omega_{Y/S})\cdot\tdfrak(T_{f|_{Y}})\rangle_{Y/S}^{(-1)^{m+1}},
    }
\end{equation}
where $\beta$ is the same as in \eqref{eq:first-last-diag-duality}. To conclude the proof, we need to show that \eqref{eq:second-last-diag-duality} commutes. This is a consequence of the more general Proposition \ref{prop:crazy-diagrams} below, to the effect that $\RR_{i}$ is compatible with Grothendieck duality.
\end{proof}

\begin{lemma}\label{lemma:GD-birational}
In the setting \eqref{eq:cartesian-birational-invariance}, suppose furthermore that $\varphi$ is the identity and $\psi$ is an isomorphism in a neighborhood of $i(Y)$. Then, for every virtual perfect complex $E$ on $Y$, the Grothendieck duality isomorphism $(i_{!}E)^{\vee}\simeq i_{!}(E^{\vee}\otimes N_{i}[-r])$ commutes with the pullback by $\psi$.
\end{lemma}
\begin{proof}
It is enough to verify the statement for \eqref{eq:remainder-GD-1} with $f=i$. Since $\psi$ is an isomorphism in a neighborhood of $i(Y)$, in particular $i$ and $\psi$ are Tor-independent. One concludes as for the commutativity of \eqref{eq:remainder-GD-1} with base change, namely by \cite[\href{https://stacks.math.columbia.edu/tag/0E2L}{0E2L}, \href{https://stacks.math.columbia.edu/tag/0B6M}{0B6M}, \href{https://stacks.math.columbia.edu/tag/0BZG}{0BZG}]{stacks-project} and because all the complexes are perfect, so that $L\psi^{\ast}$ commutes with $R\mathcal{H}om$. Alternatively, one can reason more concretely as follows. One can argue locally over $X$, cf. \cite[\href{https://stacks.math.columbia.edu/tag/0A9P}{0A9P}]{stacks-project}. On a suitable open neighborhood of $i(Y)$, it is clear that $L\psi^{\ast}$ commutes with \eqref{eq:remainder-GD-1}, since $\psi$ is an isomorphism. On an open subset disjoint from $i(Y)$, both sides of \eqref{eq:remainder-GD-1} are naturally trivial, and the analogous fact holds after pulling back by $\psi$, so there is nothing to say. This concludes the proof of the lemma.
\end{proof}

\begin{proposition}\label{prop:crazy-diagrams}
Let $X\to S$ and $Y\to S$ be morphisms of $S$-schemes, of relative dimensions $n$ and $m$, respectively. Let $i\colon Y\to X$ be a regular closed immersion of codimension $r=n-m$. Then, for every virtual perfect complex $E$ on $Y$, the diagram
\begin{equation}\label{eq:GD-immersion}
    \xymatrix{
        [\chfrak(i_{!}E)]_{X/S}^{\ast}\ar[rr]^-{\RR_{i}(E)^{\ast}}\ar[d]_{\eqref{eq:phiast}}  &   &(i_{\ast}[\chfrak(E)\cdot\tdfrak(N_{i})^{-1}]_{Y/S})^{\ast}\ar[d]^{\eqref{eq:RHS-RR-duality}}\\
        [\chfrak((i_{!}E)^{\vee})]_{X/S}\ar[rr]^-{\substack{\RR_{i}(E^{\vee}\otimes N_{i}[-r])\\ \vspace{0.1cm}}}   &  &i_{\ast}[\chfrak(E^{\vee}\otimes N_{i}[-r])\cdot\tdfrak(N_{i})^{-1}]_{Y/S}
    }
\end{equation}
commutes. Here, the lower horizontal arrow is induced by $\RR_{i}(E^{\vee}\otimes N_{i}[-r])$ after applying Grothendieck duality $(i_{!}E)^{\vee}\simeq i_{!}(E^{\vee}\otimes N_{i}[-r])$.
\end{proposition}
\begin{proof}
Applying first the operation $\ast$ to diagram \eqref{eq:GD-immersion}, the composition of the resulting left vertical arrow, the lower horizontal arrow and the inverse of the right vertical arrow, provides a DRR isomorphism for closed immersions. Indeed, from the axioms of a DRR isomorphism (Definition \ref{def:DRRisoforclassofmorphisms}), the only non-trivial facts to check are the base change property, the compatibility with the composition of immersions and the compatibility with the projection formula. These properties are derived from the corresponding properties of Grothendieck duality and from Lemma \ref{lemma:phi-star} and Lemma \ref{lemma:prelim-iso-duality-bis}. By the characterization provided in Theorem \ref{thm:DRRi-general}, we are led to treat the birational invariance therein and the commutativity of \eqref{eq:GD-immersion} for the immersion of a relative effective Cartier divisor and the trivial vector bundle of rank one. The argument for the birational invariance is formal and we omit the details. We just indicate that it is obtained as a combination of the birational invariance of $\RR_{i}$ together with Lemma \ref{lemma:ignorepart}, Lemma \ref{lemma:phi-star}, Lemma \ref{lemma:prelim-iso-duality}, Lemma \ref{lemma:prelim-iso-duality-bis}, Lemma \ref{lemma:GD-birational} and the constructions in the respective proofs. One also needs the analogue of the isomorphism \eqref{eq:statement-ignore-1} in Lemma \ref{lemma:ignorepart} for $[\chfrak((i_{!}E)^{\vee})]_{X/S}$, established in a similar way.

Henceforth, we focus on the commutativity of \eqref{eq:GD-immersion} in the case of a relative effective Cartier divisor $D$ and $E=\Ocal_{D}$. Hence $\RR_{i}$ is given by $\RR_{sec}(\Ocal_{X},1)$, where $1$ is the canonical section of $\Ocal(D)$. The latter is described in \textsection\ref{subsubsec:RRimmersionD}. We will decompose \eqref{eq:GD-immersion} into two subdiagrams, according to the Borel--Serre and restriction isomorphisms in \textsection\ref{subsubsec:RRimmersionD}. In preparation, we begin with some observations. By the resolution $\Ocal(-D)\to\Ocal_{X}$ of $i_{\ast}\Ocal_{D}$, we have  natural isomorphisms
\begin{equation}\label{eq:i!E}
    i_{!}\Ocal_{D}\simeq \Ocal_{X}-\Ocal(-D)
\end{equation}
and
\begin{equation}\label{eq:i!Edual}
    (i_{!}\Ocal_{D})^{\vee}\simeq \Ocal_{X}-\Ocal(D).
\end{equation}
Grothendieck duality for $i$ is described in terms of \eqref{eq:i!Edual} as (cf. \cite[\href{https://stacks.math.columbia.edu/tag/0B4A}{0B4A}]{stacks-project})
\begin{equation}\label{eq:i!EGroth}
   \begin{split}
    (i_{!}\Ocal_{D})^{\vee}\overset{\eqref{eq:i!Edual}}{\simeq}& \Ocal_{X}-\Ocal(D)\simeq \Ocal(D)[-1]\otimes (\Ocal_{X}-\Ocal(-D))\\
    &\overset{\eqref{eq:i!E}}{\simeq} \Ocal(D)[-1]\otimes i_{!}\Ocal_{D}\simeq i_{!}i^{\ast}(\Ocal(D)[-1])= i_{!}(N_{i}[-1]).
    \end{split}
\end{equation}
This can be used to render the lower horizontal arrow in \eqref{eq:GD-immersion} explicit.

Taking into account \eqref{eq:i!E} and \eqref{eq:i!Edual}, the first subdiagram of \eqref{eq:GD-immersion} is

\begin{equation}\label{eq:first-piece}
 \xymatrix{
        [\chfrak(\Ocal_{X}-\Ocal(-D))]_{X/S}^{\ast}\ar[rr]^-{\text{Borel--Serre}}\ar[d]_{\eqref{lemma:phi-star}}  &  &[\cfrak_{1}(\Ocal(D))\cdot\tdfrak(\Ocal(D))^{-1}]_{X/S}^{\ast}\ar[d]^{\eqref{eq:phiast}+\eqref{eq:prelim-iso-duality-inverse}}\\
        [\chfrak(\Ocal_{X}-\Ocal(D))]_{X/S}\ar[rr]^-{\text{Borel--Serre}}  &    &[\chfrak(\Ocal(D))\cdot\cfrak_{1}(\Ocal(D))\cdot\tdfrak(\Ocal(D))^{-1}]_{X/S}^{-1}.
    } 
\end{equation}
Here, the left vertical arrow corresponds to the left vertical arrow in \eqref{eq:GD-immersion}. The right vertical arrow combines \eqref{eq:phiast} and \eqref{eq:prelim-iso-duality-inverse}. We claim that this diagram commutes. By the very construction of the Borel--Serre isomorphism, Lemma \ref{lemma:phi-star} and Lemma \ref{lemma:prelim-iso-duality}, after using the rank triviality isomorphism and $[\cfrak_{1}(\Ocal(-D))]\simeq -[\cfrak_{1}(\Ocal(D))]$, we reduce to check an equality of power series in $\cfrak_{1}$, which is easily seen to hold.

The second subdiagram is
\begin{equation}\label{eq:second-piece}
    \xymatrix{
        [\cfrak_{1}(\Ocal(D))\cdot\tdfrak(\Ocal(D))^{-1}]_{X/S}^{\ast}\ar[d]_{\eqref{eq:phiast}+\eqref{eq:prelim-iso-duality-inverse}}\ar[rr]^-{\text{restriction along } D}    &   & (i_{\ast}[\tdfrak(N_{i})^{-1}]_{D/S})^{\ast}\ar[d]^{\eqref{eq:prelim-iso-duality-inverse}+\text{Lemma} \ref{lemma:properties-ast} \eqref{item:properties-ast-3}} \\
        [\chfrak(\Ocal(D))\cdot\cfrak_{1}(\Ocal(D))\cdot\tdfrak(\Ocal(D))^{-1}]_{X/S}^{-1}\ar[rr]^-{\substack{\text{restriction along } D\\ \vspace{0.1cm}}}    &   &i_{\ast}[\chfrak(N_{i}[-1])\cdot\tdfrak(N_{i})^{-1}]_{D/S}.\\
    }
\end{equation}
The left vertical arrow is the same as the right vertical arrow in \eqref{eq:first-piece}. The right vertical arrow uses \eqref{eq:prelim-iso-duality-inverse} and Lemma \ref{lemma:properties-ast} \eqref{item:properties-ast-3} to commute $i_{\ast}$ and the duality operator. The latter produces a minus sign, corresponding to the codimension of $D\to X$, which is hidden in the shift in $N_{i}[-1]$. This diagram commutes by Lemma \ref{lemma:phi-star} and Lemma \ref{lemma:prelim-iso-duality}, which state that the involved operations commute with restriction along $D$.

Finally, by the construction of \eqref{eq:RHS-RR-duality} and the description of Grothendieck duality in \eqref{eq:i!EGroth}, we see that \eqref{eq:GD-immersion} for $i\colon D\to X$ and $E=\Ocal_{D}=i^{\ast}\Ocal_{X}$ is the concatenation of the subdiagrams \eqref{eq:first-piece} and \eqref{eq:second-piece}. These commute, thus completing the proof.
\end{proof}

In contrast with Proposition \ref{prop:crazy-diagrams}, we did not verify whether the DRR isomorphism for projective bundles constructed in Section \ref{section:projective-bundles} is compatible with Grothendieck duality, although this seems likely to be the case. If this were true, the general DRR isomorphism of Theorem \ref{thm:general-DRR} would also be compatible with Grothendieck duality. Since the characterization in Theorem \ref{thm:DRR-det-coh} does not require any condition on projective bundles, but only regular closed immersions, Proposition \ref{prop:crazy-diagrams} suffices to prove Theorem \ref{thm:Groth-duality}. It would be interesting to know if the DRR isomorphism for projective bundles is indeed compatible with Grothendieck duality. 

\section{Consequences}\label{sec:compatibilities}
In this section we provide several applications of the DRR isomorphism. On the one hand, we can recover or complete existing results in the literature, such as the original DRR isomorphism for curves, the Knudsen--Mumford expansion, or the key formula of Moret-Bailly. On the other hand, we prove new results, such as intersection bundle expressions for the determinant of the de Rham cohomology requested by Saito and Terasomo \cite{Saito-Terasoma}, or the BCOV isomorphism conjectured in \cite{cdg3} in connection with mirror symmetry. 

\subsection{Relationship with other works}
We discuss how our results relate to previous works, such as the original Deligne--Riemann--Roch isomorphism for families of curves and the Knudsen--Mumford expansion of the determinant of the cohomology of the powers of a line bundle.

\subsubsection{The Deligne--Riemann--Roch isomorphism for curves}
As recalled in the Introduction, Deligne established a functorial Riemann--Roch-type isomorphism for families of curves. See \cite[Appendice C]{Deligne:letter-Quillen} and \cite{Deligne-determinant}. Let $f\colon C\to S$ be a smooth proper morphism of relative dimension one, whose geometric fibers are connected and of genus $g$. In particular, $f$ satisfies the condition $(C_{1})$. Let $E$ be a vector bundle on $C$. Then, the original Deligne--Riemann--Roch isomorphism provides a canonical isomorphism of line bundles

\begin{equation}\label{eq:DRRforcurvesgeneral}
    \lambda_f(E)^{12} \simeq \langle \omega_{C/S}, \omega_{C/S} \rangle_{C/S}^{\rk E} \otimes \langle \det E, \det E \otimes \omega^{-1}_{C/S} \rangle^{6}_{C/S} \otimes \langle \cfrak_2(E) \rangle^{-12}_{C/S}.
\end{equation}
The formation of this isomorphism is compatible with base change, and it is naturally multiplicative on exact sequences. It is built out of three isomorphisms:

\begin{enumerate}
    \item $\lambda_f(\Ocal_C)^{12} \simeq \langle \omega_{C/S}, \omega_{C/S} \rangle_{C/S}$.
    \item For a line bundle $L$ on $X$, 
    \begin{displaymath} 
    \lambda_f(L-\Ocal_C)^{2} \simeq \langle L, L \otimes \omega^{-1}_{C/S} \rangle_{C/S}.
    \end{displaymath}
    \item For a vector bundle $E$, there is an isomorphism 
    \begin{displaymath}
        \lambda_f(E) \simeq \lambda_f(\Ocal_C)^{\rk E-1} \otimes \lambda_f(\det E) \otimes \langle \cfrak_2(E) \rangle^{-1}_{C/S}.
    \end{displaymath}
\end{enumerate}
The first one is a reformulation of Mumford's isomorphism when $g \geq 2$, recalled in the Introduction. For $g=0,1$, Deligne  performs special constructions in \cite[Appendice C]{Deligne:letter-Quillen}. The second construction combines an interpretation of Deligne pairings via the determinant of the cohomology (cf. \textsection \ref{subsubsec:Ducrot} below) with Grothendieck duality. The last isomorphism is one of the equivalent presentations of $\langle \cfrak_2(E) \rangle_{C/S}$ discussed in \cite{Deligne-determinant}. We note that Deligne's definition of this bundle is different from Elkik's approach \cite{Elkikfib} that we use here. However, they are both canonically isomorphic, cf. \cite[Theorem 3.22]{Eriksson-Freixas-Wentworth}.  

On the other hand, in this setting, Theorem \ref{thm:A} can be simplified by means of the isomorphisms in Theorem \ref{thm:cor86}, and then it takes the same shape as \eqref{eq:DRRforcurvesgeneral}, as an isomorphism of $\QBbb$-line bundles. They in fact agree, as we now prove.

\begin{proposition}\label{prop:comparison-with-Deligne}
  For families of smooth projective curves over a quasi-compact base as above, the isomorphism in Theorem \ref{thm:A} coincides with the original Deligne--Riemann--Roch isomorphism  \eqref{eq:DRRforcurvesgeneral}, as isomorphisms of $\QBbb$-line bundles. 
\end{proposition}

\begin{proof} 
We denote by $\RR_{C/S}(E)$ our isomorphism, and by $\RR_{C/S}^{\prime}(E)$ Deligne's original one. The proof proceeds by a uniqueness argument.

By the splitting principles from \cite[Section 5.4]{DRR1}, one can reduce to the line bundle case $E=L$. Since the base scheme is quasi-compact, for the purpose of comparing isomorphisms of $\QBbb$-line bundles, we may work both Zariski and \emph{fpqc} locally over $S$. We may therefore assume: (1) the fiberwise degree of $L$ is constant and equal to $d$; (2) if $g=0$, then $C\to S$ is isomorphic to $\PBbb^{1}_{S}\to S$, and in particular it admits three disjoint sections, which we fix; ($2^{\prime}$) if $g=1$, then $C\to S$ admits a section, which we fix. For the first property, it suffices to localize for the Zariski topology. For (2) and ($2^{\prime}$), note that after base-changing by $C\to S$, which is \emph{fpqc}, we have the diagonal section. In the case where $g=0$ and $C\to S$ has a section, the family is Zariski-locally (over $S$) isomorphic to $\PBbb^{1}_{S}$.

Denote by $\Mcal_{g,n}$ the Deligne--Mumford stack over $\ZBbb$ of curves of genus $g$ with $n$ marked points \cite{Knudsen:Mgn-1}, where we assume that $n=0$ if $g\geq 2$, $n=1$ if $g=1$, and $n=3$ if $g=0$. Let $\Ccal$ be the universal curve over $\Mcal_{g,n}$. Let also $\Pcal=\Pic^d(\Ccal/\Mcal_{g,n})$ be the Deligne--Mumford stack parametrizing line bundles of fiberwise degree $d$ on $\Ccal$. Then, our geometric data, given by the family of curves with the chosen sections together with the line bundle $L$, provides a classifying map $\varphi\colon S\to\Pcal$. 

Over $\Pcal$, we can define universal isomorphisms $\RR_{\Ccal/\Pcal}(\Lcal)$ and $\RR_{\Ccal/\Pcal}^{\prime}(\Lcal)$, where we continue to denote by $\Ccal$ the pullback to $\Pcal$ of the universal curve and by $\Lcal$ the universal line bundle. This is possible by the base-change functoriality of both constructions, and since $\Ccal\to\Mcal_{g,n}$ is representable. Then, $\RR_{C/S}(L)$ and $\RR_{C/S}^{\prime}(L)$ are the pullbacks by $\varphi$ of these universal isomorphisms. We must thus compare the universal isomorphisms. 

For an appropriate integer $N\geq 1$, the isomorphisms $\RR_{\Ccal/\Pcal}(\Lcal)^{N}$ and $\RR_{\Ccal/\Pcal}^{\prime}(\Lcal)^{N}$ are both isomorphisms of line bundles, and they can be compared. They differ by a nowhere-vanishing invertible regular function, which descends to $\Mcal_{g,n}$, since $\Pcal\to\Ccal$ and $\Ccal\to\Mcal_{g,n}$ are smooth and proper, with geometrically connected fibers. In all the cases under consideration, the only such  functions defined over $\ZBbb$ are $\lbrace\pm 1\rbrace$, by \cite[Lemme 2.3.3]{Moret-Bailly} if $g\geq 1$, and since $\Mcal_{0,3}\simeq\Spec\ZBbb$ if $g=0$. Taking an additional square power of the universal isomorphisms, we conclude that they coincide. 
\end{proof}

\subsubsection{Moret-Bailly's formule clé}
Let $S$ be a quasi-compact scheme and $f\colon A\to S$ an abelian scheme of relative dimension $g$, which we assume is locally projective over $S$. It thus satisfies the condition $(C_{g})$. We denote by $e\colon S\to A$ the zero section and we define $\underline{\omega}_{A/S}=\det e^{\ast}\Omega_{A/S}$. The following statement is a generalization of the \emph{formule cl\'e canonique} of Moret-Bailly. We refer to \cite[Chapitre VIII, Section 1.2]{Moret-Bailly:pinceaux} for the definition and also Appendice 2, Remarque 1.5 in \emph{op. cit.} where potential generalizations are suggested, as the one we propose. 

\begin{proposition}
With the assumptions as above, let $L$ be a line bundle on $A$, which is symmetric and rigidified along the zero section. Then, there exists a canonical isomorphism of $\QBbb$-line bundles
\begin{equation}\label{eq:MB-key-formula-statement}
    \lambda_{f}(L)^{2}\simeq\underline{\omega}_{A/S}^{-d},
\end{equation}
which is compatible with base change. Here, $d$ is the fiberwise Euler characteristic of $L$.
\end{proposition}
\begin{proof}
We apply Theorem \ref{thm:A} to $L$. To simplify the right-hand side of the isomorphism, we combine the canonical isomorphism $T_{f}\simeq f^{\ast}e^{\ast}T_{f}$ with several properties from Theorem \ref{thm:cor86}: the rank triviality isomorphism for $L$, the first Chern class isomorphism for $T_{f}$ and the projection formula. We also use the isomorphism $[\cfrak_{1}(M^{\vee})]\simeq -[\cfrak_{1}(M)]$ from Lemma \ref{lemma:phi-star} and \cite[Proposition 9.11]{DRR1}. We obtain canonical isomorphisms, compatible with base change, 
\begin{equation}\label{eq:MB-key-formula}
    \lambda_{f}(L)\simeq \langle\chfrak(L)\cdot f^{\ast}\tdfrak(e^{\ast}T_{f})\rangle_{A/S}\simeq \underline{\omega}_{A/S}^{\kappa}\otimes \langle L,\ldots, L\rangle_{A/S},
\end{equation}
where
\begin{displaymath}
    \kappa=-\frac{1}{2g!}\int_{A/S} c_{1}(L)^{g}.
\end{displaymath}
The integral is evaluated fiberwise, and its value is given by the Grothendieck--Riemann--Roch theorem, resulting in $\kappa=-d/2$. To conclude, we will show that the Deligne pairing $\langle L,\ldots, L\rangle_{A/S}$ can be canonically trivialized. 

Let $n\geq 2$ be an integer. Then, we have canonical isomorphisms, compatible with base change,
\begin{equation}\label{eq:MB-1}
    \langle L,\ldots, L\rangle_{A/S}^{n^{2g+2}}\simeq \langle L^{n^{2}},\ldots,L^{n^{2}}\rangle_{A/S}\simeq
    \langle [n]^{\ast}L,\ldots,[n]^{\ast}L\rangle_{A/S},
\end{equation}
where we use that $L$ is symmetric and rigidified to ensure the existence of a canonical isomorphism $[n]^{\ast}L\simeq L^{n^{2}}$. Since $A\to S$ is locally projective of finite presentation, and $[n]$ is finite and flat, we deduce that $[n]$ satisfies the condition $(C_{0})$. We can thus apply the projection formula from Theorem \ref{thm:cor86}, and we find a canonical isomorphism
\begin{equation}\label{eq:MB-2}
     \langle [n]^{\ast}L,\ldots,[n]^{\ast}L\rangle_{A/S}\simeq \langle L,\ldots,L\rangle_{A/S}^{n^{2g}},
\end{equation}
compatible with base change. Here the exponent $n^{2g}$ arises as the degree of the isogeny $[n]$. Combining \eqref{eq:MB-1} and \eqref{eq:MB-2}, we conclude that there is a canonical isomorphism
\begin{displaymath}
    \langle L,\ldots,L\rangle_{A/S}^{n^{2}}\simeq \Ocal_{S}.
\end{displaymath}
Specializing to coprime values of $n$, such as $2$ and $3$, we deduce that the Deligne pairing is canonically trivial, in a way compatible with base change. 
 
\end{proof}

When the line bundle in the proposition is further assumed to be relatively ample and $S$ is Noetherian, we have $R^{k}f_{\ast}L=0$ for $k>0$ and \eqref{eq:MB-key-formula} becomes
\begin{equation}\label{eq:qisoformulecle}
    (\det f_{\ast}L)^{2}\simeq \underline{\omega}_{A/S}^{-d },
\end{equation}
where $d$ is the rank of $f_{\ast}L$. This compares with the original result of Moret-Bailly, cf. \cite[Chapitre VIII, Th\'eor\`eme 3.4]{Moret-Bailly:pinceaux} in the abelian case. Concretely, Moret-Bailly needs to suppose that $d$ is invertible on $S$. Hence, Theorem \ref{thm:A} allows one to bypass this assumption, as well as to drop the positivity condition on $L$. The key advantage over \cite{Moret-Bailly} is that we do not need any moduli-space argument. 

 We also mention the result of \cite{MR:abelian-schemes} which determines that the two sides of \eqref{eq:qisoformulecle}, to the 12th power, are equal integrally in the Picard group of $S$, i.e. without tensoring with $\QBbb$. We do not know if an isomorphism such as \eqref{eq:qisoformulecle} can be made integrally so that it also commutes with base change, with the same power.

\subsubsection{The Knudsen--Mumford expansion and higher CM line bundles}

Suppose $f: X\to S$ is a projective, flat and finitely presented morphism of schemes of relative dimension $n$, endowed with a line bundle $L$. In this setting, it follows from \cite{KnudsenMumford} (cf. \cite[Corollary A.23]{Boucksom-Eriksson} for the generalization of the assumptions, ultimately relying on the work of \cite{Ducrot}), that there exist natural line bundles $\mathcal{M}_\ell, \ell = 0, \ldots, n+1$ on $S$, and a canonical isomorphism
\begin{equation}\label{eq:KnudsenMumfordiso}
    \lambda_{f}(L^k) \simeq \bigotimes_{\ell=0}^{n+1} \mathcal{M}_\ell^{k \choose \ell},
\end{equation}
referred to as the Knudsen--Mumford expansion. It is moreover asserted that the dominant term $\Mcal_{n+1}$ is related to the Chow divisor of the family. In \cite{Zhangheights} there is a canonical isomorphism
\begin{equation}\label{eq:dom-coeff-KM}
    \Mcal_{n+1} \simeq \langle L, \ldots, L \rangle_{X/S}.
\end{equation}
If $X$ and $S$ are assumed to be smooth quasi-projective varieties, in \cite{PhongSturmRoss} the second term is shown to satisfy a canonical isomorphism
\begin{equation}\label{eq:subdom-coeff-KM}
    \Mcal_{n}^{2}\simeq \langle L^{n} K_{X/S}^{-1},L,\ldots, L\rangle_{X/S}.
\end{equation}

The following is a direct application of our main theorem. 
\begin{theorem} \label{thm:KMourtheorem}
    Suppose that $X \to S$ is a projective, flat and finitely presented morphism of relative dimension $n$, which is also lci. Then, there are canonical isomorphisms of  $\QBbb$-line bundles
    \begin{equation}\label{eq:general-KM-RR}
        \mathcal M_{\ell} \simeq  \langle \chfrak((L-1))^\ell \cdot\tdfrak(T_{f}) \rangle_{X/S}.
    \end{equation}
\end{theorem}
\qed 

In particular, by making the left-hand side of \eqref{eq:general-KM-RR} explicit via Theorem \ref{thm:cor86}, we recover canonical isomorphisms of $\QBbb$-line bundles of the form \eqref{eq:dom-coeff-KM} and \eqref{eq:subdom-coeff-KM}. For \eqref{eq:subdom-coeff-KM}, our geometric assumptions are much more general than in \cite{PhongSturmRoss}. In a similar vein, one can exhibit expressions for lower order terms.  For instance, we obtain:
\begin{corollary}\label{cor:Knudsen-Mumford-subsubdominant}
Let the assumptions be as in Theorem \ref{thm:KMourtheorem}. Then, there exists a canonical isomorphism of $\QBbb$-line bundles
\begin{displaymath}
    \Mcal_{n-1}^{12}\simeq \Mcal_{n+1}^{-12}\otimes\Mcal_{n}^{6}\otimes\langle K_{X/S}, K_{X/S},L,\ldots, L\rangle_{X/S}
    \otimes \langle\cfrak_{2}(T_{f})\cdot\cfrak_{1}(L)^{n-1}\rangle_{X/S}.
\end{displaymath}
\end{corollary}
\qed

The Knudsen--Mumford expansion \eqref{eq:KnudsenMumfordiso} can be rearranged into a polynomial expansion in powers of $k$, with $\QBbb$-line bundle coefficients, in the form
\begin{displaymath}
    \lambda_f(L^k) \simeq \bigotimes_{\ell=0}^{n+1} \Ncal_{\ell}^{k^{\ell}}.
\end{displaymath}
Whenever $f: X\to S$ is moreover assumed to be lci, we hence obtain a canonical functorial isomorphism,
\begin{equation}\label{eq:Ncal_ell}
    \Ncal_{\ell}=\langle\cfrak_{1}(L)^{\ell}\cdot\tdfrak(T_{f})\rangle_{X/S}^{1/\ell!}.
\end{equation}
This allows us to exhibit explicit expressions for the higher CM line bundles in \cite{Vedova-Zuddas}, denoted by $\lambda_{CM,\ell}(X/S)$, which in turn generalize the classical CM line bundle of \cite{Paul-Tian, PhongSturm1}, and play a key role in the study of asymptotic Chow stability. Concretely, we have:
\begin{proposition}
Let $f\colon X\to S$ be as in Theorem \ref{thm:KMourtheorem}. Then, there are canonical isomorphisms of $\QBbb$-line bundles
\begin{equation}\label{eq:higher-CM}
    \lambda_{CM,\ell}(X/S) \simeq \Ncal_{n+1-\ell}^{\frac{1}{a_{n}}}\otimes \Ncal_{n+1}^{-\frac{a_{n-\ell}}{a_{n}^{2}}},
\end{equation}
where the $\Ncal_{\ell}$ are defined as in \eqref{eq:Ncal_ell} and the exponents are given by the coefficients of the Hilbert polynomial of $L$, written in the form $a_{n}k^{n}+a_{n-1}k^{n-1}+\cdots+a_{0}$, provided that $a_{n}\neq 0$ (e.g. if $L$ is relatively nef and big).
\end{proposition}
\qed

\subsubsection{Ducrot's interpretation of Deligne pairings}\label{subsubsec:Ducrot}
In \cite{Ducrot}, Ducrot showed that the Deligne pairing of line bundles, as developed in the work of Elkik \cite{Elkikfib}, admits a description via the determinant of the cohomology, which generalizes the corresponding property for families of curves already obtained by Deligne in \cite{Deligne-determinant}. Ducrot's results, while initially limited to locally Noetherian base schemes, extend to morphisms satisfying the condition $(C_{n})$, after the work of Boucksom and Eriksson \cite[Appendix A]{Boucksom-Eriksson}.

Let $f\colon X\to S$ be a morphism satisfying the condition $(C_{n})$. Ducrot's work provides a canonical isomorphism
\begin{equation}\label{eq:Ducrot-iso}
    \lambda_{f}((L_{0}-\Ocal_{X})\otimes\cdots \otimes (L_{n}-\Ocal_{X}))\simeq \langle L_{0},\ldots,L_{n}\rangle_{X/S}.
\end{equation}
It is compatible with base change and with the basic features of Deligne pairings recalled in \cite[Proposition 6.1]{DRR1}, notably the multilinearity and the restriction properties. A delicate point in Ducrot's work is the control of signs in the commutativity constraints. Below, since we work in the category of $\QBbb$-line bundles, these sign issues are irrelevant and can be ignored.

On the other hand, if $f$ is lci and $S$ is quasi-compact, Theorem \ref{thm:A} can be applied to $(L_{0}-\Ocal_{X})\otimes\cdots\otimes (L_{n}-\Ocal_{X})$, thought of as an object of the virtual category. By the multiplicativity of $\chfrak$ with respect to the tensor product \eqref{eq:chern-tensor-product} and the canonical isomorphism $[\chfrak(\Ocal_{X})]_{X/S}\simeq 1$ from \eqref{eq:chern-trivial} and \cite[Proposition 9.1]{DRR1}, we find a canonical isomorphism
\begin{equation}\label{eq:DRR-deligne-pairing}
    \begin{split}
           \langle\chfrak((L_{0}-\Ocal_{X})\otimes &\cdots \otimes(L_{n}-\Ocal_{X}))\cdot\tdfrak(T_{f})\rangle_{X/S}\\
            \simeq &\langle (\cfrak_{1}(L_{0})\cdots\cfrak_{1}(L_{n})+\mathrm{h.o.})\cdot (1+\frac{1}{2}\cfrak_{1}(T_{f})+\mathrm{h.o.})\rangle_{X/S}\\
            &=\langle L_{0},\ldots, L_{n}\rangle_{X/S},
    \end{split}
\end{equation}
where h.o. is a shorthand for higher order terms. We thus obtain an isomorphism of $\QBbb$-line bundles of the shape \eqref{eq:Ducrot-iso}. Our aim is to show that they coincide.

\begin{proposition}\label{prop:Ducrot-DRR}
Let $f\colon X\to S$ be an lci morphism of quasi-compact schemes, satisfying the condition $(C_{n})$. Ducrot's isomorphism \eqref{eq:Ducrot-iso} coincides with the analogous isomorphism deduced from Theorem \ref{thm:A} as above, as isomorphisms of $\QBbb$-line bundles. 
\end{proposition}

Before the proof, we briefly review and reformulate Ducrot's approach to \eqref{eq:Ducrot-iso}, based on the notion of $(n+2)$-cube structure. He introduces a category of $k$-cubes of line bundles and a subcategory of decorated $k$-cubes, cf. \cite[Sections 1.3--1.5]{Ducrot}. To a $k$-cube $A$, one can associate an object $[A]$ in the virtual category. This is done in such a way that for a $(k+1)$-cube of the form $A-B$, where $A$ and $B$ are $k$-cubes, we have $[A-B]=[B]-[A]$. For the associated determinant of the cohomology, we will simply write $\lambda_{f}(A)=\lambda_{f}([A])$. In the virtual category, any decorated $k$-cube is naturally isomorphic to an object of the form $L\otimes (L_{1}-\Ocal_{X})\otimes\cdots\otimes (L_{k}-\Ocal_{X})$, cf. \cite[Example 1.5.1 (d)]{Ducrot}. An $(n+2)$-cube structure on the determinant of the cohomology consists of giving a trivialization $s_{K}$ of $\lambda_{f}(K)$, for any $(n+2)$-cube $K$, satisfying several properties:
\begin{enumerate}
    \item[(C0)] \emph{Compatibility with isomorphisms of decorated cubes.} These are induced by isomorphisms of line bundles, and induce isomorphisms of the associated objects in the virtual category and of their determinant of the cohomology.
    \item[(C1)] \emph{Gluing.} Consider two decorated $(n+2)$-cubes of the form $A-B$ and $C-D$, where $A,B,C,D$ are $(n+1)$-cubes. Suppose we are given an isomorphism $u\colon B\to C$. Then, $A- D$ inherits a decoration, and we thus have trivializations $s_{A-B}$, $s_{C-D}$ and $s_{A-D}$ of the corresponding determinant bundles. We require that $s_{A-B}\otimes s_{C-D}$ corresponds to $s_{A-D}$ through the isomorphism $\lambda_{f}(A-B)\otimes\lambda_{f}(C-D)\to\lambda_{f}(A-D)$ induced by $u$.
    \item[(C2)] \emph{Normalization.} If $A$ is an $(n+1)$-cube, then $A-A$ is canonically decorated, and the trivialization $s_{A-A}$ corresponds to the canonical trivialization through the canonical isomorphism $\lambda_{f}(A-A)\to\Ocal_{S}$.
\end{enumerate}
Ducrot's main result \cite[Theorem 4.2]{Ducrot} is the existence of a unique $(n+2)$-cube structure which is compatible with base change and such that, for every decorated $(n+1)$-cube $K$ and any relative effective Cartier divisor $D$ in $X$ (or any base change thereof), the trivialization $s_{K-K(D)}$ corresponds to the trivialization $s_{K(D)|_{D}}$ via the restriction isomorphism
\begin{displaymath}
    \lambda_{f}(K(D))\otimes\lambda_{f}(K)^{-1}\simeq \lambda_{f|_{D}}(K(D)|_{D}).
\end{displaymath}
These two properties characterize the construction of the cube structure. This result is the key to the comparison isomorphism \eqref{eq:Ducrot-iso}, cf. \cite[Section 5]{Ducrot}. It is formal to adapt Ducrot's work and consider a cube structure on the determinant of the cohomology as a $\QBbb$-line bundle, which is a weaker notion. 

\begin{proof}[Proof of Proposition \ref{prop:Ducrot-DRR}]
We need to construct, via the Deligne--Riemann--Roch isomorphism, an $(n+2)$-cube structure on the determinant of the cohomology, as a $\QBbb$-line bundle. We may assume that all the schemes are divisorial, and work with virtual categories of vector bundles. 

By the functoriality of the DRR isomorphism in the bundle, for the purpose of constructing the trivialization defining the cube structure, we can reduce to an object of the form 
\begin{displaymath}
    E=L\otimes (L_{0}-\Ocal_{X})\otimes\cdots\otimes (L_{n+1}-\Ocal_{X})
\end{displaymath}in the virtual category. A similar argument as in \eqref{eq:DRR-deligne-pairing} shows that
\begin{equation}\label{eq:RR-distribution-trivial-for-degree-reasons}
    \langle\chfrak(E)\cdot\tdfrak(T_{f})\rangle_{X/S} \simeq \langle (\mathrm{h.o.})\cdot\tdfrak(T_{f})\rangle_{X/S}=\Ocal_{S}.
\end{equation}
Therefore, by the DRR isomorphism, we obtain a trivialization $t_{E}$ of $\lambda_{f}(E)$ as a $\QBbb$-line bundle. Note that this is well defined after taking an appropriate power of $\lambda_{f}(E)$, which only depends on $f$ and not on $E$. We first check that $t_{E}$ satisfies the properties corresponding to $\mathrm{(C0)}$--$\mathrm{(C2)}$ above.

Property $\mathrm{(C0)}$ is a consequence of the construction of $t_{E}$ and the functoriality of the DRR isomorphism in the bundle. Property $\mathrm{(C1)}$ is obtained as a combination of the multiplicativity of the DRR isomorphism and the functoriality in the bundle. Finally, for $\mathrm{(C2)}$, in the virtual category, an object of the form $E=F-F$ is isomorphic to $0$, and both sides of the DRR are trivial. These trivializations correspond to each other via the DRR isomorphism, since it is an isomorphism of functors of commutative Picard categories, which includes a compatibility with the implicit trivializations of the functors evaluated on the 0 object. We next need to show that the construction of $t_{E}$ is compatible with base change and has the restriction property. 

The compatibility of $t_{E}$ with base change follows from the base-change compatibility of the DRR isomorphism. For the restriction property, let $K$ be a decorated $(n+1)$-cube, and denote by $F$ the associated object in the virtual category. As reviewed above, we may assume that 
\begin{displaymath} F = L\otimes (L_{0}-\Ocal_{X})\otimes \cdots \otimes (L_{n}-\Ocal_{X}).
\end{displaymath} 
Suppose that $D$ is a relative effective Cartier divisor in $X$, and $i\colon D\to X$ the associated closed immersion. Then, $K-K(D)$ is a decorated $(n+2)$-cube, which corresponds to $E=F(D)-F$ in the virtual category. Note that there is a canonical isomorphism $F(D)-F\simeq i_{!}i^{\ast}F(D)$. We write down the following diagram, with horizontal arrows given by the DRR isomorphism:
\begin{equation}\label{eq:big-diag-Ducrot}
    \xymatrix{
            \lambda_{f|_{D}}(i^{\ast}F(D)) \ar[r]\ar[d]                &\langle\chfrak(i^{\ast}F(D))\cdot\tdfrak(T_{f|_{D}})\rangle_{D/S}\ar[d]\\
            \lambda_{f}(i_{!}i^{\ast}F(D)) \ar[r]\ar[d]                  &\langle\chfrak(i_{!}i^{\ast}F(D))\cdot\tdfrak(T_{f})\rangle_{X/S}\ar[d]\\
            \lambda_{f}(F(D)-F) \ar[r]    &\langle\chfrak(F(D)-F)\cdot\tdfrak(T_{f})\rangle_{X/S}.
    }
\end{equation}
The top square commutes, by the compatibility of the DRR isomorphism with the composition of morphisms, applied here to the composition $D\to X\to S$, cf. Definition \ref{def:compositionofmorphisms} and Theorem \ref{thm:general-DRR}. The bottom square commutes by the functoriality of the DRR isomorphism in the bundle. The up-right and down-right corners are trivial, and we need to show that the trivializations correspond to each other via the composition of the right vertical arrows. For this, we will unravel the construction of this composition. 

The composition of the right vertical arrows in \eqref{eq:big-diag-Ducrot} is obtained as in Definition \ref{def:compositionofmorphisms}, and involves the DRR isomorphism for the closed immersion $D\to X$. It is given by Construction/Definition \ref{construction-definition-RRsec} applied to $F(D)$. The latter is based on the projection formula and the isomorphisms (cf. \textsection\ref{subsubsec:RRimmersionD}) 

\begin{equation}\label{eq:DRRi-Ducrot-1}
    [\chfrak(\Ocal_{X}-\Ocal(-D))]_{X/S}\overset{\substack{\text{Borel--Serre}\\ \vspace{0.1cm}}}{\simeq} [\cfrak_{1}(\Ocal(D))\cdot\tdfrak(\Ocal(D))^{-1}]_{X/S}\overset{\substack{\text{restriction}\\ \vspace{0.1cm}}}{\simeq} [\tdfrak(\Ocal(D)|_{D})^{-1}]_{D/S}.
\end{equation}
 To conclude, it suffices to show that if we multiply \eqref{eq:DRRi-Ducrot-1} by $\chfrak(F(D))$, then all the terms become trivial by the same procedure as in \eqref{eq:RR-distribution-trivial-for-degree-reasons}, and that the trivializations thus produced correspond to each other via the isomorphism. 
 
 After applying the multiplicativity of $\chfrak$ to $F(D)$, the isomorphism $[\chfrak(\Ocal_{X})]_{X/S}\simeq 1$ and the rank triviality, we see that the middle term in \eqref{eq:DRRi-Ducrot-1} is isomorphic to some $[P]_{X/S}$, where $P$ has degree $\geq n+2$, and hence it is trivial. By the same token, the rightmost term in \eqref{eq:DRRi-Ducrot-1} multiplied with $\chfrak(F(D))$ is isomorphic to some $[Q]_{D/S}$, where $Q$ has degree $\geq n+1$, and hence it is trivial too. Finally, the operations performed in obtaining these trivializations commute with the restriction isomorphism, from which we can infer that the trivializations correspond to each other. This completes the proof. 
\end{proof}

\subsubsection{The DRR isomorphism for finite morphisms}
As an application of Theorem \ref{thm:Groth-duality} and Proposition \ref{prop:Ducrot-DRR}, we can describe the DRR isomorphism for the determinant of the cohomology in relative dimension $0$. 

Let $\pi\colon T\to S$ be a finite, flat, lci morphism of Noetherian schemes. Since the DRR isomorphism is compatible with the projection formula, and any vector bundle on $T$ is a pullback by $\pi$ locally over $S$, it suffices to describe the DRR isomorphism for $\Ocal_{T}$. In this case, using $[\cfrak_{1}(\Ocal_{T})]_{T/S}\simeq 0$ and $[\cfrak_{1}(T_{f})]_{T/S}\simeq -[\cfrak_{1}(\omega_{T/S})]_{T/S}$, we find
\begin{equation}\label{eq:DRR-finite-1}
    \RR_{\pi}(\Ocal_{T})^{2}\colon (\det\pi_{\ast}\Ocal_{T})^{2}\simeq \langle \omega_{T/S}\rangle_{T/S}^{-1}=\Delta_{T/S},
\end{equation}
where we define $\Delta_{T/S}=N_{T/S}(\omega_{T/S})^{-1}$. We would like to understand \eqref{eq:DRR-finite-1}.

On the one hand, by Proposition \ref{prop:Ducrot-DRR}, the DRR isomorphism for $\omega_{T/S}-\Ocal_{T}$ is induced by Ducrot's isomorphism
\begin{equation}\label{eq:DRR-finite-2}
    \det\pi_{\ast}(\omega_{T/S}-\Ocal_{T})\simeq\langle\omega_{T/S}\rangle_{T/S}=\Delta^{-1}_{T/S}.
\end{equation}
On the other hand, the DRR isomorphism for $\omega_{T/S}+\Ocal_{T}$ yields a trivialization
\begin{equation}\label{eq:DRR-finite-3}
    \det\pi_{\ast}(\omega_{T/S}+\Ocal_{T})\simeq \Ocal_{S},
\end{equation}
which by Theorem \ref{thm:Groth-duality} is given by Grothendieck duality. Since the DRR isomorphism is additive in the bundle, we see that $\eqref{eq:DRR-finite-3}\otimes\eqref{eq:DRR-finite-2}^{-1}$ amounts to \eqref{eq:DRR-finite-1}. In other words, given Ducrot's interpretation of Deligne pairings, the DRR isomorphism $\RR_{\pi}(\Ocal_{T})$ is equivalent to and given by Grothendieck duality for $\det\pi_{\ast}\Ocal_{T}$. 

The isomorphism above can be described explicitly, at least locally, by the theory of differents and discriminants. For the latter, we follow \cite[\href{https://stacks.math.columbia.edu/tag/0DWH}{0DWH}]{stacks-project}, which in particular discusses the equivalence with Grothendieck duality in the Gorenstein (e.g. lci) case. Suppose that $\pi$ is as above, with $S = \Spec A$ and $T=\Spec B$, where $B$ is a free $A$-algebra of rank $d$. Consider the pairing given by 
\begin{displaymath}
    \begin{split}
        Q\colon B\times B &\longmapsto A\\
        (e,f) &\longmapsto\Tr_{B/A}(ef).
    \end{split}
\end{displaymath}
If $\pi$ is \'etale, this pairing is non-degenerate and induces an isomorphism
\begin{equation}\label{eq:DRRfinite-morphisms}
        \begin{split}
            (\operatorname{det}_{A}B)^{\otimes 2}    &\longrightarrow A\\
            (e_1 \wedge \ldots \wedge  e_d)\otimes (f_1 \wedge \ldots \wedge f_d) &\longmapsto \det (\Tr_{B/A} (e_i f_j))_{i,j}.
        \end{split}
\end{equation}
See \cite[\href{https://stacks.math.columbia.edu/tag/0BVH}{0BVH}]{stacks-project}. More generally, by \emph{loc. cit.}, the image of the pairing is the discriminant ideal. If $\pi$ is generically \'etale, i.e. \'etale at the associated points of $T$, then by \cite[\href{https://stacks.math.columbia.edu/tag/0BTC}{0BTC}, \href{https://stacks.math.columbia.edu/tag/0DWM}{0DWM}, \href{https://stacks.math.columbia.edu/tag/0C14}{0C14}]{stacks-project}, the discriminant ideal identifies with $\Delta_{T/S}$ as defined above. Altogether, \eqref{eq:DRRfinite-morphisms} induces a natural isomorphism
\begin{displaymath}
    (\det\pi_{\ast}\Ocal_{T})^{\otimes 2}\to\Delta_{T/S},
\end{displaymath}
which coincides with \eqref{eq:DRR-finite-1}.

\subsection{The determinant of the de Rham cohomology}\label{subsec:det-dR-coh}
Deligne’s work  has influenced the study of determinants of de Rham cohomology for local systems and related objects. Some developments include Patel’s formalism of de Rham $\varepsilon$-factors \cite{Patel} and the results of Saito and Terasoma \cite{Saito-Terasoma} on determinants of period integrals. These motivate a specialization of Theorem \ref{thm:A} in this setting.

Let $X$ be a smooth projective variety over a field $k$ of characteristic $0$, of pure dimension $n$, and $D$ a simple normal crossings divisor in $X$. Let $E$ be a coherent sheaf on $X$, endowed with a flat logarithmic connection $\nabla\colon E \to E\otimes\Omega_{X/k}(\log D)$. The restriction of $E$ to $U = X \setminus D$ is necessarily locally free, and we assume it has constant rank $r$. Denote by $\Ecal^{\bullet}$ the logarithmic de Rham complex with $\Ecal^p = E \otimes \Omega^p_{X/k}(\log D)$ and differential induced by $\nabla$. We have a canonical isomorphism:
\begin{displaymath}
    \det H^{\bullet}(X,\Ecal^{\bullet})\simeq \bigotimes_{p}\det H^{\bullet}(X,E\otimes\Omega_{X/k}^{p}(\log D))^{(-1)^{p}}.
\end{displaymath}
Applying Theorem \ref{thm:A}, together with the Borel--Serre-type isomorphism for intersection distributions from \cite[Theorem 9.5]{DRR1} and the first Chern class isomorphism from Theorem \ref{thm:cor86}, we obtain:
\begin{proposition}\label{prop:C}
With the assumptions and notation as above, there exists a canonical isomorphism of one-dimensional $k$-vector spaces
\begin{equation}\label{eq:GRR-dR}
    \det H^{\bullet}(X,\Ecal^{\bullet})\otimes\det H^{\bullet}(X,\Omega_{X/k}^{\bullet}(\log D))^{-r}\simeq\langle\cfrak_{1}(\det E)\cdot\cfrak_{n}(\Omega_{X/k}(\log D))\rangle_{X/k}^{(-1)^{n}},
\end{equation}
well defined up to a power. The isomorphism is compatible with base field extensions and is multiplicative on exact sequences. 
\end{proposition}\qed

The right-hand side of \eqref{eq:GRR-dR} can be seen as a global de Rham $\varepsilon$-factor, and the goal of the theory of de Rham $\varepsilon$-factors consists in factoring it in terms of local contributions. For instance, if $n=1$ and $\theta$ is a meromorphic differential form, whose divisor as a section of $\Omega_{X/k}(\log D)$ is written as $\sum_{i}n_{i}P_{i}$, then the restriction property of Deligne pairings yields
\begin{displaymath}
   \langle\cfrak_{1}(\det E)\cdot\cfrak_{1}(\Omega_{X/k}(\log D))\rangle_{X/k}\simeq \bigotimes_{i}(\det E_{P_{i}})^{n_{i}}.
\end{displaymath}
The analogous remark applies for general $n$ if $\theta$ is any holomorphic section of $\Omega_{X/k}(\log D)$, with at most isolated zeros.

In the case $k=\CBbb$, we may further relate the algebraic logarithmic de Rham cohomology to the cohomology of a local system. We follow \cite[Chapter II]{connexionsregulieres} and \cite[Section 1]{Saito-Terasoma}. To lighten notation, by an abuse of notation we identify $X$, $D$, $U=X\setminus D$, etc. with the associated complex analytic counterparts. Suppose that $\mathbb{V}$ is a local system on $U$, of rank $r$, and introduce the corresponding flat vector bundle $\Vcal=\mathbb{V}\otimes_{\underline{\CBbb}_U}\Ocal_{U}$, with flat connection $\nabla$. Then, there exists a natural extension of $\Vcal$ to a vector bundle $E$ on $X$, in such a way that $\nabla$ extends to a flat logarithmic connection on $E$ and the associated de Rham complex $\Ecal^{\bullet}$ satisfies 
\begin{equation}\label{eq:comparison-dR-Deligne-canonical}
    R\Gamma(X,\Ecal^{\bullet})\simeq R\Gamma(U,\mathbb{V})
\end{equation}
via a natural morphism $\Ecal^{\bullet}\to Rj_{\ast}\mathbb{V}$ in the derived category, where $j\colon U\to X$ is the open immersion. The extension $E$ is singled out by imposing that the residue eigenvalues of $\nabla$ have real part in the interval $[0,1)$. It is usually called the Deligne canonical extension of $\Vcal$.

Given the comparison between analytic and algebraic de Rham cohomology, we can combine the isomorphism \eqref{eq:comparison-dR-Deligne-canonical} with Proposition \ref{prop:C} and obtain a canonical isomorphism
\begin{equation}\label{eq:GRR-dR-top}
    \begin{split}   
         &\det R\Gamma(U,\mathbb{V})\otimes \det R\Gamma(U,\underline{\CBbb}_{U})^{-r}\simeq \langle\cfrak_{1}(\det E)\cdot\cfrak_{n}(\Omega_{X/\CBbb}(\log D))\rangle_{X/\CBbb}^{(-1)^{n}},
    \end{split}
\end{equation} 
well defined up to a power, which is multiplicative on exact sequences. An analogous statement holds for the compactly supported cohomology of $\mathbb{V}$. The compatibility of the DRR isomorphism with duality (cf. Theorem \ref{thm:Groth-duality}) implies that \eqref{eq:GRR-dR-top} and its compactly supported counterpart are similarly compatible via Poincar\'e duality. This proves Proposition \ref{prop:det-dR-intro} from the Introduction, and answers the question of Saito and Terasoma in \cite[Remark, Section 6 (b), p. 925]{Saito-Terasoma}.

\subsection{The BCOV isomorphism for Calabi--Yau families}\label{subsubsec:BCOV-isomorphism}
In this subsection, we address the BCOV isomorphism for Calabi--Yau families announced in the Introduction, cf. \textsection\ref{subsubsec:BCOV-iso}. 

By a Calabi--Yau variety over a field $k$, we mean a smooth geometrically connected projective variety $Y$ over $k$, such that for some integer $m\geq 1$, the $m$-pluricanonical bundle $K_{Y}^{m}$ is trivial. The following lemma deals with the Calabi--Yau condition in families.

\begin{lemma}\label{lemma:evaluationmap}
Let $f\colon X\to S$ be a smooth projective morphism of locally Noetherian schemes, with geometrically connected fibers.
\begin{enumerate}
    \item Suppose that, for some integer $m\geq 1$, the evaluation map
    \begin{equation}\label{eq:evaluationmKXS}
        f^{\ast}f_{\ast}K_{X/S}^{m}\to K_{X/S}^{m}
    \end{equation}
    is an isomorphism. Then:
    \begin{enumerate}
        \item\label{item:evaluationmap-1a} The direct image $f_{\ast}K_{X/S}^{m}$ is an invertible sheaf over $S$, whose formation commutes with base change.
        \item\label{item:evaluationmap-1b} The formation of \eqref{eq:evaluationmKXS} commutes with base change.
    \end{enumerate}
    \item\label{item:evaluationmap-2} The evaluation map \eqref{eq:evaluationmKXS} is an isomorphism in either of the following cases:
    \begin{enumerate}
        \item\label{item:evaluationmap-2a} The scheme $S$ is reduced and $K_{X/S}^{m}$ is trivial on fibers. 
        \item\label{item:evaluationmap-2b} The relative canonical bundle $K_{X/S}^{m}$ is trivial on fibers and $H^{1}(X_{s},\Ocal_{X_{s}})=0$ for every $s\in S$. 
    \end{enumerate}
\end{enumerate}
Moreover, the analogues of the above assertions hold in the complex analytic category.
\end{lemma}
\begin{proof}
Assume first that \eqref{eq:evaluationmKXS} is an isomorphism. By restricting on fibers, we see that for every $s\in S$, the vector space $(f_{\ast}K_{X/S}^{m})\otimes k(s)$ is necessarily one-dimensional and $K_{X_{s}}^{m}\simeq\Ocal_{X_{s}}$. Because the fibers are proper and geometrically connected and \eqref{eq:evaluationmKXS} is an isomorphism, the natural map $(f_{\ast}K_{X/S}^{m})\otimes k(s)\to H^{0}(X_{s},K_{X_{s}}^{m})$ is an isomorphism. Then \eqref{item:evaluationmap-1a} and \eqref{item:evaluationmap-1b} follow from the base change theorem \cite[Theorem 25.1.6]{Vakil}.

Similarly, the second point is a standard application of the base change theorems \cite[Theorem 25.1.5 \& Theorem 25.1.6]{Vakil}, and we omit the details. See also Proposition 25.1.11 in \emph{op. cit.}, which addresses \eqref{item:evaluationmap-2a}.

In the complex analytic category, the argument is analogous, replacing \cite[Theorem 25.1.5]{Vakil} by \cite[Theorem 4.12]{Banica} and \cite[Theorem 25.1.6]{Vakil} by \cite[Theorem 4.10]{Banica}.

\end{proof}

\begin{definition}[Calabi--Yau family]\label{def:CY-family}
By a Calabi--Yau family we mean a smooth projective morphism of locally Noetherian schemes $f\colon X \to S$ such that, for some integer $m \geq 1$, the evaluation map \eqref{eq:evaluationmKXS} is an isomorphism. The same definition applies in the complex analytic category.
\end{definition}

Lemma \ref{lemma:evaluationmap} provides conditions which guarantee the Calabi--Yau family condition from the more familiar Calabi--Yau property on fibers.

We now recall the definition of the BCOV bundle.
\begin{definition}[BCOV bundle]\label{def:BCOVbundle}
Let $f: X \to S$ be a smooth locally projective family with geometrically connected fibers of dimension $n$, with $S$ a quasi-compact scheme. The associated BCOV bundle is defined as the combination of determinant bundles
\begin{displaymath}
    \lambda_{BCOV}(X/S)=\bigotimes_{p=0}^{n}\lambda_f(\Omega_{X/S}^p)^{ (-1)^p p }.
\end{displaymath}
\end{definition}
Note that, because the determinant of the cohomology and the sheaves of differentials commute with base change, the formation of the BCOV bundle commutes with base change too. 

The Deligne--Riemann--Roch isomorphism provides the following expression for the BCOV bundle, which does not require the Calabi--Yau condition, and covers Proposition \ref{prop:pre-bcov-iso-intro} in the introduction.

\begin{proposition}\label{prop:BCOViso} Let $f\colon X\to S$ be as in Definition \ref{def:BCOVbundle}. Then, there is a canonical isomorphism of $\QBbb$-line bundles
    \begin{displaymath}
        \lambda_{BCOV}(X/S)^{12} \simeq \left\langle \cfrak_{1}(\Omega_{X/S}) \cdot \cfrak_n(\Omega_{X/S}) \right\rangle^{(-1)^{n}}_{X/S},
    \end{displaymath}
compatible with base change.
\end{proposition}

\begin{proof}
    The BCOV bundle is the determinant of the cohomology of the virtual bundle 
    \begin{displaymath}
        \sum (-1)^p p [\Omega_{X/S}^p] 
    \end{displaymath}
    and we apply the Riemann--Roch isomorphism to this. An application of the splitting principle for $\Omega_{X/S}$ along the lines of \cite[Lemma 4.6]{cdg} and the proof of the Borel--Serre isomorphism \cite[Theorem 9.5]{DRR1} yields the isomorphism.
\end{proof}

We now proceed to specialize the previous lemma in the Calabi--Yau setting. In preparation for the statement, we recall that for a smooth projective variety $Y$ of dimension $n$ over a field, one can define the topological Euler characteristic by means of intersection theory as $\chi(Y)=\int_{Y} c_{n}(T_{Y})$. This actually agrees with the $\ell$-adic or Betti topological Euler characteristic, by the Lefschetz fixed point formula, and is locally constant in flat families.

\begin{theorem}[BCOV isomorphism]\label{thm:bcoviso}
    Let $f\colon X\to S$ be a Calabi--Yau family over a locally Noetherian scheme $S$, as in Definition \ref{def:CY-family}. Then, there is a canonical isomorphism of $\QBbb$-line bundles 
    \begin{displaymath}
        \lambda_{BCOV}(X/S)^{12m} \simeq  \left(f_\ast K_{X/S}^{m}\right)^{\chi}
    \end{displaymath}
    compatible with base change.
\end{theorem}
\begin{proof}
Because the evaluation map $f^{\ast}f_{\ast}K_{X/S}^{m}\to K_{X/S}^{m}$ is an isomorphism, we have 
    \begin{displaymath}
        [\cfrak_1(\Omega_{X/S})] \simeq [\cfrak_1(K_{X/S})] \simeq \frac{1}{m}[\cfrak_1(K_{X/S}^m)]\simeq \frac{1}{m}[f^{\ast}\cfrak_{1}(f_{\ast}K_{X/S}^{m})].
    \end{displaymath}
    This isomorphism commutes with base change, by Lemma \ref{lemma:evaluationmap}. By an application of Proposition \ref{prop:BCOViso} and the projection formula, we find an isomorphism
    \begin{displaymath}
         \lambda_{BCOV}(X/S) \simeq \left(f_\ast K_{X/S}^m\right)^{(-1)^n \frac{1}{12m}\int_{X/S} c_n(\Omega_{X/S})},
    \end{displaymath}
    compatible with base change. We conclude since $\int_{X/S} c_n(\Omega_{X/S}) = (-1)^n \chi$. 
\end{proof}
We note that the bounded denominator property of the DRR isomorphism guarantees that the BCOV isomorphism has a denominator depending only on the dimension of a projective embedding of $X\to S$, and it is invariant under base change.

\begin{corollary}
 Let $f\colon X\to S$ be a Calabi--Yau family over a locally Noetherian scheme $S$, whose fibers have $\chi=0$. Then, for some integer $N\geq 1$, the $N$-th power $\lambda_{BCOV}(X/S)^{N}$ is canonically trivial. 
\end{corollary}
\qed

We conclude with the counterpart of the BCOV isomorphism in the complex analytic category.
\begin{corollary}\label{cor:BCOV-analytic}
Let $f\colon X\to S$ be a Calabi--Yau family of complex analytic spaces, satisfying one of the conditions in Lemma \ref{lemma:evaluationmap} \eqref{item:evaluationmap-2}. Then, the analogue of Theorem \ref{thm:bcoviso} holds for $f$.
\end{corollary}
\begin{proof}
    We begin with a preliminary observation. Let $g\colon\Xcal\to\Hcal$ be a flat projective morphism of locally Noetherian schemes. Then:
    \begin{enumerate}
        \item\label{item:BCOV-iso-anal-obs-1} If $\Hcal$ is locally of finite type over $\CBbb$ and reduced, the locus of points $t\in\Hcal$ where $\Xcal_{t}$ is smooth, geometrically connected and $K_{\Xcal_{t}}^{m}$ is trivial, is open.
        \item\label{item:BCOV-iso-anal-obs-2} The locus of points $t\in\Hcal$ where $\Xcal_{t}$ is smooth, geometrically connected,  $K_{\Xcal_{t}}^{m}$ is trivial and $H^{1}(\Xcal_{t},\Ocal_{\Xcal_{t}})=0$, is open. 
    \end{enumerate}
    The locus where $g$ is smooth and has geometrically connected fibers is open by \cite[\href{https://stacks.math.columbia.edu/tag/01V9}{01V9}, \href{https://stacks.math.columbia.edu/tag/0E0N}{0E0N}]{stacks-project}. Then, the first point follows from the invariance of the plurigenera \cite[Corollary 0.3]{Siu:plurigenera-CY} and Grauert's base change theorem \cite[Theorem 25.1.5]{Vakil}, and the second point follows from the base change theorem \cite[Theorem 25.1.6]{Vakil}.
    
    Now, for the proof of the corollary. Without loss of generality, we may assume that $S$ is connected. For a fixed projective embedding, we can consider the associated Hilbert scheme $\Hcal$ and its universal family $g\colon\Xcal\to\Hcal$. If moreover $S$ is reduced, we endow $\Hcal$ with the reduced scheme structure. We have a classifying map $\varphi\colon S\to\Hcal^{\an}$, where $\Hcal^{\an}$ is the analytification of the Hilbert scheme; cf. \cite[Proposition 5.3]{Simpson:moduli-1} for the analytic properties of the latter. By \eqref{item:BCOV-iso-anal-obs-1} and \eqref{item:BCOV-iso-anal-obs-2} above, there is an open subscheme $\Ucal\subset\Hcal$ where the corresponding conditions are satisfied and such that $\varphi(S)\subseteq \Ucal^{\an}$. We next apply Theorem \ref{thm:bcoviso} to the universal family $\Xcal$ restricted to $\Ucal$, and then take the analytification of the resulting BCOV isomorphism. We finally pull back this isomorphism by $\varphi$. This is the sought analytic BCOV isomorphism, provided we show it is independent of the projective embedding. 

    To verify the independence of the projective embedding, we proceed along the lines of \cite[Proposition 3.26]{Eriksson-Freixas-Wentworth}, which deals with the analytification of Deligne pairings. By a completion argument and the compatibility with arbitrary base change in Theorem \ref{thm:bcoviso}, it is enough to check the independence claim after base change to a finite analytic subspace $T$ of $S$, in which case the projective morphism $X_{T}\to T$ is algebraizable and corresponds to an algebraic classifying map $T\to\Ucal$. The claim then follows, since the algebraic BCOV isomorphism commutes with base change and does not depend on the projective embedding. We note that it is crucial that the isomorphism has a denominator which depends only on the relative dimension of the projective embedding and is invariant under base change.  
\end{proof}

\providecommand{\bysame}{\leavevmode\hbox to3em{\hrulefill}\thinspace}
\providecommand{\MR}{\relax\ifhmode\unskip\space\fi MR }
\providecommand{\MRhref}[2]{%
  \href{http://www.ams.org/mathscinet-getitem?mr=#1}{#2}
}

\end{document}